\documentclass[11pt]{amsart}
\usepackage{color,array,amssymb,amsmath, amsthm, amsfonts, tikz-cd}

\usepackage{graphicx}
\usepackage{listings}
\usepackage[margin=1in]{geometry}
\usepackage{lstautogobble}
\usepackage{enumerate}
\usepackage{thmtools}
\usepackage{thm-restate}
\usepackage{amsthm}
\usepackage{verbatim}

\usepackage{mathtools}
\usepackage{physics}
\usepackage{color}
\usepackage{hyperref}
\usepackage[capitalise]{cleveref}
\usepackage[bottom]{footmisc}
\crefformat{equation}{(#2#1#3)}

\usepackage{float}
\restylefloat{table}

\usepackage{mathrsfs}

\theoremstyle{plain}
\newtheorem{thm}{Theorem}[section]
\newtheorem{prop}[thm]{Proposition}
\newtheorem{lemma}[thm]{Lemma}
\newtheorem{cor}[thm]{Corollary}

\theoremstyle{definition}
\newtheorem{mydef}[thm]{Definition}

\newtheorem{remark}[thm]{Remark}

\Crefname{prop}{Proposition}{Propositions}

\numberwithin{equation}{section} 

\DeclarePairedDelimiter{\paren}{\lparen}{\rparen}

\DeclareMathOperator{\supp}{supp}

\DeclareMathOperator{\sgn}{sgn}

\newcommand{\M}{{\mathcal{M}}}

\newcommand{\Fc}{\mathcal{F}}
\newcommand{\zg}{|z|^\gamma}

\newcommand{\p}{{\partial}}

\renewcommand{\d}{\delta}

\newcommand{\R}{{\mathbb{R}}}

\newcommand{\N}{{\mathbb{N}}}

\newcommand{\Z}{{\mathbb{Z}}}

\newcommand{\T}{{\mathbb{T}}}
\newcommand{\g}{{\mathsf{g}}}
\newcommand{\G}{{\mathsf{G}}}

\renewcommand{\M}{{\mathbb{M}}}

\renewcommand{\k}{\mathsf{k}}
\newcommand{\I}{\mathbb{I}}

\newcommand{\tl}{\tilde}

\newcommand{\D}{\Delta}

\newcommand{\ph}{\phantom{=}}
\newcommand{\nn}{\nonumber}

\newcommand{\ol}{\overline}
\newcommand{\ul}{\underline}

\newcommand{\ux}{\underline{x}}

\newcommand{\vep}{\epsilon}
\newcommand{\ep}{\varepsilon}
\newcommand{\al}{\alpha}
\newcommand{\be}{\beta}
\newcommand{\ka}{\kappa}
\newcommand{\la}{\lambda}

\newcommand{\K}{\mathsf{K}}

\newcommand{\W}{{\mathbf{W}}}

\newcommand{\Tc}{\mathcal{T}}

\newcommand{\wh}{\widehat}

\newcommand{\Dm}{|\nabla|}
\newcommand{\fs}{\mathsf{f}}
\newcommand{\Fs}{\mathsf{F}}
\newcommand{\rs}{\mathsf{r}}
\newcommand{\Cs}{\mathsf{C}}
\newcommand{\Fr}{{F}}
\newcommand{\Ec}{\mathcal{E}}
\renewcommand{\P}{\mathcal{P}}
\newcommand{\Te}{\mathrm{Term}}

\newcommand{\Cb}{\mathbf{C}}
\newcommand{\bmu}{\bar\mu}
\newcommand{\gep}{\g_{(\ep)}}
\newcommand{\Vext}{V_{\text{ext}}}

\newcommand{\ga}{\gamma}

\newcommand{\Dc}{\mathcal{D}}
\newcommand{\Ent}{\mathrm{Ent}}
\newcommand{\Eng}{\mathrm{Eng}}

\let\div\relax
\DeclareMathOperator{\div}{div}

\def\XXint#1#2#3{{\setbox0=\hbox{$#1{#2#3}{\int}$ }
\vcenter{\hbox{$#2#3$ }}\kern-.6\wd0}}

\def \hal{\frac{1}{2}}
\def\({\left(}
\def\){\right)}
\def \ep{\varepsilon}
\def\nab{\nabla}
\def\indic{\mathbf{1}}
\def\cd{\mathsf{c_{d,s}}}

\title[Sharp uniform-in-time mean-field convergence for singular periodic Riesz flows]{Sharp uniform-in-time mean-field convergence for singular periodic Riesz flows}

\author[A. Chodron de Courcel]{Antonin Chodron de Courcel}
\email{antonin.chodron-de-courcel@polytechnique.org}
\author[M. Rosenzweig]{Matthew Rosenzweig}
\email{mrosenzw@mit.edu}
\thanks{M. R. is supported by the Simons Foundation through the Simons Collaboration on Wave Turbulence and by NSF grants DMS-2052651, DMS-2206085.}
\author[S. Serfaty]{Sylvia Serfaty}
\email{serfaty@cims.nyu.edu}
\thanks{S. S. is supported by the Simons Foundation through a Simons Investigator award and by NSF grant DMS-2000205.}

\begin{document}
\begin{abstract}
We consider conservative and gradient flows for $N$-particle Riesz energies with mean-field scaling on the torus $\T^d$, for $d\geq 1$, and with thermal noise of McKean-Vlasov type. We prove global well-posedness and relaxation to equilibrium rates for the limiting PDE. Combining these relaxation rates with the modulated free energy of Bresch et al. \cite{BJW2019crm,BJW2019edp,BJW2020} and recent sharp functional inequalities of the last two named authors for variations of Riesz modulated energies along a transport, we prove  uniform-in-time mean-field convergence in the gradient case with a rate which is sharp for the modulated energy pseudo-distance. For gradient dynamics, this completes in the periodic case the range $d-2\leq s<d$ not addressed by previous work \cite{RS2021} of the second two authors. We also combine our relaxation estimates with the relative entropy approach of Jabin and Wang \cite{JW2018} for so-called $\dot{W}^{-1,\infty}$ kernels, giving a proof of uniform-in-time propagation of chaos alternative to Guillin et al. \cite{GlBM2021}.
\end{abstract}
\maketitle

\section{Introduction}\label{sec:intro}

\subsection{The problem}\label{ssec:introprob}
We are interested in proving mean-field convergence, i.e. the large $N$ limiting behavior of dynamics for stochastic singular interacting particle systems of the form
\begin{equation}\label{eq:SDE}
\begin{cases}
dx_{i}^t = \displaystyle\frac{1}{N}\sum_{1\leq j\leq N : j\neq i} \M\nabla\g(x_i^t-x_j^t)dt + \sqrt{2\sigma}dW_i^t\\
x_i^t|_{t=0} = x_i^0,
\end{cases}\qquad i\in\{1,\ldots,N\}.
\end{equation}
Above, $x_i^0\in\T^d$, the flat torus in dimension $d\geq 1$ which we identify with $[-\frac{1}{2},\frac{1}{2}]^d$ under periodic boundary conditions, are the pairwise distinct initial positions; $\{W_i\}_{i=1}^N$ are independent standard Brownian motions in $\T^d$, so that the noise in \eqref{eq:SDE} is of so-called additive type; and the coefficient $\sigma$, which may be interpreted as temperature, is nonnegative. $\M$ is a $d\times d$ matrix with constant real entries. We shall either choose $\M =-\I$, corresponding to \emph{gradient/dissipative} dynamics or choose $\M$ to be antisymmetric, corresponding to \emph{Hamiltonian/conservative} dynamics. Mixed flows are also allowable, but here our main results will concern gradient flows. The motivation for considering $\T^d$, as opposed to Euclidean space $\R^d$, will be explained below.

The interaction potential $\g$ that we will study is a periodic Riesz potential (indexed by a parameter $-1 \le s <d $), that is the zero average solution to
\begin{equation}\label{eq:gmod}
|\nabla|^{d-s}\g = \cd(\d_0-1), \qquad \cd\coloneqq \begin{cases} \frac{4^{\frac{d-s}{2}}\Gamma((d-s)/2)\pi^{d/2}}{\Gamma(s/2)}, & {-1\leq s<d} \\ \frac{\Gamma(d/2)(4\pi)^{d/2}}{2}, & {s=0}. \end{cases}
\end{equation}
The notation $\Dm^{d-s}$ denotes the Fourier multiplier with symbol $(2\pi |\xi|)^{d-s}$. As explained in \cref{sec:PRP}, the potential $\g$ behaves like $|x|^{-s}$, if $-1\leq s<d$, or $-\log|x|$, if $s=0$, near the origin. This is a model choice for studying systems, such as \eqref{eq:SDE}, with interactions between particles that become singular as the inter-particle distance tends to zero. The family of potentials defined by \eqref{eq:gmod} includes the physically important {\it Coulomb} case $s=d-2$, as well as the {\it sub-Coulomb} range $s<d-2$ and {\it super-Coulomb} range $d-2<s<d$. We are primarily interested in the super-Coulomb case, as we shall explain momentarily. Unlike in the setting of $\R^d$, where $\g(x)=|x|^{-s}$ for $s\neq 0$ and $\g(x)=-\log|x|$ for $s=0$, the potential $\g$ on $\T^d$ does not have a simple form (see \cite{HSS2014, HSSS2017} for various representations of periodic Riesz potentials). However, one can show that in a neighborhood of the origin, $\g$ equals its Euclidean analogue plus a smooth correction (see \eqref{eq:ggE} below). We limit ourselves to the \emph{potential} case $s<d$, in which $\g\in L^1(\T^d)$. The \emph{hypersingular} case $s\geq d$ is also interesting, but is a fundamentally different regime and will not be considered in this article. We refer to \cite[Chapter 8]{BHS2019}, \cite{HLSS2018, HSST2020} and references therein for more on this case.

Applications of systems of the form \eqref{eq:SDE} are numerous. Since this topic has been discussed at length elsewhere, we will not repeat this discussion. Instead, we refer the reader to the introduction of \cite{RS2021}, the survey \cite{JW2017_survey}, and the recent lecture notes \cite{CD2021, Golse2022ln}. 

One can show by a truncation and stopping time argument that there is a unique, local strong solution to the system \eqref{eq:SDE}. When $s\leq d-2$, one can then use the energy of the system (which has nonincreasing expectation) to show that the solution is global. In particular, with probability one, the particles never collide. This has been shown in \cite[Section 4]{RS2021} for the case $s<d-2$ in Euclidean space, but the argument is adaptable to the periodic setting without issue and, with a little more work, to the Coulomb case $s=d-2$ as well. When $d-2<s<d$, the aforementioned global existence argument fails, in short because $\D\g \rightarrow +\infty$ as $|x|\rightarrow\infty$, as opposed to $-\infty$ when $s<d-2$. Consequently, it is unclear how to make sense of the system of SDEs \eqref{eq:SDE}, except on very short timescales which \emph{a priori} vanish as $N\rightarrow\infty$.\footnote{Elias Hess-Childs has recently informed us that it is possible to show well-posedness of the SDEs in the gradient flow case for $\max(0,d-2)<s<d$.} Accordingly, rather than work with \eqref{eq:SDE} directly, we work with the \emph{Liouville/forward Kolmogorov equation}
\begin{equation}\label{eq:liou}
\begin{cases}
\p_t f_N = \displaystyle-\sum_{i=1}^N \div_{x_i}\paren*{f_N\frac{1}{N}\sum_{1\leq j\leq N: j\neq i}\M\nabla\g(x_i-x_j)} + \sigma\sum_{i=1}^N\D_{x_i}f_N \\
f_N|_{t=0} = f_N^0,
\end{cases}
\end{equation}
which is obtainable from \eqref{eq:SDE} through It\^o's formula. Here, the initial positions of the particles are thought of as random vectors in $\T^d$ distributed according to a probability density $f_N^0$, and $f_N^t$ is the law of the solution $\ux_N^t\coloneqq (x_1^t,\ldots,x_N^t)$ to \eqref{eq:SDE}.

Establishing the  mean-field limit refers to showing the weak convergence (in the sense of probability measures) as $N \to \infty$ of the {\it empirical measure} 
\begin{equation}\label{munt}
\mu_N^t\coloneqq \frac1N \sum_{i=1}^N \delta_{x_i^t}
\end{equation}
associated to a solution $\ux_N^t \coloneqq (x_1^t, \dots, x_N^t)$ of the system \eqref{eq:SDE}. For fixed $t$, we note that the empirical measure is a random Borel probability measure on $\T^d$. If the points $x_i^0$, which themselves depend on $N$, are such that $\mu_N^0$   converges  to some sufficiently regular measure $\mu^0$, then a formal application of It\^o's lemma leads to the expectation that for $t>0$, $\mu_N^t$ converges as $N\rightarrow\infty$ to the solution of the Cauchy problem
\begin{equation}
\label{eq:lim}
\begin{cases}
\p_t\mu = -\div(\mu\M\nabla\g\ast\mu) +\sigma\D \mu\\
\mu|_{t=0} = \mu^0,
\end{cases}
\qquad (t,x)\in \R_+\times\T^d.
\end{equation}
While the underlying $N$-body dynamics are stochastic, we stress that the equation \eqref{eq:lim} is completely deterministic, and the noise has been averaged out to become diffusion, which is consistent with the independence of the Brownian motions and the mean-field limit being a law of large numbers type result. Proving the convergence of the empirical measure is closely related to proving {\it propagation of molecular chaos}: if $f_N^0(x_1, \dots, x_N)$ is the initial law of the distribution of the $N$ particles in $\R^d$ and if $f_N^0$ converges to some factorized law  $(\mu^0)^{\otimes N}$, then  the $k$-point marginals $f_{N;k}^t$ converge for all time to $(\mu^t )^{\otimes k}$. It is known that mean-field convergence and propagation of chaos are qualitatively equivalent (e.g., see \cite{HM2014}), though quantitative results for one form of convergence do not \emph{a priori} carry over to the other.

\medskip
The topic of mean-field limits for singular interactions has seen tremendous progress in recent years. In particular, we mention the works of Jabin and Wang \cite{JW2018} which allowed to treat so-called $\dot{W}^{-1, \infty}$ interactions via a relative entropy method; the introduction of the {\it modulated energy} by Duerinckx in \cite{Duerinckx2016} to noiseless systems of the form \eqref{eq:SDE}, following earlier usage in a different context by the third author in \cite{Serfaty2017}; the generalization of the modulated energy method to all super-Coulombic interactions in \cite{Serfaty2020},  which allowed to treat cases without noise, both conservative, dissipative, or mixed; and the {\it modulated free energy method} of Bresch \textit{et al.} \cite{BJW2019crm, BJW2019edp, BJW2020}, which combines both the relative entropy and modulated energy approaches in a physical way to treat the case (of gradient flows only) with noise. Subsequent work by the last two authors with Nguyen \cite{NRS2021} generalized the modulated energy method to sub-Coulombic interactions (cf. \cite{Hauray2009, CCH2014}), and to singular interactions that are not exactly of Riesz type (e.g., repulsive Lenard-Jones potentials in the case of gradient flows). Further extensions and improvements of the modulated energy method concerning regularity of solutions to \eqref{eq:lim} \cite{Rosenzweig2022, Rosenzweig2022a} and incorporation of multiplicative noise \cite{Rosenzweig2020spv} have been achieved by the second author. Much progress has also been made by Lacker \cite{Lacker2021}, who introduced a novel usage of the relative entropy in conjunction with the BBGKY hierarchy to obtain the sharp $O((k^2/N^2))$ rate for the asymptotic factorization of the $k$-point marginals measured by the relative entropy, but only for less singular cases, such as bounded interactions. The recent work of Bresch et al. \cite{BJS2022} also introduces a novel usage of the BBGKY hierarchy to prove uniform-in-$N$ weighted $L^p$ estimates for the $k$-point marginals, which allows to treat second-order systems with degenerate noise and singular interactions (e.g., Coulomb in dimension 2) of Vlasov-Fokker-Planck type,\footnote{The work of Lacker \cite{Lacker2021} (see Remarks 2.11 and 4.5 in that work) is also capable of proving (sharp) propagation of chaos for second-order kinetic models with degenerate noise, but again only for less singular interactions.} as well as first-order systems with interactions more singular than in \cite{JW2018}. We emphasize that these last two works strongly rely on the dissipative effect of the noise, i.e. they require $\sigma>0$. 
 
The aforementioned work \cite{BJW2019crm, BJW2019edp, BJW2020} of Bresch \textit{et al.} focuses on treating as general as possible \emph{repulsive} singular interactions with a mildly \emph{attractive} part (e.g., logarithmic). In particular, the latter work proves the mean-field limit for the Patlak-Keller-Segel (PKS) equation on $\T^2$, which corresponds to \eqref{eq:gmod} with $s=0$ and $\g$ replaced by $-\g$, up to, but not including, the critical temperature.\footnote{In the literature on PKS dynamics, the critical parameter for the global existence vs. finite-time blowup is typically formulated in terms of a critical mass with fixed unit temperature (i.e., diffusion coefficient). Since the mass in our setting is normalized to one, this critical mass can be equivalently expressed as a critical temperature.} However, when considering only repulsive Riesz interactions of Coulomb or super-Coulomb type, their modulated free energy method, which leverages algebraic cancellations specific to the gradient flow structure, can lead to a much quicker proof of convergence, as outlined in the introduction of \cite{Serfaty2020}. In addition, their work, which was restricted to the torus, left as an assumption the existence of a sufficiently regular limiting solution. The essential content of this restriction to the torus is the need for compactness of the underlying domain in order to prove lower bounds for the density that are uniform in space, which are needed in order to show certain norms involving $\log\mu^t$ are finite. Such pointwise bounds are seemingly incompatible with the setting of $\R^d$, without some form of confinement in the form of an external potential,\footnote{In unpublished work by the last two authors with Huang \cite{HRS2022}, we show how to extend the modulated free energy to the case of $\R^d$ when a confining term $-\nabla V_{ext}$ is added to the dynamics to the dynamics in \eqref{eq:SDE} and provided one starts from initial data which are small perturbations of the equilibrium for equation \eqref{eq:lim} with the additional confining term.} as $\mu^t$ vanishes as $|x|\rightarrow\infty$ and the $L^\infty$ norm of $\mu^t$ vanishes as $t\rightarrow\infty$.

\medskip
Our goal in this paper is to present a streamlined version of mean-field convergence for periodic Riesz interactions \eqref{eq:gmod} in the case $d-2\leq s<d$, along with a complete analysis of the limiting equation \eqref{eq:lim}. Specifically, we prove \eqref{eq:lim} is globally well-posed (either in the dissipative or conservative case), and solutions and their derivatives satisfy exponentially fast relaxation estimates (see \Cref{sec:WP,sec:Rlx} below). By combining these relaxation estimates with the modulated free energy method  and new sharp functional inequalities for the variations of Coulomb/Riesz modulated energies obtained by the last two named authors \cite{RS2022}, we manage to show the first instance of {\it uniform-in-time} convergence for singular dissipative flows, with a rate which is sharp in $N$.

There have been a number of results obtained by probabilistic arguments over the years on uniform-in-time mean-field convergence/propagation of chaos for McKean-Vlasov type systems. We mention the sample of works \cite{Malrieu2003, CGM2008, Salem2018, AdM2020, DEGZ2020, DT2021}, which are related to the long-time dynamics of nonlinear Fokker-Planck equations, e.g. \cite{BRTV1998i, BRTV1998ii, BCCP1998, Malrieu2003, CMV2003, CMV2006, CGM2008, BGM2010, BGG2013}. Importantly, these results are restricted to regular potentials. We also note that uniform-in-time propagation of chaos may fail for certain potentials \cite{BGP2010, BGP2014}. It is only very recently that uniform-in-time results for the much more difficult case of singular potentials have been obtained.
 
The uniform-in-time convergence (without sharp rate) was previously shown by the last two authors in \cite{RS2021} only for $d\geq 3$ and $0\le s < d-2$ on Euclidean space, though the proof is adaptable to the torus. The idea of our proof for the relaxation to equilibrium of the limiting solutions is inspired by the method of \cite{RS2021}, which itself builds on earlier ideas of Carlen and Loss \cite{CL1995}. We note that uniform-in-time convergence (without sharp rate) was established by Guillin et al. \cite{GlBM2021} for $\dot{W}^{-1,\infty}$ kernels through a refinement of the argument in \cite{JW2018}, in particular the exploitation of the Fisher information through a uniform-in-time log Sobolev inequality. We also mention that recent work of Lacker and Le Flem \cite{LlF2022}  builds on \cite{Lacker2021} to obtain uniform-in-time propagation of chaos with a sharp rate. The results of \cite{LlF2022} have been subsequently extended to slightly more singular interactions---though not covering the case $s=0$ of \eqref{eq:gmod}---in \cite{Han2022}, subject to a number of conditions. Finally, we mention recent work of Guillin et al. \cite{GlBM2022}, which proves uniform-in-time propagation of chaos for one-dimensional log and Riesz gases, exploiting the convexity of the log/Riesz interaction in dimension one---and only dimension one (cf. \cite{BO2019}).  

Lest the reader think otherwise, uniform-in-time convergence is not merely aesthetically pleasing. It is important for both theory and practice, such as when using a particle system to approximate the limiting equation or its equilibrium states and for quantifying stochastic gradient methods, such as those used in machine learning for general interaction kernels.
  
\subsection{The modulated free energy method}\label{ssec:introMFE}
In order to present our results, let us introduce the modulated free energy from \cite{BJW2019crm, BJW2019edp, BJW2020}, which is a combination of two quantities: the modulated energy from \cite{Duerinckx2016, Serfaty2020} and the relative entropy from \cite{JW2016, JW2018}. See also earlier incarnations---in other contexts---of modulated energy/relative entropy methods in \cite{Dafermos1979, Yau1991,Brenier2000}.

The modulated energy is a Coulomb/Riesz-based ``metric'' that can be understood as a renormalization of the negative-order homogeneous Sobolev norm corresponding to the energy space of the equation \eqref{eq:lim}. More precisely, it  is defined to be
\begin{equation}\label{def:modulatedenergy}
\Fr_N(\ux_N,\mu) \coloneqq \frac12\int_{(\T^d)^2\setminus\triangle} \g(x-y)d\paren*{\frac{1}{N}\sum_{i=1}^N\d_{x_i} - \mu}^{\otimes 2}(x,y),
\end{equation}
where we remove the infinite self-interaction of each particle by excising the diagonal $\triangle \coloneqq \{(x,x) \in (\T^d)^2\}$.
Since we work in the statistical setting of the Liouville equation \eqref{eq:liou}, one needs to average this quantity with respect to the joint law $f_N$ of the positions $\ux_N$. We then define the (normalized) relative entropy with respect to the $N$-fold tensor product of a probability density $\mu$ (denoted $\mu^{\otimes N}$, which is the distribution of $N$ iid random points in $\T^d$ with law $\mu$) as 
\begin{equation}\label{eq:REdef}
H_N \left(f_N \vert \mu^{\otimes N}\right) \coloneqq \frac{1}{N} \int_{(\T^d)^N} \log\paren*{\frac{f_N}{\mu^{\otimes N}}} df_N.
\end{equation}
With the modulated energy and relative entropy, we now define the \textit{modulated free energy} following \cite{BJW2019crm, BJW2019edp, BJW2020}:
\begin{equation}\label{def:modulatedfreenrj}
E_N(f_N, \mu) \coloneqq \sigma H_N\left(f_N \vert \mu^{\otimes N}\right) + \int_{(\mathbb{T}^d)^N} F_N(\ux_N, \mu) df_N(\ux_N).
\end{equation}
As explained above, the consideration of $\T^d$, as opposed to $\R^d$, stems from the need for a confined domain.

When there is no noise (i.e., $\sigma=0$), the relative entropy is unnecessary and a pure modulated energy approach suffices, which is consistent with the weighting of the relative entropy by $\sigma$ in \eqref{def:modulatedfreenrj}. This has been shown in \cite{Duerinckx2016, Serfaty2020, NRS2021}, the last of which treats the full range $0\leq s<d$ for \eqref{eq:gmod}. Initially, it was unclear whether a pure modulated energy approach could also handle noise of the form in \eqref{eq:SDE}.\footnote{A pure modulated energy approach is known to be well-suited to a completely different kind of noise, of multiplicative type, thanks to \cite{Rosenzweig2020spv, NRS2021}.} Such an extension was finally shown in \cite{RS2021}, but only for $s<d-2$. This limitation stems from treating the nontrivial quadratic variation contribution to the evolution of the expectation of the modulated energy as a term which is nonpositive up to negligible error. Such nonpositivity is no longer expected to hold if $d-2\leq s<d$, and we are skeptical a pure modulated energy approach is feasible for $d-2\leq s<d$. Note that this work also makes sense of and works directly with the SDE \eqref{eq:SDE}, and not the Liouville equation \eqref{eq:liou}, so that the modulated energy is a stochastic process.

The modulated free energy method comes at the cost of only treating gradient flows case. This method consists of computing the evolution of the quantity $E_N(f_N^t,\mu^t)$, given solutions $f_N^t$ and $\mu^t$ of equations \eqref{eq:liou} and \eqref{eq:lim}, respectively, and establishing an inequality in caricature of the form
\begin{equation}
\frac{d}{dt}E_N(f_N^t,\mu^t) \leq C\paren*{E_N(f_N^t, \mu^t) + N^{-\be}},
\end{equation}
where $C$ is some constant depending on norms of $\mu^t$ and $\be>0$ is some exponent determined by $d,s$. One then concludes by the Gr\"onwall-Bellman lemma. Not only is it a physically well-motivated quantity, but mathematically, $E_N(f_N^t,\mu^t)$ is a good quantity for showing propagation of chaos because it metrizes both convergence of the $k$-point marginals (thanks to the relative entropy) and convergence of the empirical measure (thanks to the modulated energy). See \cref{rem:MFEcoer} below for further elaboration. The beautiful observation of \cite{BJW2019crm, BJW2019edp, BJW2020} is that when computing $\frac{d}{dt}E_N(f_N^t,\mu^t)$, the contribution of the noise to the evolution of the modulated energy cancels exactly with terms coming from the relative entropy, but only for the gradient flow case. Then, one is left with having to control the average with respect to the measure $df_N(\ux_N)$ of an expression of the form
\begin{equation}\label{derivmoden}
I\coloneqq \int_{(\T^d)^2\setminus \triangle} (v(x) -v(y)) \cdot \nab\g(x-y) d\paren*{\frac{1}N \sum_{i=1}^N \delta_{x_i}-\mu}^{\otimes 2} (x,y),\end{equation}
where $v$ is a vector field, by $E_N(f_N,\mu)$ and some negligible error. Note that the expression \eqref{derivmoden} arises naturally by pushing forward the measure $\frac{1}N \sum_{i=1}^N \delta_{x_i}-\mu$ under the map $\I+ t v$ in the modulated energy $\Fr_N(\ux_N,\mu)$  and computing the first derivative at $t=0$; in other words, computing the first variation of the modulated energy along the transport $v$.

Such a functional inequality was previously shown by the third author \cite[Proposition 1.1]{Serfaty2020} in the form\footnote{Strictly speaking, this cited work considers $\R^d$, not $\T^d$, but the argument is adaptable to the torus, as for instance shown in \cite[Proposition 3.9]{Rosenzweig2021ne}. See also \cref{ssec:MEme} below for explanation.}
\begin{equation}\label{eq:SerFI}
|I|\leq C\|\nab v\|_{L^\infty}\paren*{F_N(\ux_N, \mu) + \frac{(\log N)}{2dN}\indic_{s=0} + C(1+\|\mu\|_{L^\infty})N^{-\frac{d-s}{d(d+1)} } } + \text{(other terms)},
\end{equation}
for a constant $C$ depending only $d,s$. The $\text{(other terms)}$ are not so important for our discussion, and we choose not to make them explicit. These functional inequalities have proven to be extremely powerful for mean-field limit and related problems, and we mention a sample of recent applications \cite{GP2021, HkI2021, Rosenzweig2021ne, Rosenzweig2021qf, CC2021, Serfaty2021, Menard2022, Porat2022}. The original functional inequality \eqref{eq:SerFI} has since been improved in the Coulomb case $s=d-2$ in \cite[Corollary 4.3]{Serfaty2021}, \cite[Proposition 3.9]{Rosenzweig2021ne}, where the exponent $-\frac{d-s}{d(d+1)}$ is improved to $-1+\frac{s}{d}$. This is sharp in the sense that the modulated energy scales as $N\rightarrow\infty$ like $N^{-1+\frac{s}{d}}$: the minimal value of the modulated energy among all point configurations $\ux_N$, for a fixed background density $\mu$, scales like $N^{-1+\frac{s}{d}}$. See  \cite{HSSS2017} specifically for the periodic case and \cite{SS2015log, SS2015, RS2016, PS2017} for the Euclidean case with a confining potential. Recent work by the last two authors \cite{RS2022} goes further, in particular covering the full range $d-2\leq s<d$, and replaces the right-hand side in \eqref{eq:SerFI} by
\begin{equation}\label{eq:RSFI}
C\|\nabla v\|_{L^\infty}\paren*{\Fr_N(\ux_N, \mu) + \frac{\log(\|\mu\|_{L^\infty})}{2dN}\indic_{s=0}+ \|\mu\|_{L^\infty}^{\frac{s}{d}} N^{-1+\frac{s}{d}}}.
\end{equation}
This estimate is sharp in its rate $N^{-1+\frac{s}{d}}$. Moreover, it is improved in its dependence in $\|\mu\|_{L^\infty}$; and as first observed in \cite{RS2021}, this improved dependence can be used in conjunction with decay estimates for $\|\mu^t\|_{L^\infty}$ to obtain uniform-in-time bounds. The question of functional inequalities with sharp dependence on $N$ in the sub-Coulomb range $s<d-2$ remains an interesting open problem.

Returning to the modulated free energy method, one wants to apply the above described functional inequalities with $v=u^t\coloneqq \sigma\nabla\log\mu^t + \nabla \g \ast \mu^t$, where $\mu^t$ solves \eqref{eq:lim}. A point we stress to the reader is that when there is no noise, there is no $\nabla\log\mu^t$ term, and thus, the vector field is not the same in the modulated energy and modulated free energy methods. We have now arrived at a PDE question, which is control on the Lipschitz seminorm of $u^t$. Assuming that $\mu^t$ is bounded from below, this translates to $W^{2,\infty}$ control on $\mu^t$. Another point we stress to the reader is that control of $\|\nabla^{\otimes 2}\log\mu^t\|_{L^\infty}$ is delicate on the Euclidean space due to the  decay of $\mu^t $ to $0$ at infinity. This issue, of course, disappears on the torus (likely more generally a bounded domain with appropriate boundary conditions), since the solution to \eqref{eq:lim} remains bounded from below, provided the initial data is (see \cref{remark:positiveness} below). This is our main reason for considering the periodic setting. It is possible, however, to implement the modulated free energy method on $\R^d$ with a confining potential $\Vext$ added to the dynamics \eqref{eq:SDE}, \eqref{eq:lim} and for solutions of \eqref{eq:lim} which start near equilibrium (which is no longer uniform) \cite{HRS2022}.

We now come to the main concern of the present article. In light of the work \cite{RS2021} on uniform-in-time convergence for sub-Coulomb Riesz interactions, it is natural to ask if such a uniform-in-time result is also possible for the modulated free energy, which would then yield uniform-in-time convergence for the full Riesz range $-1\leq s<d$, at least in the periodic setting. Such a result also necessitates having a satisfactory solution theory for the limiting equation \eqref{eq:lim}, in particular global solutions in $W^{2,\infty}$. The well-posedness of \eqref{eq:lim}, even locally in time, is taken for granted in \cite{BJW2019edp}, which sketches the use of the modulated free energy for local-in-time convergence for general Riesz interactions. Additionally, one seeks estimates for the modulated free energy which are sharp in their dependence on $N$. Such estimates are obtained in the forthcoming work \cite{RS2022} for the Coulomb/super-Coulomb case without noise, but to our knowledge no work to date has achieved the sharp rate of convergence for Riesz interactions with noise.

\subsection{Main theorem}\label{ssec:introMR}
We are now ready to state our main results, which establish in complete generality the global existence and asymptotic behavior of solutions to \eqref{eq:lim} in both the conservative and dissipative cases for $d-2\leq s<d$ and show the first uniform-in-time propagation of chaos result for both Coulomb/super-Coulomb gradient flows on the torus. Moreover, the convergence is at the sharp rate $N^{-1+\frac{s}{d}}$. The function space notation in the statements of \cref{thm:mainWP} and \cref{thm:main} below is standard, but we recall it anyway for the reader's benefit in \cref{ssec:intronot}.

\begin{thm}\label{thm:mainWP}
Let $d\geq 1$, $d-2 \le s < d$ and $\sigma > 0$. Define the space $X \coloneqq L^\infty(\T^d) \cap \dot{W}^{\al,p}(\T^d)$, where
\begin{align}
\begin{cases} {\al\geq 0, \ 1\leq p\leq \infty}, & {\text{if $d-2\leq s\leq d-1$}} \\ {\al > \max(d-s+1,d-s+\frac{d}{p}), \ 1\leq p\leq \infty}, & {\text{if $d-1<s<d$}.} \end{cases}
\end{align}
Assume further that the initial datum $\mu^0\geq 0$ if $\M=-\I$. Then the equation \eqref{eq:lim} is globally well-posed in the space $C\left([0,\infty), X\right)$; smooth on $(0,\infty)\times\T^d$; $\inf_{\T^d} \mu^t \geq \inf_{\T^d} \mu^0$; and for any $n\geq 0$ and $1\leq p\leq \infty$, we have\footnote{The notation $\indic_{d-1<s<d}$ denotes the indicator function for the condition $d-1<s<d$.}
\begin{equation}\label{eq:mainrlxbnd}
\forall t>0, \qquad \|\nabla^{\otimes n}(\mu^t-1)\|_{L^p} \leq \W(\|\mu^0\|_{L^\infty}, \|\mu^0-1\|_{L^1}, \Fc_{\sigma}(\mu^0),\|\mu^0\|_{\dot{H}^{s+d-1}}\indic_{d-1<s<d}, \sigma^{-1})(\sigma t)^{-m}e^{-Ct},
\end{equation}
where $m>0$ depends on $n,s,d,p$, $C$ depends on $n,s,d,\M$, and $\W: [0,\infty)^5\rightarrow [0,\infty)$ is continuous, nondecreasing in its arguments, and depends on the parameters similar to $C$. Here, $\Fc_{\sigma}(\mu^0)$ is the free energy associated to equation \eqref{eq:lim} in the case $M=-\I$ (see \eqref{eq:FE} below), and $\W$ is independent of $\Fc_{\sigma}(\mu^0)$ if $\M$ is antisymmetric.
\end{thm}

\begin{thm}\label{thm:main}
Let $d\geq 1$, $d-2 \le s < d$ and $\sigma > 0$. Let $f_N$ be an entropy solution to \eqref{eq:liou}, in the sense of \cref{def:entsol}, and let $\mu^0 \in \mathcal{P}(\T^d)\cap  W^{2,\infty}(\T^d)$ with associated solution $\mu \in C([0,\infty), \P(\T^d)\cap W^{2,\infty}(\T^d))$. Assume further that $\inf_{\T^d} \mu^0 > 0$, which implies that $\inf_{\T^d}\mu^t>\inf_{\T^d}\mu^0$ for all $t\geq 0$. Suppose now that $\M=-\I$, so that we consider gradient flows. Define the quantity
\begin{equation}
\mathcal{E}_N^t \coloneqq E_N(f_N^t,\mu^t) + \frac{\log(N\|\mu^t\|_{L^\infty})}{2Nd}\indic_{s=0} + \Cs\|\mu^t\|_{L^\infty}^{\frac{s}{d}}N^{\frac{s}{d}-1},
\end{equation}
where $\Cs>0$ is a certain constant to ensure that $\mathcal{E}_N^t\geq 0$ (see \eqref{eq:MElbN} below). There exists a function $\mathcal{A}: [0,\infty)\rightarrow [0,\infty)$, depending on $s, d, \sigma, \inf_{\T^d} \mu^0, \| \mu^0\|_{W^{2,\infty}}, \|\mu^0\|_{\dot{H}^{1+s-d}}, \\ \|\mu^0-1\|_{L^1},\Fc_{\sigma}(\mu^0)$, such that $\mathcal{A}^0 = 1$, $\sup_{t\geq 0}\mathcal{A}^t<\infty$, and 
\begin{equation}\label{eq:mainEnbnd}
\forall t\geq 0,\qquad \Ec_N(f_N^t,\mu^t) \le \mathcal{A}^t\Ec_N(f_N^0,\mu^0).
\end{equation}
\end{thm}

We record the following remarks concerning \cref{thm:mainWP} and \cref{thm:main}.

\begin{remark}

The relaxation/decay estimate \eqref{eq:mainrlxbnd} is not the most general possible statement. One also has estimates which hold for fractional derivatives $\Dm^\al$, $\al>0$. We refer the reader to \cref{sec:Rlx} for further details. The fact that we have an exponential decay as $t\rightarrow\infty$, as opposed to an algebraic decay, as for instance in \cite{RS2021}, is a special feature of the confined setting of the torus vs. Euclidean space. The reader may easily convinces themselves of this by ignoring the nonlinearity and considering the asymptotic behavior of solutions $\mu^t$ to the linear heat equation, for which $\|\mu^t\|_{L^\infty}$ decays at the optimal rate $O(t^{-\frac{d}{2}})$ on $\R^d$. 
\end{remark}

\begin{remark}
Concerning regularity assumptions, Bresch et al. \cite{BJW2019edp} assume---but do not prove---the existence of a local solution $\mu \in C([0,T], W^{2,\infty}(\T^d))$, for some $T>0$, to equation \eqref{eq:lim}, which remains bounded from below on $[0,T]$. For such a solution and for $t\in [0,T]$, they can prove an estimate with the same structure as \eqref{eq:mainEnbnd}, but with nonsharp exponents.
\end{remark}

\begin{remark}
An explicit form of the right-hand side in the bound \eqref{eq:mainEnbnd} is given in \cref{ssec:MEconc} (see the inequality \eqref{eq:Antbnd}). We have not presented the explicit form above, so as to keep the introduction accessible.
\end{remark}

\begin{remark}
As is by now well-known, the vanishing of the relative entropy or modulated energy, and therefore the modulated free energy, as $N\rightarrow\infty$, implies propagation of chaos. We refer to \cref{rem:MFEcoer} below for further explanation.
\end{remark}

\begin{remark}
To the best of our knowledge, equation \eqref{eq:lim} has not been studied in the complete generality presented here. Some special cases are treated (on $\R^d$) in \cite{CW1999QG, Cordoba2004, GW2005, CFP2012, CCCGW2012, BIK2015, CJ2021pm}. However, the decay estimates seem generally to be new. If $d=2$ and $s=0$ and $\M$ is a $90^\circ$ rotation, then this is the well-known Navier-Stokes in vorticity form (e.g., see \cite{GW2005}). Staying in dimension two, but letting $0<s<2$, this the generalized SQG equation with subcritical dissipation (e.g., see \cite{Cordoba2004}). If $d=2$ and $s=0$, but now $\M=-\I$, then this is a repulsive analogue of the famous Patlak-Keller-Segel equation (e.g., see \cite{Blanchet2013}). Usually, these equations are studied on $\R^d$, but many of the results are expected to carry over to $\T^d$ \emph{mutatis mutandis}. We also mention that the case without temperature (and generally on $\R^d$, not $\T^d$) has been studied in several works, e.g. \cite{Yudovich1963, CMT1994, LZ2000, Cordoba2004, MZ2005, AS2008, CV2011, CV2011ab, CCCGW2012, CFP2012, CSV2013, SV2014, BIK2015, CHSV2015,  ZX2017, Duerinckx2018, LMS2018, CJ2021pm}.
\end{remark}

\begin{remark}
The results of \cref{thm:mainWP}, \cref{thm:main} are valid for any $\sigma>0$ (i.e., positive temperature), but essentially all our estimates blow up as $\sigma\rightarrow 0$. Naturally, one asks if it is possible to have a uniform-in-time mean-field convergence/propagation of chaos result when $\sigma=0$ (i.e., zero temperature). Interestingly, the answer is yes. For instance, for 2D Coulomb gradient dynamics, we can show that any $L^\infty$ solution of \eqref{eq:lim} which is bounded from below converges exponentially fast to the uniform distribution as $t\rightarrow\infty$, and this relaxation can be combined with the refinement of the modulated energy developed by the second author in \cite{Rosenzweig2022} to obtain uniform-in-time mean-field convergence. These findings and others will be reported elsewhere.
\end{remark}

The modulated free energy method was originally developed \cite{BJW2019crm, BJW2020} to treat propagation of chaos for the gradient dynamics of the $d$-dimensional \emph{attractive log gas}, which coincides with the aforementioned Patlak-Keller-Segel model if $d=2$. The cited works show a non-sharp rate for propagation of chaos, which deteriorates exponentially fast in time, leaving as a question whether a uniform-in-time rate is possible. In forthcoming work \cite{CdCRS2022}, we answer this question for sufficiently high temperatures using the modulated free energy and relaxation estimates for the limiting equation. The attractive case is substantially more difficult than the repulsive case considered here due to the existence of phase transitions: at a certain critical temperature, the long-time dynamics of the system completely change and one encounters issues of non-uniqueness and instability of stationary states. In fact, we show that a uniform-in-time estimate for the modulated free energy may fail if the temperature $\sigma$ is too low.

\subsection{Comments on the proof}\label{ssec:introproof}
Let us comment more on the proofs of \cref{thm:mainWP}, \cref{thm:main}. 

\medskip
Beginning with the \cref{thm:mainWP}, which allows for $\M$ to be either conservative or dissipative, we need to show global well-posedness of the equation \eqref{eq:lim} and that solutions relax in any Sobolev norm to their equilibrium $1$ (we assume that $\mu^0$ is a probability density, so the mass is conserved to be one; see \cref{remark:massconserv}) at an exponential rate. The local well-posedness (see \cref{prop:lwp} and more generally \cref{sec:WP}) proceeds through a fixed point argument for the mild formulation of \eqref{eq:lim}. This technique is classical, but some care is needed in the case $d-1<s<d$, as the vector field $\M\nabla\g\ast\mu$ loses derivatives compared to $\mu$. We mention that the local well-posedness step imposes no conditions on $\M$. It is not difficult to show that the norms $\|\mu^t\|_{L^q}$, for $1\leq q\leq\infty$, are nonincreasing as a function of $t$. In the case $d-2\leq s\leq d-1$, a uniform-in-time bound for $\|\mu^t\|_{L^q}$, in particular for $q=\infty$, implies that the solution must be global, since the time of existence in the local theory is determined by $\|\mu^0\|_{L^\infty}$. But for $d-1<s<d$, the time of existence also depends on a fractional Sobolev norm of $\mu^0$, for which we do not have an \emph{a priori} bound. However, through a variant of Gr\"onwall's lemma, one can show that fractional Sobolev norms cannot blow up in finite time, leading to global existence in all cases $d-2\leq s<d$ (see \cref{prop:gwp}). 

The relaxation estimates (see \Cref{ssec:WPfe,ssec:WPLpLq,sec:Rlx}) are considerably more involved. The first step is to show exponential decay of $\|\mu^t-1\|_{L^q}$, for $1\leq q\leq\infty$, as $t\rightarrow\infty$. We obtain such a rate by combining Poincar\'{e} inequalities and the dissipation of free energy (in the gradient case) together with hypercontractivity of the flow, which is shown by adapting an argument from \cite{RS2021}---which in turn was an extension of a work by Carlen and Loss \cite{CL1995}---to the periodic setting. The second step is to show exponential decay of $L^q$ norms of derivatives of $\mu^t$, which is important for subsequent application to uniform-in-time modulated free energy bounds. To obtain decay of derivatives (see \cref{prop:LpLq}), we exploit the mild formulation of the equation and the smoothing properties of the heat kernel together with a bootstrap argument to first show that $\|\nabla^{\otimes n}\mu^t\|_{L^q}$ can be controlled by $\|\mu^0-1\|_{L^q}(\sigma t)^{-n/2}$ for short times (i.e., $\sigma t$ small). By a time translation trick, we can then combine this short-time smoothing property with the previously established exponential decay of $\|\mu^t-1\|_{L^q}$ to obtain exponential decay of $\|\nabla^{\otimes n }\mu^t\|_{L^q}$. Related ideas have been used, for instance, for 2D Navier-Stokes (see \cite[Section 2.4]{GGS2010}) and perhaps in other contexts as well; but to our knowledge there has not been a treatment at the level of generality and for such singular vector fields as in our equation \eqref{eq:lim}. As one would expect, the case $d-1<s<d$ is the most delicate, due to the aforementioned loss of regularity in the velocity field, and a major portion of the technical effort in this paper consists of establishing these decay estimates for higher regularities.

\medskip
Transitioning to \cref{thm:main}, the idea is to combine  the use of functional inequalities in the modulated free energy method, as described in \cref{ssec:introMFE}, with the relaxation estimates (see \cref{ssec:MEconc}). Such a combination was first observed in \cite{RS2021} (for just the modulated energy) and in particular relies on the observation that one should choose the error terms in the coerciveness of the modulated free energy and in the functional inequalities for expressions \eqref{derivmoden} to depend on norms of $\mu^t$ (i.e., on the density). For instance, see \eqref{eq:EcNdef} and \eqref{eq:ENtineq}. Otherwise, one would be left with $N$-dependent error terms, which  while vanishing as $N\rightarrow\infty$, may grow as $t\rightarrow\infty$ (compare \eqref{eq:SerFI} vs. \eqref{eq:RSFI}). There is an additional ingredient concerning the sharpness of the functional inequalities. As advertised in the title of our work, the factor $N^{-1+\frac{s}{d}}$ is of the same order as the modulated energy as $N\rightarrow\infty$ and is therefore sharp. To obtain the exponent $-1+\frac{s}{d}$, we adapt to the periodic setting aforementioned recent work \cite{RS2022} by the last two authors on sharp estimates for variations of Coulomb/Riesz modulated energies along a transport. This result, presented in \cref{prop:FI}, is of independent interest.

\medskip
Using ideas inspired by the proof of uniform-in-time propagation chaos in \cref{thm:main}, we are also able to give a proof of uniform-in-time propagation of chaos for systems like \eqref{eq:SDE} but with $\M\nabla\g$ replaced by a kernel $\k$ which belongs to the space $\dot{W}^{-1,\infty}$ (i.e., it is the divergence of an $L^\infty$ matrix field). A precise statement of the result is given in \cref{sec:JW} with \cref{thm:JWunif}. This improves the result of Jabin-Wang \cite{JW2018}, which had a growing factor $e^{Ct}$ in their relative entropy estimate, and, indeed, our result should be understood as a refinement of their original proof, as the main novelty is the incorporation of decay estimates for $\mu^t$ to obtain a uniform-in-time result. As mentioned in \cref{ssec:introprob}, Guillin et al. previously obtained a uniform-in-time version of the Jabin-Wang result, also through a refinement of the original proof of \cite{JW2018}, but not relying on decay estimates. See the beginning of \cref{sec:JW} for comparison between the two proofs.

\subsection{Organization of paper}\label{ssec:introorg}
Let us briefly comment on the organization of the remaining body of the article.

In \cref{sec:PRP}, we review some basic facts about periodizations of Riesz potentials and estimates for the heat kernel. In \cref{sec:WP}, we show the global well-posedness of the limiting equation \eqref{eq:lim} (see \Cref{prop:lwp,prop:gwp}), the convergence to the uniform distribution as $t\rightarrow\infty$ (see \Cref{ssec:WPLp,ssec:WPfe}), and the $L^p$-$L^q$ smoothing property of solutions (see \cref{ssec:WPLpLq}). In \cref{sec:Rlx}, we prove estimates for the exponential decay rate of $L^p$ norms of derivatives of solutions of \eqref{eq:lim}, the main result being \cref{prop:LpLq}, which together with \cref{prop:gwp} completes the proof of \cref{thm:mainWP}. This section is divided into several subsections, each corresponding to a step in the proof of \cref{prop:LpLq}, as elaborated on at the beginning of \cref{sec:Rlx}. In \cref{sec:ME}, we combine our decay estimates with the modulated free energy to prove uniform-in-time propagation of chaos for the system \cref{eq:SDE}. This then completes the proof of \cref{thm:main}. \cref{sec:ME} also contains results of possible independent interest: new (sharp) functional inequalities for the modulated energy on the torus adapted from \cite{RS2022}, which treats the Euclidean setting. Finally, in \cref{sec:JW}, we show how our decay estimates approach may be combined with the relative entropy approach in a straightforward manner to give a proof of uniform-in-time propagation of chaos for systems with $\dot{W}^{-1,\infty}$ kernels.

\subsection{Acknowledgments}\label{ssec:introack}
This work was completed as a result of the first author's internship visit to the Courant Institute of Mathematical Sciences, NYU, which he thanks for its hospitality. The second author thanks Toan T. Nguyen for helpful discussion and feedback about the results of this paper.

\subsection{Notation}\label{ssec:intronot}
We close the introduction with the basic notation used throughout the article without further comment. We mostly follow the conventions of \cite{NRS2021,RS2021}.

Given nonnegative quantities $A$ and $B$, we write $A\lesssim B$ if there exists a constant $C>0$, independent of $A$ and $B$, such that $A\leq CB$. If $A \lesssim B$ and $B\lesssim A$, we write $A\sim B$. Throughout this paper, $C$ will be used to denote a generic constant which may change from line to line. Also, $(\cdot)_+$ denotes the positive part of a number.

$\N$ denotes the natural numbers excluding zero, and $\N_0$ including zero. Similarly, $\R_+$ denotes the positive reals. Given $N\in\N$ and points $x_{1,N},\ldots,x_{N,N}$ in some set $X$, $\ux_N=(x_{1,N},\ldots,x_{N,N})\in X^N$. Given $x\in\R^d$ and $r>0$, $B(x,r)$ and $\p B(x,r)$ respectively denote the ball and sphere centered at $x$ of radius $r$. Given a function $f$, we denote the support of $f$ by $\supp f$. The notation $\nabla^{\otimes k}f$ denotes the $k$-tensor field with components $(\p_{i_1\cdots i_k}^k f)_{1\leq i_1,\ldots,i_k\leq d}$.

$\P(\T^d)$ denotes the space of Borel probability measures on $\T^d$. If $\mu$ is absolutely continuous with respect to Lebesgue measure, we shall abuse notation by writing $\mu$ for both the measure and its density function. $C(\T^d)$ denotes the Banach space of  continuous, bounded functions on $\R^d$ equipped with the uniform norm $\|\cdot\|_{\infty}$. $C^k(\T^d)$ denotes the Banach space of $k$-times continuously differentiable functions with bounded derivatives up to order $k$ equipped with the natural norm, and $C^\infty \coloneqq \bigcap_{k=1}^\infty C^k$. The subspace of smooth functions with compact support is denoted with a subscript $c$.

$\Dm=(-\D)^{-\frac12}$ denotes the Fourier multiplier with symbol $2\pi|\xi|$. Functions of $\Dm$ can be defined through the spectral calculus (i.e., by using the Fourier transform). For integers $n\in \N_0$ and exponents $1\leq p\leq\infty$, $W^{n,p}$ denotes the usual Sobolev space. For general $\al\in \R$ and $1<p<\infty$, $W^{\al,p}$ denotes the Bessel potential space defined by
\begin{equation}
\left\{\mu \in \mathcal{D}'(\T^d) : \|(I-\D)^{\al/2}\mu\|_{L^p} < \infty\right\}.
\end{equation}
In other words, the space of distributions $\mu$ such that $(I-\D)^{\al/2}\mu$ is an $L^p$ function. When $\al$ is a positive integer, then $W^{\al,p}$ coincides with the classical Sobolev space above. For $p\in \{1,\infty\}$, these fractional Sobolev spaces are awkward to consider and will be generally avoided in this paper. When $p=2$, we instead use the customary notation $H^\al$. As is convention in the literature, a $\dot{}$ superscript indicates the corresponding homogeneous space.

\section{Periodic Riesz potentials and heat kernel estimates}\label{sec:PRP}
We recall from the introduction that $\g$ is the unique distributional solution to the equation
\begin{equation}
|\nabla|^{d-s}\g = \cd(\d_0-1), \qquad x\in \T^d,
\end{equation}
subject to the constraint that $\int_{\T^d}\g=0$. Equivalently, $\g$ is the distribution with Fourier coefficients $\hat{\g}(\xi) = \cd(2\pi|\xi|)^{s-d}\indic_{\xi\neq 0}$ for $\xi\in\Z^d$. One can show that $\g \in C^\infty(\T^d\setminus\{0\})$. Moreover, if we let
\begin{equation}
\g_E(x) \coloneqq -\log|x|\indic_{s=0} + |x|^{-s}\indic_{0<s<d}, \qquad \forall x\in\R^d
\end{equation}
denote the Euclidean log/Riesz potential, then
\begin{equation}\label{eq:ggE}
\g-\g_E\in C^\infty\paren*{B(0,\frac{1}{4})}.
\end{equation}
For proofs of these facts, we refer the reader to \cite{HSSS2017}. In particular, these facts imply that $\g\in L^{\frac{d}{s},\infty}(\T^d)$ (the weak $L^{\frac{d}{s}}$ space), \emph{a fortiori} in $L^1(\T^d)$, and
\begin{equation}
\forall n\geq 0 , \ x\in\T^d\setminus\{0\}, \qquad |\nabla^{\otimes n}\g(x)| \lesssim_n |x|^{-s-n} + 1.
\end{equation}

\medskip
We let $e^{t\D}$ denote the Fourier multiplier on $\T^d$ with coefficients $(e^{-4\pi^2 t|\xi|^2})_{\xi\in\Z^d}$, and we let $\K_t$ denote the convolution kernel of $e^{t\D}$. It is easy to check from the Fourier representation that $\K_t \in C^\infty(\T^d)$  and $\int_{\T^d}\K_t=1$, for every $t>0$. One can explicitly write $\K_t$ as the periodization of the Euclidean heat kernel,
\begin{equation}
\K_t(x) = (4\pi t)^{-d/2}\sum_{n\in\Z^d} e^{-\frac{|x-n|^2}{4 t}}.
\end{equation}
For instance, see \cite[Section 10.3]{BHS2019}. Since $\wh{\K_t}(\xi) = e^{-4\pi^2 t|\xi|^2}$, it follows that for any $m>\frac{d}{2}$,
\begin{equation}\label{eq:Kt1Linf}
\|\K_t-1\|_{L^\infty} \leq \sum_{\xi\in\Z^d : \xi\neq 0} e^{-4\pi^2 t|\xi|^2} \lesssim \sum_{\xi\in\Z^d : \xi\neq 0} (4\pi^2 t|\xi|^2)^{-m} \lesssim t^{-m}.
\end{equation}
The decay as $t\rightarrow\infty$ may, in fact, be improved to exponential by applying $\frac{d}{dt}$ to the second expression in \eqref{eq:Kt1Linf} and using Gr\"{o}nwall's lemma:
\begin{align}
\forall t\geq t_0, \qquad \|\K_t-1\|_{L^\infty} \leq  e^{-4\pi^2(t-t_0)}\sum_{\xi\in\Z^d : \xi\neq 0} e^{-4\pi^2t_0|\xi|^2}.
\end{align}
Additionally, by Riemann sum approximation, we have
\begin{equation}
\|\K_t-1\|_{L^\infty} \leq  \sum_{\xi\in\Z^d : \xi\neq 0} e^{-4\pi^2 t|\xi|^2} \approx t^{-\frac{d}{2}} \int_{\R^d} e^{-4\pi^2 |\xi|^2}d\xi,
\end{equation}
which shows that $\|\K_t\|_{L^\infty} = O(t^{-\frac{d}{2}})$ as $t\rightarrow 0$. Hence,
\begin{align}
\|\K_t-1\|_{L^\infty} \lesssim_d \min(t,1)^{-\frac{d}{2}}e^{-4\pi^2\max(t,1)}.
\end{align}
One can repeat the same analysis for derivatives and use interpolation (i.e., H\"older's inequality) to show 
\begin{equation}\label{eq:nabnK}
\forall n\in\N_0, \ t>0, \qquad \|\nabla^{\otimes n}\paren*{\K_t-1}\|_{L^p} \lesssim_{n,d,p} \min(t,1)^{-\frac{d}{2}\paren*{1-\frac{1}{p}} -\frac{n}{2}} e^{-C_{n,p} \max(t,1)}
\end{equation}
for any $1\leq p\leq \infty$. In fact, one can show that if $m(D)$ is a Fourier multiplier with symbol $m(\xi)$ which is homogeneous of degree $\ka$ and has symbol $m\in C^\infty(\R^d\setminus\{0\})$, then
\begin{equation}
\|m(D)\paren*{\K_t-1}\|_{L^p} \lesssim_{n,d,p,m(D)} \min(t,1)^{-\frac{d}{2}\paren*{1-\frac{1}{p}}-\frac{\ka}{2}}e^{-C_{p,m(D)}\max(t,1)}.
\end{equation}
From these properties, we deduce that if $\int_{\T^d}\mu=0$, then
\begin{equation}\label{eq:mDhk}
\forall t>0 \qquad \|m(D) e^{t\D}\mu\|_{L^p} \lesssim_{n,d,p,q,m(D)} \|\mu\|_{L^q}\min(t,1)^{-\frac{d}{2}\paren*{\frac{1}{q}-\frac{1}{p}}-\frac{\ka}{2}}e^{-C_{p,m(D)}\max(t,1)}
\end{equation}
for any $1\leq q\leq p\leq \infty$. We will use this property, sometimes referred to as \emph{hypercontractivity}, in the remaining body of the paper without further comment. 

\section{Well-posedness and $L^p$ control for the mean-field equation}\label{sec:WP}
In this section, we prove well-posedness for the limiting PDE \eqref{eq:lim} in all cases $d\geq 1$ and $d-2\leq s<d$. The case where $0\leq s<d-2$ and on $\R^d$ was previously treated in \cite{RS2021}, and the proof here is an adaptation of that argument; however the super-Coulombic case, in particular when $d-1<s<d$, is more complicated, as the reader will see, due to the loss of regularity in the velocity field $\M\nabla\g\ast\mu$. While to the best of our knowledge, equation \eqref{eq:lim} has not been previously considered in the literature in its full generality, some special cases of this well-posedness result are known in certain function spaces, as commented in the introduction.

\subsection{Local well-posedness and basic properties}\label{ssec:WPlwp}
We prove local well-posedness of \eqref{eq:lim} via a fixed point argument. To this end, we rewrite the equation in the mild form as
\begin{align}\label{eq:mild}
    \mu^t = e^{t\sigma\Delta}\mu^0 - \int_0^t e^{\sigma (t-\tau) \Delta} \div\paren*{\mu^\tau \M\nabla\g \ast \mu^\tau}d\tau.
\end{align}

The main result of this subsection is the following proposition. The regularity assumptions on $\mu^0$ are not optimal (see \cref{rem:lwpoptreg} below), but we have imposed them to simplify the proof and because we will need such regularity for the modulated free energy method in \cref{sec:ME}.

\begin{prop}[Local well-posedness] 
\label{prop:lwp}
Let $d\geq 1$, $s < d$, and $\sigma > 0$.

\textbullet \ $(s\le d-1)$ If $\mu^0\in L^\infty(\T^d)$, then there exists a time
\begin{equation}
T \geq \paren*{\frac{\sigma^{1/2}}{C\|\mu^0\|_{L^\infty}}}^{2}\indic_{s<d-1} + \paren*{\frac{\sigma^{\frac{d}{2p}+\frac12}}{C_p \|\mu^0\|_{L^\infty}} }^{\frac{2 p}{p-d}}\indic_{s=d-1},
\end{equation}
where $d<p<\infty$ is arbitrary, such that equation \eqref{eq:mild} has a unique solution $\mu\in C([0,T],L^\infty)$. Moreover, if $\mu_1,\mu_2$ are two solutions to \eqref{eq:mild} on $[0,T]$, then
\begin{equation}\label{eq:IDlips<}
\|\mu_1-\mu_2\|_{C([0,T], L^\infty)} \leq 2\|\mu_1^0-\mu_2^0\|_{L^\infty}.
\end{equation}

\textbullet \ $(s>d-1)$ Let $1\leq p<\infty$ and $\al\geq s+1-d$ satisfy $p>d$ or $\al > s-d+\frac{d}{p}$. If $\mu^0\in L^\infty(\T^d) \cap \dot{W}^{\al,p}(\T^d)$,\footnote{We exclude the case $p=\infty$ because expressions of the form $\|\Dm^{\al}\mu\|_{L^\infty}$ are awkward from the point of view of harmonic analysis. If $\al=n$ is a positive integer, then there is no issue in adapting our proof to the usual Sobolev spaces $\dot{W}^{n,\infty}$.} then for arbitrary $\d \in (s+d-1,1)$, there exists a time
\begin{equation}
T \geq \paren*{\frac{\sigma^{\frac{1+\d}{2}}}{C_\d\|\mu^0\|_{L^\infty\cap \dot{W}^{\al,p}}}}^{\frac{2}{1-\d}},
\end{equation}
such that equation \eqref{eq:mild} has a unique solution $\mu\in C([0,T],L^\infty \cap \dot{W}^{\al,p})$. Moreover, if $\mu_1,\mu_2$ are two solutions to \eqref{eq:mild} on $[0,T]$, then
\begin{equation}\label{eq:IDlips>}
\|\mu_1-\mu_2\|_{C([0,T], L^\infty\cap \dot{W}^{\al,p})} \leq 2\|\mu_1^0-\mu_2^0\|_{L^\infty\cap \dot{W}^{\al,p}}.
\end{equation}

The constant $C$ above depends only on $d,s,\M$ if $d-2\leq s\leq d-1$ and $C_p$ additionally on $p$ if $s=d-1$; $C_\d$ depends additionally on $\al,\d,p$ if $d-1<s<d$. 
\end{prop}
\begin{proof}
We first consider the case $d-2\leq s\leq d-1$. Note that in this case $\mu\mapsto \M\nabla\g\ast\mu$ is an order $s+1-d$ operator, which is either smoothing ($s<d-1$) or of the same order ($s=d-1$) compared to the regularity of $\mu$. Consider the Banach space\footnote{If we were to adapt the proof to $\R^d$, the definition of $X$ should be modified to $X\coloneqq C([0,T], L^1(\R^d)\cap L^\infty(\R^d))$.}
\begin{equation}\label{eq:Xdef}
X \coloneqq C([0,T], L^\infty(\T^d)),
\end{equation}
for some $T>0$ to be determined. For fixed $\mu^0\in L^\infty$, the right-hand side of \eqref{eq:mild} defines a map $\Tc$,
\begin{equation}
\mu \mapsto e^{t\sigma\Delta}\mu^0 - \int_0^{(\cdot)} e^{\sigma ((\cdot)-\tau) \Delta} \div (\mu^\tau \mathbb{M}\nabla\g \ast \mu^\tau) d\tau.
\end{equation}
We aim to show that $\Tc$ is a contraction on the ball $B_R\subset X$, for $R,T>0$ appropriately chosen. Once we have shown this, we can appeal to the Banach fixed point theorem to obtain a unique solution to \eqref{eq:mild} in the class $X$.

By the triangle inequality, for any $t\geq 0$, 
\begin{align}
\|\mu^t\|_{L^\infty} &\leq \|e^{t\sigma\D}\mu^0\|_{L^\infty} + \left\|\int_0^t e^{\sigma(t-\tau)\D}\div\paren*{\mu^\tau\M\nabla\g\ast\mu^\tau}d\tau\right\|_{L^\infty} \nn\\
&\leq \|\mu^0\|_{L^\infty} + \left\|\int_0^t e^{\sigma(t-\tau)\D}\div\paren*{\mu^\tau\M\nabla\g\ast\mu^\tau}d\tau\right\|_{L^\infty},
\end{align}
where the second line follows from the heat kernel being in $L^1$. If $s<d-1$, then we may use Minkowski's inequality together with $\|e^{(t-\tau)\sigma\D}\div\|_{L^1} \lesssim \paren*{\sigma(t-\tau)}^{-1/2}$ to obtain
\begin{align}\label{eq:s<NLwp}
\left\|\int_0^t e^{\sigma(t-\tau)\D}\div\paren*{\mu^\tau\M\nabla\g\ast\mu^\tau}d\tau\right\|_{L^\infty} &\leq \int_0^t \paren*{\sigma(t-\tau)}^{-\frac{1}{2}} \|\mu^\tau\|_{L^\infty} \|\nabla\g\ast\mu^\tau\|_{L^\infty} \nn\\
&\lesssim\|\mu\|_{X}^2\int_0^t \paren*{\sigma(t-\tau)}^{-\frac{1}{2}}d\tau \nn\\
&\lesssim \|\mu\|_{X}^2 (t/\sigma)^{\frac{1}{2}},
\end{align}
where we use that $\nabla\g\in L^1$.\footnote{Note this is obviously false for $\R^d$, and this step needs to be modified with the usual $L^\infty$ Riesz potential interpolation estimate.} If $s=d-1$, then $\M\nabla\g\ast\mu^\tau$ is a matrix transformation of the Riesz transform vector of $\mu^\tau$, which is not bounded on $L^\infty$. Instead, we use the smoothing property \eqref{eq:mDhk} to obtain
\begin{equation}\label{eq:Lpsd-1}
\|e^{\sigma(t-\tau)\D}\div\paren*{\mu^\tau\M\nabla\g\ast\mu^\tau}d\tau\|_{L^\infty} \lesssim \paren*{\sigma(t-\tau)}^{-\frac{d}{2p}-\frac{1}{2}} \|\mu^\tau\M\nabla\g\ast\mu^\tau\|_{L^p}
\end{equation}
for any $p<\infty$ such that $\frac{d}{2p}+\frac{1}{2}<1$ (i.e., $p>d$). By H\"older's inequality and the boundedness of the Riesz transform on $L^p$,
\begin{equation}
\|\mu^\tau\M\nabla\g\ast\mu^\tau\|_{L^p} \leq \|\mu^\tau\|_{L^\infty} \|\M\nabla\g\ast\mu^\tau\|_{L^p} \lesssim \|\mu^\tau\|_{L^\infty}\|\mu^\tau\|_{L^p}.
\end{equation}
Since $\|\cdot\|_{L^p}\leq \|\cdot\|_{L^\infty}$,
\begin{equation}
\left\|\int_0^t e^{\sigma(t-\tau)\D}\div\paren*{\mu^\tau\M\nabla\g\ast\mu^\tau}d\tau\right\|_{L^\infty} \lesssim \sigma^{-\frac{d}{2p}-\frac{1}{2}} t^{\frac{1}{2} - \frac{d}{2p}} \|\mu\|_X^2.
\end{equation}
Repeating the preceding argument, we also have for any $\mu_1,\mu_2\in X$,
\begin{multline}
\left\|\int_0^t e^{\sigma(t-\tau)\D}\div\paren*{\mu_1^\tau\M\nabla\g\ast\mu_1^\tau}d\tau - \int_0^t e^{\sigma(t-\tau)\D}\div\paren*{\mu_2^\tau\M\nabla\g\ast\mu_2^\tau}d\tau\right\|_{L^\infty} \\
\lesssim \|\mu_1-\mu_2\|_X \paren*{ \|\mu_1\|_{X} + \|\mu_2\|_X} \paren*{(t/\sigma)^{\frac{1}{2}}\indic_{s< d-1} + \sigma^{-\frac{d}{2p}-\frac{1}{2}}t^{\frac{1}{2}-\frac{d}{2p}}\indic_{s=d-1}}.
\end{multline}
Suppose that $\mu_1,\mu_2\in B_R \subset X$. Then we have shown
\begin{equation}
\|\Tc(\mu_1)-\Tc(\mu_2)\|_X \leq  CR\|\mu_1-\mu_2\|_X \Big((T/\sigma)^{\frac{1}{2}}\indic_{s< d-1} + \sigma^{-\frac{d}{2p}-\frac{1}{2}}T^{\frac{1}{2}-\frac{d}{2p}}\indic_{s=d-1}\Big) .
\end{equation}
for a constant $C>0$ depending on $d,s,\M$. Fix $R > 2\|\mu^0\|_{L^\infty}$, where $\mu^0$ is the initial datum for the Cauchy problem. Evidently, for $\mu\in B_R$,
\begin{equation}
\|\Tc\mu\|_{X} < \frac{R}{2} + 2CR^2\Big((T/\sigma)^{\frac{1}{2}}\indic_{s< d-1} + \sigma^{-\frac{d}{2p}-\frac{1}{2}}T^{\frac{1}{2}-\frac{d}{2p}}\indic_{s=d-1}\Big).
\end{equation}
Choosing $T>0$ such that
\begin{equation}
2CR\Big((T/\sigma)^{\frac{1}{2}}\indic_{s< d-1} + \sigma^{-\frac{d}{2p}-\frac{1}{2}}T^{\frac{1}{2}-\frac{d}{2p}}\indic_{s=d-1}\Big) = \frac12,
\end{equation}
we have that $\Tc(B_R) \subset B_R$. Now for any $\mu_1,\mu_2\in B_R$, we additionally have
\begin{align}
\|\Tc(\mu_1) - \Tc(\mu_2)\|_{X} &< \frac{1}{2}\|\mu_1-\mu_2\|_{X},
\end{align}
which shows that $\Tc$ is a contraction on $B_R$. We can also extract from our analysis the local Lipschitz dependence on the initial data,
\begin{equation}
\label{eq:IDsled1}
\forall \mu_1,\mu_2\in B_R, \qquad \|\mu_1-\mu_2\|_X \le 2\|\mu_1^0-\mu_2^0\|_{L^\infty}.
\end{equation}

\medskip
We now consider the case $d-1<s<d$. The velocity field $\M\nabla\g\ast\mu$ now has less regularity than $\mu$, and we need to exploit additional smoothing to avoid this loss of regularity. Let $\al,p$ be as in the statement of the proposition. Recycling notation, we will show that the map $\Tc$ above is a contraction on the ball $B_R$ of the Banach space
\begin{equation}
X \coloneqq C([0,T], L^\infty(\T^d) \cap \dot{W}^{\al,p}(\T^d)),
\end{equation}
for some $T,R>0$ to be determined.

By arguing similarly to in \eqref{eq:s<NLwp}, we observe that for any $q>d$,
\begin{align}
\left\|\int_0^t e^{\sigma(t-\tau)\D}\div\paren*{\mu^\tau\M\nabla\g\ast\mu^\tau}d\tau\right\|_{L^\infty} &\leq \int_0^t \paren*{\sigma(t-\tau)}^{-\frac{1}{2}-\frac{d}{2q}} \|\mu^\tau\|_{L^\infty} \|\nabla\g\ast\mu^\tau\|_{L^q}d\tau. \label{eq:Linf1}
\end{align}
If $\infty>p>d$, then we may choose $q=p$ and estimate
\begin{equation}
\|\nabla\g\ast\mu^\tau\|_{L^q} \lesssim \|\Dm^{1+s-d}\mu^\tau\|_{L^p} \lesssim \|\Dm^{\al}\mu^\tau\|_{L^p}. \label{eq:Linf2}
\end{equation}
If $p\leq d$, then our assumption $\al>s-d+\frac{d}{p}$ implies by Sobolev embedding that we may find a $q>d$ such that
\begin{equation}
\|\nabla\g\ast\mu^\tau\|_{L^q} \lesssim \|\Dm^{\al}\mu^\tau\|_{L^p}. \label{eq:Linf3}
\end{equation}
We now conclude that
\begin{align}\label{eq:lwpNLs>Linf}
\left\|\int_0^t e^{\sigma(t-\tau)\D}\div\paren*{\mu^\tau\M\nabla\g\ast\mu^\tau}d\tau\right\|_{L^\infty} \lesssim \|\mu\|_{X}^2 \int_{0}^t \paren*{\sigma(t-\tau)}^{-\frac{1}{2}-\frac{d}{2q}}d\tau \lesssim \|\mu\|_{X}^2 \sigma^{-\frac{1}{2}-\frac{d}{2q}} t^{\frac{1}{2}-\frac{d}{2q}}.
\end{align}

Next, we let $\d \in (s+1-d,1)$. Then again using the smoothing of the heat kernel, followed by the fractional Leibniz rule (e.g., see see \cite[Theorem 7.6.1]{Grafakos2014m} or \cite[Theorem 1.5]{Li2019}),\footnote{The estimates are stated for $\R^d$, but they carry over to $\T^d$ as well when restricted to zero-mean functions.} we find
\begin{multline}\label{eq:lwpFLs>}
\|\Dm^{\al} e^{\sigma(t-\tau)\D}\div\paren*{\mu^\tau\M\nabla\g\ast\mu^\tau}\|_{L^p} \lesssim (\sigma(t-\tau))^{-\frac12 - \frac{\d}{2}}\Big(\|\Dm^{\al-\d}\mu^\tau\|_{L^{p_1}} \|\nabla\g\ast\mu^\tau\|_{L^{p_2}} \\
+\|\mu^\tau\|_{L^{\tl{p}_1}}\|\Dm^{\al-\d}\nabla\g\ast\mu^\tau\|_{L^{\tl p_2}}\Big),
\end{multline}
where $p\leq p_1,\tl{p}_1, p_2, \tl{p}_2\leq\infty$ satisfy $\frac{1}{p_1}+\frac{1}{p_2}=\frac{1}{\tl p_1}+\frac{1}{\tl p_2} = \frac{1}{p}$. We choose $\tl p_1 = \infty$ and $\tl p_2 = p$, so that
\begin{align}\label{eq:Walp1}
\|\mu^\tau\|_{L^{\tl{p}_1}}\|\Dm^{\al-\d}\nabla\g\ast\mu^\tau\|_{L^{\tl p_2}} \lesssim \|\mu^\tau\|_{L^\infty} \|\Dm^{\al}\mu^\tau\|_{L^p},
\end{align}
since $\al-\d+s+1-d<\al$ by choice of $\d$. We choose $p_1 = \frac{\al p }{\al-\d}$ and $p_2 = \frac{\al p}{\d}$, which are evidently H\"older conjugate to $p$. Then by the Gagliardo-Nirenberg interpolation inequality (e.g., see \cite[Theorem 2.44]{BCD2011}),\footnote{Again, this result is stated on $\R^d$ but is adaptable to $\T^d$ for zero-mean functions.}
\begin{align}\label{eq:Walp2}
\|\Dm^{\al-\d}\mu^\tau\|_{L^{p_1}} \lesssim \|\mu^\tau\|_{L^\infty}^{\frac\d\al} \|\Dm^{\al}\mu^\tau\|_{L^{p}}^{1-\frac{\d}{\al}}
\end{align}
and
\begin{align}\label{eq:Walp3}
\|\nabla\g\ast\mu^\tau\|_{L^{p_2}} \lesssim \|\Dm^{\d}\mu^\tau\|_{L^{p_2}} \lesssim \|\mu^\tau\|_{L^\infty}^{1-\frac\d\al} \|\Dm^{\al}\mu^\tau\|_{L^{p}}^{\frac\d\al}.
\end{align}
Combining estimates, we conclude
\begin{align}
\left\|\Dm^\al\int_0^t e^{\sigma(t-\tau)\D}\div\paren*{\mu^\tau\M\nabla\g\ast\mu^\tau}d\tau\right\|_{L^p} &\lesssim \|\mu\|_{X}^2 \int_0^t  (\sigma(t-\tau))^{-\frac12 - \frac{\d}{2}} d\tau \nn\\
&\lesssim \sigma^{-\frac12 -\frac\d2} t^{\frac12 -\frac\d2} \|\mu\|_{X}^2. \label{eq:lwpNLs>Lp}
\end{align}

Putting together the estimates \eqref{eq:lwpNLs>Linf}, \eqref{eq:lwpNLs>Lp} and using the properties of the heat kernel for the linear part of the map $\Tc$, we arrive at
\begin{equation}\label{eq:Tcmus>}
\|\Tc(\mu)\|_{X} \leq \|\mu^0\|_{L^\infty \cap \dot{W}^{\al,p}} + C\|\mu\|_{X}^2\paren*{ \sigma^{-\frac{1}{2}-\frac{d}{2q}} T^{\frac{1}{2}-\frac{d}{2q}} +  \sigma^{-\frac12 -\frac\d2} T^{\frac12 -\frac\d2}},
\end{equation}
for some constant $C>0$ depending on $d,s,p,\al,\d,\M$. Completely analogous analysis also shows that
\begin{equation}\label{eq:Tcmu12s>}
\|\Tc(\mu_1) - \Tc(\mu_2)\|_{X} \leq C\paren*{\|\mu_1\|_{X} + \|\mu_2\|_{X}}\|\mu_1-\mu_2\|_{X}\paren*{ \sigma^{-\frac{1}{2}-\frac{d}{2q}} T^{\frac{1}{2}-\frac{d}{2q}} +  \sigma^{-\frac12 -\frac\d2} T^{\frac12 -\frac\d2}}. 
\end{equation}
Without loss of generality, we may choose $\d$ sufficiently close to $1$ and $q$ sufficiently close to $d$ so that $\d = \frac{d}{q}$. Letting $R>2\|\mu^0\|_{L^\infty \cap \dot{W}^{\al,p}}$, we see that if we choose $T$ so that
\begin{equation}
2C R \sigma^{-\frac{1+\d}{2}} T^{\frac{1-\d}{2}} = \frac{1}{2},
\end{equation}
then $\Tc$ is a contraction on $B_R$. So by the Banach fixed point theorem, we obtain a solution to \eqref{eq:mild} in $X$. Additionally, we can extract from the reasoning used to obtain \eqref{eq:Tcmus>}, \eqref{eq:Tcmu12s>} that if $\mu_1,\mu_2$ are two solutions to \eqref{eq:mild} in $B_R$ with initial data $\mu_1^0,\mu_2^0$, respectively, then
\begin{equation}
\|\mu_1-\mu_2\|_{X} \leq 2\|\mu_1^0-\mu_2^0\|_{L^\infty \cap \dot{W}^{\al,p}},
\end{equation}
which shows local Lipschitz continuity of the solution map on the initial data. This completes the proof of \cref{prop:lwp}.
\end{proof}

We record a series of remarks about \cref{prop:lwp} and basic properties of solutions to equation \eqref{eq:lim}.

\begin{remark}
The proof of local well-posedness makes no assumptions on the \emph{sign} of the interaction or the symmetry properties of the matrix $\M$. In particular, it is valid for both conservative and gradient flows, as well as repulsive and attractive interactions.
\end{remark}

\begin{remark}
By using the fractional Leibniz rule as in \eqref{eq:lwpFLs>}, but skipping the Gagliardo-Nirenberg interpolation, one can also show in the case $d-2\leq s\leq d-1$ that given $\mu^0\in L^\infty \cap \dot{W}^{\al,p}$, for any $\al\geq 0$ and $1\leq p<\infty$, there exists a unique solution to \eqref{eq:mild} in $C([0,T], L^\infty\cap\dot{W}^{\al,p})$ for some $T>0$. 
\end{remark}

\begin{remark}\label{rem:lwpoptreg}
An examination of the proof of \cref{prop:lwp} reveals that the $L^\infty$ assumption is overkill. Indeed, it would suffice to have $\mu\in L^p$ for sufficiently high $p$. We do not pursue such extensions here because we will want to consider solutions which are always in $L^\infty$ for application of the modulated free energy method later in \cref{sec:ME}.
\end{remark}

\begin{remark}\label{rem:BU}
The proof of \cref{prop:lwp} implies a blow-up criterion for the solution. Namely, if $X=L^\infty \cap \dot{W}^{\al,p}$, where $\al,p$ are as above if $d-1<s<d$, equipped with its natural norm, then by iterating the local existence argument, we obtain a maximal lifespan solution $\mu \in C([0,T_{\max}), X)$. If $T_{\max}<\infty$, then we must have that
\begin{equation}
\limsup_{T\rightarrow T_{\max}^-} \|\mu\|_{C([0,T],X)} = \infty.
\end{equation}
Otherwise, there exists an $M>0$, such that for every $T<T_{\max}$, $\|\mu\|_{C([0,T],X)}\leq M$. We choose $T$ sufficiently close to $T_{\max}$ so that $T_{\max}-T$ is less than the lower bound for the time of existence given for a solution by \cref{prop:lwp}. Using \cref{prop:lwp} with initial datum $\mu^{T}$, we can then increase the lifespan of the solution beyond $T_{\max}$, which is a contradiction.
\end{remark}

\begin{remark}\label{rem:clssol}
Using the dependence on initial data estimates \eqref{eq:IDlips<} and \eqref{eq:IDlips>}, we see that we can always approximate solutions to \eqref{eq:mild} by $C^\infty$ solutions, in particular classical solutions.
\end{remark}

\begin{remark}\label{remark:massconserv}
The solutions we have constructed necessarily conserve mass and have nonincreasing $L^1$ norm (conserved if $\M$ is antisymmetric); note that we did not limit ourselves to nonnegative solutions in \cref{prop:lwp}, so the mass and $L^1$ norm \textit{a priori} do not coincide. To show this, it suffices by approximation to consider a smooth solution. Conservation of mass is straightforward:
\begin{equation}
    \frac{d}{dt}\int_{\T^d} \mu^t dx = \int_{\T^d} \div\left((\mu^t \M\nab\g\ast\mu^t) + \sigma\nabla\mu^t \right) dx = 0
\end{equation}
by the fundamental theorem of calculus and periodicity. For nonincrease of the $L^1$ norm, we regularize the function $|\cdot |$ by $\sqrt{\ep^2 + |\cdot|^2}$, which is $C^\infty$ for fixed $\ep > 0$. By the chain rule, 
\begin{equation}
    \frac{d}{dt}\int_{\T^d} \sqrt{\ep^2 + |\mu^t|^2} dx = \int_{\T^d}\frac{\mu^t \left( \div (\mu^t\M\nab\g\ast\mu^t) + \sigma\D\mu^t\right)}{\sqrt{\ep^2 + |\mu^t|^2}}
\end{equation}
and the result is obtained by integration by parts and letting $\ep \to 0$.
\end{remark}

\begin{remark}\label{rem:timers}
By rescaling time, we may always normalize the mass to be unital up to a change of temperature. More precisely, suppose that $\mu$ is a solution to \eqref{eq:lim}. Letting $\bar\mu = \int_{\T^d}\mu^0$, set $\nu^t \coloneqq \bmu^{-1}\mu^{t/\bmu}$. Then using the chain rule,
\begin{equation}
\p_t\nu^t = -\bmu^{-2}\div\paren*{\mu^{t/\bmu}\nabla\g\ast\mu^{t/\bmu} } + \sigma\bmu^{-2}\D\mu^{t/\bmu} = -\div\paren*{\nu^t\nabla\g\ast\nu^t} + \tl\sigma \D\nu^t,
\end{equation}
where $\tl\sigma \coloneqq \sigma/\bmu$. 
\end{remark}

\begin{remark}
\label{remark:positiveness}
If $\mu^0 \geq 0$, then the solution to \eqref{eq:mild} satisfies $\mu^t \geq 0$ on its lifespan. Indeed, let $\mu_+^t, \mu_-^t$ denote the positive and negative parts of $\mu^t$ respectively. Then
\begin{equation}
    \frac{d}{dt}\int_{\T^d} \mu_\pm^t dx = \hal\frac{d}{dt}\int_{\T^d} \left(|\mu^t| \pm \mu^t \right)dx \leq 0,
\end{equation}
where the final inequality follows from \cref{remark:massconserv}. In particular, if $\int_{\T^d} \mu_-^0 dx =0$, then $0\le \int_{\T^d} \mu_-^t dx \le \int_{\T^d} \mu_-^0 dx =0$, which implies that $\mu^t_- = 0$ a.e. As we show later in \cref{sec:Rlx} that $\mu^t$ is smooth, \textit{a fortiori} continuous, for $t>0$, the equality holds pointwise everywhere. The same reasoning shows that if $\mu^0 \le 0$, then $\mu^t \le 0$ for every $t\geq 0$.

Using this result, we can also show in the conservative case that if $c_1 \le \mu^0 \le c_2$ a.e. (i.e., the initial data is bounded from above and below), then $c_1\le\mu^t\le c_2$ a.e. for every $t>0$. Simply consider $\mu^t -c_1$ and $\mu^t-c_2$ which are solutions to \eqref{eq:lim} with initial data $\mu^0 -c_1 \geq 0$ and $\mu^0 - c_2 \leq 0$, respectively (see \cref{rem:consub} below).

For the gradient flow case, let $0\leq c\leq \bmu$ and use integration by parts to compute
\begin{align}
\frac{d}{dt}\int_{\T^d}(\mu^t-c)_{-}dx &= -\int_{\mu^t\leq c}\paren*{\div(\mu^t\nabla\g\ast\mu^t) + \sigma\D\mu^t}dx \nn\\
&=-\int_{\mu^t\leq c}\nabla\mu^t \cdot\nabla\g\ast\mu^tdx + \cd\int_{\mu^t\leq c}\mu^t\Dm^{2+s-d}(\mu^t-\bmu) dx - \sigma\int_{\mu^t=c}\nabla\mu^t\cdot \frac{\nabla\mu^t}{|\nabla\mu^t|}dx. \label{eq:dtmucrhs}
\end{align}
The last term is nonpositive and may be discarded. For the first term, observe that $\nabla\mu^t\indic_{\mu^t\leq c} = -\nabla(\mu^t-c)_{-}$ a.e. Hence, integrating by parts,
\begin{align}
-\int_{\mu^t\leq c}\nabla\mu^t \cdot\nabla\g\ast\mu^t = \cd\int_{\T^d}(\mu^t-c)_{-}\Dm^{2+s-d}(\mu^t)dx.
\end{align}
Similarly, $\mu^t\indic_{\mu^t\leq c} = -(\mu^t-c)_{-} +c\indic_{\mu^t\leq c}$, which implies that the right-hand side of \eqref{eq:dtmucrhs} is $\leq$
\begin{align}\label{eq:rhsL1muc}
\cd c\int_{\mu^t\leq c}\Dm^{2+s-d}(\mu^t)dx.
\end{align}
In particular, in the Coulomb case $s=d-2$ and assuming $\mu^t\geq 0$, the right-hand side is \\ $\leq \cd c|\{\mu^t\leq c\}|(c-\bmu)$. For general $d-2<s<d$, we use the definition of the fractional Laplacian on $\T^d$ (e.g., see \cite[Proposition 2.2]{Cordoba2004}) to write
\begin{align}
\int_{\mu^t\leq c}\Dm^{2+s-d}(\mu^t)dx = C_{s,d}\sum_{k\in\Z^d} \int_{\mu^t\leq c}\int_{\T^d} \frac{\mu^t(x)-\mu^t(y)}{|x-y-k|^{d+(2+s-d)}}dydx.
\end{align}
Evidently,
\begin{equation}
\int_{\mu^t\leq c}\int_{\mu^t>c}\frac{\mu^t(x)-\mu^t(y)}{|x-y-k|^{d+(2+s-d)}}dydx \leq 0.
\end{equation}
Since by swapping $x\leftrightarrow y$ and making the change of variable $-k\mapsto k$,
\begin{equation}
\sum_{k}\int_{\mu^t\leq c}\int_{\mu^t\leq c}\frac{\mu^t(x)-\mu^t(y)}{|x-y-k|^{d+(2+s-d)}}dydx = \sum_{k}\int_{\mu^t\leq c}\int_{\mu^t\leq c}\frac{\mu^t(y)-\mu^t(x)}{|x-y-k|^{d+(2+s-d)}}dydx =0,
\end{equation}
we conclude that the right-hand side of \eqref{eq:rhsL1muc} is $\leq 0$. Hence, $\int_{\T^d}(\mu^t-c)_{-}dx$ is nonincreasing, and so if $c$ is chosen such that $\inf\mu^0\geq c$, then $\int_{\T^d}(\mu^t-c)_{-}dx=0$ for every $t$ in the lifespan of $\mu$. This implies that $\inf\mu^t\geq \inf\mu^0$. An analogous argument shows that if $c\geq \bmu$, then $\int_{\T^d}(\mu^t-c)_{+}dx$ is nonincreasing. In particular, if $c\geq \sup\mu^0$, then $\int_{\T^d}(\mu^t-c)_{+}dx=0$ on the lifespan of $\mu$, implying $\sup\mu^t\leq \sup\mu^0$.
\end{remark}

\begin{remark}\label{rem:consub}
If $c\in\R$, then setting $\nu\coloneqq \mu-c$, one computes
\begin{equation}\label{eq:muconsub}
\p_t\nu = -\div(\nu\M\nabla\g\ast\nu) - c\div(\M\nabla\g\ast\nu) + \sigma\D\nu.
\end{equation}
If $\M$ is antisymmetric, then $\div(\M\nabla\g\ast\mu) = 0$, so $\mu-c$ is a solution to equation \eqref{eq:lim}. Indeed, using that the $(i,j)$ entry $\M^{ij}$ of $\M$ is a constant by assumption, the commutativity of differentiation, and $\M^{ij} = -\M^{ji}$,
\begin{equation}\label{eq:asdivM}
\div(\M\nabla\g\ast\mu) =  \M^{ij}\p_i\p_j(\g\ast\mu) = -\M^{ji}\p_i\p_j(\g\ast\mu) = -\M^{ji}\p_j\p_i(\g\ast\mu) = -\div(\M\nabla\g\ast\mu),
\end{equation}
since the sum over $i,j$ is symmetric under swapping $i\leftrightarrow j$. In particular, one can always take $c=\int_{\T^d}\mu$ and reduce to considering zero-mean solutions in the conservative case.
\end{remark}

We conclude this subsection by showing that the solutions given by \cref{prop:lwp} are, in fact, global (i.e., $T_{\max}=\infty$). We recall the blow-up criterion of \cref{rem:BU}. The case $d-2\leq s\leq d-1$ is a triviality since we know from \cref{remark:positiveness} that the $L^\infty$ norm is nonincreasing. For the case $d-1<s<d$, we are not able to show yet---but we will in the next section---that $\|\mu^t\|_{\dot{W}^{\al,p}}$ is controlled by $\|\mu^0\|_{\dot{W}^{\al,p}}$. But this is unnecessary: it suffices to show that $\|\mu^t\|_{\dot{W}^{\al,p}}$ cannot blow up in finite time.

\begin{prop}[Global well-posedness]\label{prop:gwp}
Under the same assumptions as in the statement of \cref{prop:lwp}, there exists a unique global solution $\mu$ to \eqref{eq:mild} in 
\begin{equation}
\begin{cases}
C\paren*{[0,\infty), L^\infty(\T^d)}, & {d-2\leq s\leq d-1} \\
C\paren*{[0,\infty), L^\infty \cap \dot{W}^{\al,p}(\T^d)}, & {d-1<s<d}.
\end{cases}
\end{equation}
\end{prop}
\begin{proof}
For the case $d-2\leq s\leq d-1$, we have already explained the proof. For the case $d-1<s<d$, we simply need to revisit the proof of \cref{prop:lwp}. Recalling estimates \eqref{eq:Linf1}, \eqref{eq:Linf2}, \eqref{eq:Linf3}, we have
\begin{align}
\left\|\int_0^t e^{\sigma(t-\tau)\D}\div\paren*{\mu^\tau\M\nabla\g\ast\mu^\tau}d\tau\right\|_{L^\infty} \lesssim \|\mu^0\|_{L^\infty}\int_0^t \paren*{\sigma(t-\tau)}^{-\frac{1}{2}-\frac{d}{2q}} \|\Dm^{\al}\mu^\tau\|_{L^p}d\tau,
\end{align}
for some $q>d$ depending on $p$. Recalling estimates \eqref{eq:lwpFLs>}, \eqref{eq:Walp1}, \eqref{eq:Walp2}, \eqref{eq:Walp3},
\begin{align}
\left\|\Dm^{\al}\int_0^t e^{\sigma(t-\tau)\D}\div\paren*{\mu^\tau\M\nabla\g\ast\mu^\tau}d\tau\right\|_{L^p} \lesssim \|\mu^0\|_{L^\infty}\int_0^t (\sigma(t-\tau))^{-\frac12-\frac{\d}{2}}\|\Dm^{\al}\mu^\tau\|_{L^p}d\tau,
\end{align}
for $\d\in (s+1-d,1)$. Note here that in the gradient flow case, we are implicitly using the interaction is repulsive to control $\|\mu^t\|_{L^\infty} \leq \|\mu^0\|_{L^\infty}$ (recall \cref{remark:positiveness}). Hence,
\begin{multline}
\|\mu^t\|_{L^\infty\cap \dot{W}^{\al,p}} \lesssim \|\mu^0\|_{L^\infty\cap\dot{W}^{\al,p}} + \|\mu^0\|_{L^\infty}\int_0^t \paren*{\sigma(t-\tau)}^{-\frac{1}{2}-\frac{d}{2q}} \|\Dm^{\al}\mu^\tau\|_{L^p}d\tau\\
+\int_0^t (\sigma(t-\tau))^{-\frac12 - \frac{\d}{2}} \|\Dm^{\al}\mu^\tau\|_{L^p}d\tau.
\end{multline}
By adjusting $q$ or $\d$ if necessary, we may assume without loss of generality that $\frac{\d}{2} = \frac{d}{2q}$. The singular Gr\"onwall lemma \cite[Ch. 5, Lemma 6.7]{Pazy1983} now implies that for any $T>0$,
\begin{align}
\sup_{0\leq t\leq T} \|\mu^t\|_{L^\infty\cap\dot{W}^{\al,p}} < \infty,
\end{align}
completing the proof of the proposition.
\end{proof}

\subsection{Control of $L^p$ norms}\label{ssec:WPLp}
We already saw in \Cref{remark:massconserv,remark:positiveness} that the $L^1$ and $L^\infty$ norm of a solution are nonincreasing. By interpolation, this implies that for any $1\leq p\leq\infty$,
\begin{align}
\forall t\geq 0, \qquad \|\mu^t\|_{L^p} \leq \|\mu^0\|_{L^1}^{\frac1p} \|\mu^0\|_{L^\infty}^{1-\frac1p}.
\end{align}
As we show in \cref{lem:Lpcon} below, it is possible control the $L^p$ norm, for any $1<p<\infty$, of the solution in terms of the $L^p$ norm of the initial datum.

Before stating and proving the lemma, we recall some technical preliminaries. The first is a type of Poincar\'e inequality adapted from \cite[Lemma 3.2]{AR1981}. Note that when $p=2$, the inequality is the usual Poincar\'{e} inequality, which is a trivial consequence of Plancherel's theorem. We include a sketch of the proof for the reader's convenience.

\begin{lemma}\label{lem:LpPoinc}
Let $d\geq 1$ and $1\leq p<\infty$. There exists a constant $C>0$ depending only on $d,p$, such that for any $f$ with $\int_{\T^d}f = 0$ and $|f|^{\frac{p}{2}}\sgn(f)\in W^{1,2}(\T^d)$, it holds
\begin{align}
\int_{\T^d}|f|^p dx \leq C\int_{\T^d} | \nabla|f|^{\frac{p}{2}}|^2 dx.
\end{align}
\end{lemma}
\begin{proof}
The proof is by contradiction. Suppose there is a sequence of $f_n$ with $\int_{\T^d}f_n = 0$, $\|f_n\|_{L^p} = 1$, and $\|\nabla(|f_n|^{\frac{p}{2}}\sgn(f_n))\|_{L^2}\rightarrow 0$ as $n\rightarrow\infty$. By approximation, we may assume without loss of generality that $f_n$ is $C^\infty$. Set $g_n \coloneqq |f_n|^{\frac{p}{2}}\sgn(f_n)$. Since $\g_n$ is bounded in $W^{1,2}(\T^d)$, the Rellich-Kondrachov embedding implies that (up to a subsequence) there is a $g\in W^{1,2}(\T^d)$ such that $\|g_n-g\|_{L^2}\rightarrow 0$ as $n\rightarrow\infty$. Since $1=\|f_n\|_{L^p}^p = \|g_n\|_{L^2}^2$, it follows that $1=\|g\|_{L^2}$. Since $\|\nabla g_n\|_{L^2} \rightarrow 0$, it follows from lower semicontinuity that $\nabla g=0$, which implies that $g=c$ for some constant $c$. Since $\|g\|_{L^2}=1$, in fact, $c=1$, which implies that $f\coloneqq |g|^{\frac{2}{p}}\sgn(g) = 1$. Since $f_n = |g_n|^{\frac{2}{p}}\sgn(g_n)\rightarrow f$ in $L^p$ by the fact that $g_n\rightarrow g$ in $L^2$, we also have that $f_n\rightarrow f$ in $L^1$. This implies that $\int_{\T^d}fdx=0$, which contradicts that $c=1$ above.
\end{proof}

The second ingredient is a positivity lemma for the fractional Laplacian, adapted from \cite[Lemma 2.5]{Cordoba2004}, which one may view as a ``cheap'' version of the Stroock-Varopoulos inequality (e.g., see \cite[Proposition 3.1]{BIK2015}).

\begin{lemma}[Positivity lemma]
\label{lemma:positivity}
For $0 \le \alpha \le 2$, $f\in C^\infty(\T^d)$,\footnote{The $C^\infty$ assumption is qualitative. One just needs enough regularity for the integral in \eqref{eq:posinteg} to make sense.} and $1\leq q<\infty$, it holds that
\begin{align}\label{eq:posinteg}
    \int_{\T^d}|f|^{q-1}\sgn(f) \Dm^\al f dx \geq 0,
\end{align}
\end{lemma}

We now come to the main lemma of this subsection, which uses the assumption that $\g$ is repulsive in the gradient flow case and $d-2\leq s<d$, in contrast to \cref{prop:lwp}.

\begin{lemma}\label{lem:Lpcon}
Let $d-2\leq s<d$ and $1\leq p<\infty$. If $\mu^t$ is a solution to equation \eqref{eq:lim}, with the condition that $\mu^t\geq 0$ if $\M=-\I$, then
\begin{align}
\forall t\geq 0, \qquad \|\mu^t\|_{L^p} \leq \|\mu^0\|_{L^p}.
\end{align}
Furthermore, in the case when $\M$ is antisymmetric, there is a constant $C>0$ depending only on $d,p$, such that
\begin{align}\label{eq:muLPexpcon}
\forall t\geq 0, \qquad \|\mu^t-\bmu\|_{L^p} \leq e^{-C\sigma t}\|\mu^0-\bmu\|_{L^p}.
\end{align}
\end{lemma}
\begin{proof}
Without loss of generality, suppose that $\mu$ is a classical solution to equation \eqref{eq:lim}. Using the chain rule, we compute
\begin{align}
\frac{d}{dt}\|\mu^t\|_{L^p}^p &= p\int_{\T^d}|\mu^t|^{p-1}\sgn(\mu^t) \p_t\mu^t dx \nn\\
&= p\sigma\int_{\T^d}|\mu^t|^{p-1}\sgn(\mu^t)\D\mu^tdx   - p\int_{\T^d}|\mu^t|^{p-1}\sgn(\mu^t)\div(\mu^t \M\nabla\g\ast\mu^t)dx.
\end{align}
By the product rule,
\begin{multline}\label{eq:divbrhs}
-p\int_{\T^d}|\mu^t|^{p-1}\sgn(\mu^t)\div(\mu^t\M\nabla\g\ast\mu^t)dx \\
= -p\int_{\T^d}|\mu^t|^{p}\div(\M\nabla\g\ast\mu^t)dx  {- p\int_{\T^d}|\mu^t|^{p-1}\sgn(\mu^t)\nabla\mu^t\cdot (\M\nabla\g\ast\mu^t)dx}.
\end{multline}
Since
\begin{equation}
p|\mu^t|^{p-1}\sgn(\mu^t)\nabla\mu^t = \nabla(|\mu^t|^{p}),
\end{equation}
an integration by parts reveals that
\begin{align}
-p\int_{\T^d}|\mu^t|^{p-1}\sgn(\mu^t)\div(\mu^t\M\nabla\g\ast\mu^t)dx  = -(p-1)\int_{\T^d}|\mu^t|^{p}\div(\M\nabla\g\ast\mu^t)dx.
\end{align}
Finally, write
\begin{equation}
\sgn(\mu^t) = \lim_{\vep\rightarrow 0^+}\frac{\mu^t}{\sqrt{\vep^2 + |\mu^t|^2}}.
\end{equation}
Integrating by parts,
\begin{align}
&p\int_{\T^d}|\mu^t|^{p-1}\sgn(\mu^t)\D\mu^tdx \nn\\
&= \lim_{\vep\rightarrow 0^+} \Bigg(-p(p-1)\int_{\T^d} |\mu^t|^{p-2}\frac{|\mu^t|}{\sqrt{\vep^2 + |\mu^t|^2}}|\nabla\mu^t|^2dx \nn\\
&\ph  -p\int_{\T^d} |\mu^t|^{p-1}\frac{|\nabla\mu^t|^2}{\sqrt{\vep^2+|\mu^t|^2}}dx  + p\int_{\T^d} |\mu^t|^{p-1}\frac{|\mu^t|^2 |\nabla \mu^t|^2}{(\vep^2 + |\mu^t|^2)^{3/2}} dx\Bigg) \nn\\
&= -  p(p-1)\int_{\T^d} |\mu^t|^{p-2} |\nabla\mu^t|^2dx \nn\\
&= - \frac{4(p-1)}{p}\int_{\T^d} |\nabla |\mu^t|^{p/2} |^2 dx.
\end{align}
After a little bookkeeping, we arrive at
\begin{align}\label{eq:ddtLpmu}
\frac{d}{dt}\|\mu^t\|_{L^p}^p = -(p-1)\int_{\T^d}|\mu^t|^{p}\div(\M\nabla\g\ast\mu^t)dx- 4\sigma\frac{(p-1)}{p}\int_{\T^d} |\nabla |\mu^t|^{\frac{p}{2}} |^2 dx.
\end{align}

Suppose that $\M$ is antisymmetric. Then recalling \eqref{eq:asdivM}, identity \eqref{eq:ddtLpmu} becomes
\begin{align}
\frac{d}{dt}\|\mu^t\|_{L^p}^p =- 4\sigma\frac{(p-1)}{p} \|\nabla |\mu^t|^{\frac{p}{2}}\|_{L^2}^2,
\end{align}
which implies that $\|\mu^t\|_{L^p}^p$ is nonincreasing and, in fact, strictly decreasing unless $|\mu^t|$ is constant. Additionally, since $\mu^t-\bmu$ is also a solution to equation \eqref{eq:lim}, replacing $\mu^t$ above with $\mu^t-\bmu$, we find that
\begin{align}
\frac{d}{dt}\|\mu^t-\bmu\|_{L^p}^p  = - 4\sigma\frac{(p-1)}{p} \|\nabla |\mu^t-\bmu|^{\frac{p}{2}}\|_{L^2}^2 \leq -\frac{4C_{d,p}\sigma(p-1)}{p} \|\mu^t-\bmu\|_{L^p}^p,
\end{align}
where the final inequality follows from application of \cref{lem:LpPoinc}. From Gr\"onwall's lemma, we then obtain
\begin{align}
\|\mu^t-\bmu\|_{L^p}^p \leq e^{-C_{d,p}\sigma t}\|\mu^0-\bmu\|_{L^p}^p,
\end{align}
where we have redefined $C_{d,p}$ compared to above.

Suppose now that $\M=-\I$. Then
\begin{align}
    \div(\nabla\g\ast\mu^t) = -\cd \Dm^{s-d+2}(\mu^t-\bmu),
\end{align}
so that \eqref{eq:ddtLpmu} becomes
\begin{align}
\frac{d}{dt}\|\mu^t\|_{L^p}^p = - (p-1)\cd\int_{\T^d}|\mu^t|^{p}\Dm^{s-d+2}(\mu^t-\bmu)dx -\frac{ 4\sigma(p-1)}{p}\int_{\T^d} |\nabla |\mu^t|^{\frac{p}{2}} |^2 dx.
\end{align}
If $s=d-2$, then
\begin{align}
\int_{\T^d}|\mu^t|^{p}\Dm^{s-d+2}(\mu^t-\bmu)dx = \|\mu^t\|_{L^{p+1}}^{p+1} - \bmu\|\mu^t\|_{L^p}^p \geq 0,
\end{align}
where we have used that $\mu^t\geq 0$ by assumption, so that $\bmu = \|\mu^t\|_{L^1} \leq \|\mu^t\|_{L^{p}}$, and that $\|\cdot\|_{L^{p+1}}\geq \|\cdot\|_{L^p}$. If $d-2<s<d$, then we may apply \cref{lemma:positivity} with $f=\mu^t$ (again using that $\mu^t\geq 0$ by assumption) to obtain
\begin{align}
\int_{\T^d}|\mu^t|^{p}\Dm^{s-d+2}(\mu^t-\bmu)dx \geq 0.
\end{align}
In all cases, we conclude that
\begin{align}
\frac{d}{dt}\|\mu^t\|_{L^p}^p \leq -\frac{ 4\sigma(p-1)}{p}\|\nabla|\mu^t|^{\frac{p}{2}}\|_{L^2}^2,
\end{align}
which completes the proof of the lemma.
\end{proof}

\begin{remark}
In the Coulomb gradient flow case $\M=-\I$ and $s=d-2$, we can actually obtain an exponential rate of decay for $\|\mu^t-\bmu\|_{L^p}$, for $1\leq p\leq\infty$, through a modification of the proof of \cref{lem:Lpcon}. Since we do not use such a result in this paper, we do do not report on the details.
\end{remark}

\subsection{Free energy and entropy control}\label{ssec:WPfe}
In this subsection, we show how the free energy and entropy provide \emph{a priori} estimates for solutions of equation \eqref{eq:lim}, as well as rates of convergence as $t\rightarrow\infty$ to the unique equilibrium given by the uniform measure on $\T^d$. We assume through this section that $\mu$ is a probability density solution. Recall from \Cref{remark:massconserv,remark:positiveness} that mass and sign are preserved and given a nonnegative solution, we may always rescale time (up to a change of temperature) to normalize the mass to be one, as explained in \cref{rem:timers}. 

If $\M=-\I$, then the \emph{free energy} associated to equation \eqref{eq:lim} is defined by
\begin{equation}\label{eq:FE}
\Fc_{\sigma}(\mu) \coloneqq \sigma\int_{\T^d}\log(\mu)d\mu + \frac{1}{2}\int_{(\T^d)^2}\g(x-y)d\mu^{\otimes 2}(x,y) \eqqcolon \sigma\Ent(\mu) + \Eng(\mu),
\end{equation}
consisting of the entropy and the energy of $\mu$. Evidently, both terms in the definition of $\Fc_{\sigma}(\mu)$ are nonnegative and, in fact, are equal to zero if and only if $\mu\equiv 1$. Equation \eqref{eq:lim} with $\M=-\I$ is the gradient flow of $\Fc_\sigma$, in the sense that \eqref{eq:lim} may be rewritten as
\begin{align}
\p_t\mu = -\div\paren*{\mu\nabla\frac{\d\Fc_\sigma}{\d\mu}},
\end{align}
where $\frac{\d\Fc_\sigma}{\d\mu}$ is the variational derivative of $\Fc_{\sigma}$. Consequently,
\begin{align}\label{eq:FEDid}
\frac{d}{dt}\Fc_\sigma(\mu^t) = -\Dc_{\sigma}(\mu^t),
\end{align}
where $\Dc_{\sigma}$ is the \emph{free energy dissipation functional} given by
\begin{align}\label{eq:FEDdef}
\Dc_{\sigma}(\mu^t) \coloneqq \int_{\T^d}\left|\sigma\nabla\log(\mu^t) + \nabla(\g\ast \mu^t) \right|^2 d\mu^t.
\end{align}
There is a deeper significance behind the relationship of the free energy to equation \eqref{eq:lim} in terms of gradient flows on the manifold of probability densities on $\T^d$ (e.g., see \cite{AGS2008}). But we will not make use of this structure and therefore make no further comments on it.

Expanding the square,
\begin{align}
\Dc_{\sigma}(\mu^t) &= \int_{\T^d}\left|\sigma \nabla\log(\mu^t) +\nabla\g\ast \mu^t\right|^2 d\mu^t \nn\\
&= \sigma^2\int_{\T^d}|\nabla \log(\mu^t)|^2 d\mu^t + 2\sigma\int_{\T^d}\paren*{\nabla\log(\mu^t)\cdot \nabla\g\ast \mu^t}d\mu^t + \int_{\T^d}|\nabla\g\ast \mu^t|^2 d\mu^t \nn\\
&\geq \sigma^2\int_{\T^d}|\nabla\log(\mu^t)|^2 d\mu^t + {2\sigma\cd}\int_{\T^d}|\Dm^{\frac{2+s-d}{2}}(\mu^t)|^2 dx+ \int_{\T^d}|\nabla\g\ast \mu^t|^2 d\mu^t, \label{eq:FEDlb}
\end{align}
where the final line follows from integration by parts in the second integral of the penultimate line and the definition of $\g$. Thus, we see that $\Dc_{\sigma}(\mu^t) = 0$ if and only if $\mu^t\equiv 1$. Since the uniform measure is a stationary solution to equation \eqref{eq:lim}, uniqueness of solutions implies that if $\mu^t\equiv 1$ for some $t=t_0$, then $\mu^t\equiv 1$ for all $t\geq t_0$. Thus, the free energy is strictly decreasing unless $\mu^t\equiv 1$, at which point it is then constant for all future time.

\medskip

For conservative flows, the free energy is no longer the right quantity to consider. Instead, the entropy alone suffices. Given a classical solution of equation \eqref{eq:lim} with antisymmetric matrix $\M$, we compute
\begin{align}
\frac{d}{dt}\int_{\T^d}\log(\mu^t)d\mu^t &= \int_{\T^d}\p_t\mu^t\paren*{1+\log(\mu^t)}dx \nn\\
&=\int_{\T^d}\paren*{\sigma\D\mu^t-\div(\mu^t\M\nabla\g\ast\mu^t)} dx -\int_{\T^d}\div(\mu^t\M\nabla\g\ast\mu^t)\log(\mu^t)dx \nn\\
&\ph + \sigma\int_{\T^d}\log(\mu^t)\D\mu^t dx.
\end{align}
By the fundamental theorem of calculus, the first term in the right-hand side of the second equality is obviously zero. The second term is also zero. To see this, we integrate by parts using that $\div(\M\nabla\g\ast\mu^t)=0$,
\begin{align}
 -\int_{\T^d}\div(\mu^t\M\nabla\g\ast\mu^t)\log(\mu^t)dx = \int_{\T^d}\nabla\mu^t \cdot\M\nabla\g\ast\mu^t dx = -\int_{\T^d}\div(\M\nabla\g\ast\mu^t)d\mu^t = 0.
\end{align}
For the third term above, we also integrate by parts to obtain
\begin{align}
\sigma\int_{\T^d}\log(\mu^t)\D\mu^t dx = -\sigma\int_{\T^d}\nabla\log(\mu^t)\cdot \nabla\mu^t dx = -\sigma\int_{\T^d}|\nabla\log(\mu^t)|^2 d\mu^t.
\end{align}
The right-hand side is the \emph{Fisher information} of $\mu^t$. Putting everything together, we find
\begin{align}\label{eq:dtentcon}
\frac{d}{dt}\int_{\T^d}\log(\mu^t)d\mu^t = -\sigma\int_{\T^d}|\nabla\log(\mu^t)|^2 d\mu^t.
\end{align}
Similar to the free energy dissipation functional, we note that the right-hand side is zero if and only if $\mu^t\equiv 1$. If $\mu^t\equiv 1$ for some $t=t_0$, then the uniqueness of solutions and the fact that $1$ is a stationary solution imply $\mu^t\equiv 1$ for all $t\geq t_0$. Thus, the entropy is strictly decreasing unless at some time $t_0$, $\mu^{t_0}\equiv 1$, at which point the solution remains identically one and the entropy is zero for all subsequent time $t\geq t_0$.


The dissipation of free energy/entropy can be combined with the log Sobolev inequality (LSI) for the uniform measure on $\T^d$ to obtain rates of convergence to equilibrium as $t\rightarrow\infty$. We reproduce this LSI from \cite[Lemma 3]{GlBM2021}.

\begin{lemma}\label{lem:LSI}
For any probability density $f$ on $\T^d$,
\begin{align}
\int_{\T^d}\log(f)df \leq \frac{1}{8\pi^2}\int_{\T^d}|\nabla\log(f)|^2df.
\end{align}
\end{lemma}

Using \cref{lem:LSI}, we may obtain the following exponential rate of decay for the free energy/entropy. Pinsker's inequality (e.g., see \cite[Remark 22.12]{Villani2009}) and interpolation imply an exponential rate of convergence to the uniform distribution as $t\rightarrow\infty$ in any $L^p$ norm, for finite $p$.

\begin{lemma}\label{lem:entFEdcay}
Let $\mu$ be a probability density solution of equation \eqref{eq:lim}. If $\M=-\I$, then
\begin{align}\label{eq:entFEdcay}
\forall t\geq 0, \qquad \Fc_{\sigma}(\mu^t) \leq \Fc_{\sigma}(\mu^0) e^{-8\pi^2\sigma t}, 
\end{align}
and for $1\leq p<r\leq\infty$,
\begin{align}\label{eq:entFEdcaygf}
\forall t\geq 0, \qquad \|\mu^t-1\|_{L^p} \leq (1+\|\mu^0\|_{L^\infty})^{1-\frac{\frac1p-\frac1r}{1-\frac1r}}\paren*{ e^{-4\pi^2\sigma t}\sqrt{2\Fc_{\sigma}(\mu^0)/\sigma}}^{\frac{\frac1p-\frac1r}{1-\frac1r}}.
\end{align}
Similarly, if $\M$ is antisymmetric, then
\begin{align}
\forall t\geq 0,\qquad \Ent(\mu^t) \leq e^{-8\pi^2\sigma t}\Ent(\mu^0),
\end{align}
and for $1\leq p<r\leq\infty$,
\begin{align}
\forall t\geq 0, \qquad \|\mu^t-1\|_{L^p} \leq (1+\|\mu^0\|_{L^r})^{1-\frac{\frac1p-\frac1r}{1-\frac1r}}\paren*{e^{-4\pi^2\sigma t} \sqrt{2\Ent(\mu^0)}}^{\frac{\frac1p-\frac1r}{1-\frac1r}}.
\end{align}
\end{lemma}
\begin{proof}
By approximation, we may assume without loss of generality that $\mu$ is a classical solution. Consider first the gradient flow case. From \eqref{eq:FEDlb}, we find
\begin{align}
\frac{d}{dt}\Fc_{\sigma}(\mu^t) &\leq -\sigma^2\int_{\T^d}|\nabla\log(\mu^t)|^2 d\mu^t - {2\sigma\cd}\int_{\T^d}|\Dm^{\frac{2+s-d}{2}}(\mu^t)|^2 dx .
\end{align}
Observe from Plancherel's theorem that
\begin{align}
\int_{\T^d}|\Dm^{\frac{2+s-d}{2}}(\mu^t)|^2 dx \geq (2\pi)^2 \int_{\T^d} |\Dm^{\frac{s-d}{2}}(\mu^t)|^2dx = \frac{(2\pi)^2}{\cd}\int_{(\T^d)^2}\g(x-y)d(\mu^t)^{\otimes 2}(x,y).
\end{align}
Applying \cref{lem:LSI}, we then find 
\begin{align}
\frac{d}{dt}\Fc_{\sigma}(\mu^t) &\leq -8\pi^2\sigma^2\int_{\T^d}\log(\mu^t)d\mu^t -2\sigma(2\pi)^2\int_{(\T^d)^2}\g(x-y)d(\mu^t)^{\otimes 2}(x,y) \nn\\
&\leq -8\pi^2\sigma\Fc_{\sigma}(\mu^t).
\end{align}
By Gr\"onwall's lemma, we conclude that
\begin{equation}\label{eq:gfFEGron}
\Fc_{\sigma}(\mu^t) \leq \Fc_{\sigma}(\mu^0) e^{-8\pi^2\sigma t}.
\end{equation}
For the conservative case, the argument is essentially the same, except we now use the entropy dissipation \eqref{eq:dtentcon} instead of the free energy dissipation \eqref{eq:FEDid}. We ultimately obtain
\begin{align}\label{eq:conentGron}
\Ent(\mu^t) \leq e^{-8\pi^2\sigma t}\Ent(\mu^0).
\end{align}

By Pinsker's inequality, \eqref{eq:gfFEGron} and \eqref{eq:conentGron} respectively imply
\begin{align}
\frac{\sigma}{2}\|\mu^t-1\|_{L^1}^2 + \frac{1}{2}\|\mu^t-1\|_{\dot{H}^{\frac{s-d}{2}}}^2 \leq \Fc_{\sigma}(\mu^0) e^{-8\pi^2\sigma t}, \qquad \text{if $\M=-\I$}, \label{eq:GFPin}
\end{align}
\begin{align}
\frac{1}{2}\|\mu^t-1\|_{L^1}^2 &\leq e^{-8\pi^2\sigma t}\Ent(\mu^0),  \qquad \text{if $\M$ is antisymmetric.} \label{eq:conPin}
\end{align}
By \cref{remark:positiveness}, we know that $\|\mu^t\|_{L^r} \leq \|\mu^0\|_{L^r}$. Combining this fact with \eqref{eq:GFPin}, \eqref{eq:conPin} and interpolation, we conclude that for any $1\leq p<\infty$, if $\M=-\I$, then 
\begin{align}
\|\mu^t-1\|_{L^p} \leq \|\mu^t-1\|_{L^1}^{\frac{\frac1p-\frac1r}{1-\frac1r}} \|\mu^t-1\|_{L^r}^{1-\frac{\frac1p-\frac1r}{1-\frac1r}} \leq (1+\|\mu^0\|_{L^r})^{1-\frac{\frac1p-\frac1r}{1-\frac1r}}\paren*{ e^{-4\pi^2\sigma t}\sqrt{2\Fc_{\sigma}(\mu^0)/\sigma}}^{\frac{\frac1p-\frac1r}{1-\frac1r}},
\end{align}
and if $\M$ is antisymmetric, then similarly
\begin{align}
\|\mu^t-1\|_{L^p} \leq (1+\|\mu^0\|_{L^r})^{1-\frac{\frac1p-\frac1r}{1-\frac1r}}\paren*{e^{-4\pi^2\sigma t} \sqrt{2\Ent(\mu^0)}}^{\frac{\frac1p-\frac1r}{1-\frac1r}}.
\end{align}
This completes the proof of the lemma.
\end{proof}

\subsection{$L^p$-$L^q$ smoothing}\label{ssec:WPLpLq}
In the previous subsections, the decay estimate for $\|\mu^t-1\|_{L^p}$ required at least control on $\|\mu^0\|_{L^{p+}}$ when using the free energy/entropy. Moreover, we have not yet obtained a decay estimate for $\|\mu^t-1\|_{L^\infty}$. In this subsection, we show that it is possible to obtain decay estimates in terms of much weaker control on the initial data.

To do this, we need a log-Sobolev inequality for the uniform measure on $\T^d$. We could not find in the literature the exact form that we need, so we include a proof below (without any outright claims of originality).

\begin{lemma}
\label{lemma:logSobolev}
There exist constants $C_{LS,1}, C_{LS,2}>0$, which depend only on $d$, such that for any $f\in C^\infty(\T^d)$, with $\bar f\coloneqq \int_{\T^d}fdx$, and any $a > 0$, we have
\begin{equation}
\int_{\T^d} f^2 \log \left( \frac{f^2}{\int_{\T^d} f^2} \right) dx + \frac{d}{2}\log(a)\|f\|_{L^2}^2\leq  a\frac{d}{2}\left(C_{LS,1}\|\nabla f\|_{L^2}^2 + C_{LS,2} |\bar f|^2\right).
\end{equation}
\end{lemma}
\begin{proof}
The argument is classical and proceeds through Jensen's inequality and Sobolev embedding. Suppose first that $\int_{\T^d}f^2dx=1$, so that $f^2$ is a probability density on $\T^d$. Then for $p>2$, write
\begin{equation}
\int_{\T^d}f^2\log f dx = \frac{1}{p-2}\int_{\T^d}f^2\log(f^{p-2})dx \leq \frac{1}{p-2}\log(\|f\|_{L^p}^p) = \frac{p}{2(p-2)}\log(\|f\|_{L^p}^2).
\end{equation}
From the inequality $\log t\leq at -\log a$, for any $t,a>0$, it follows that
\begin{align}
\frac{p}{2(p-2)}\log(\|f\|_{L^p}^2) \leq \frac{p}{2(p-2)}\paren*{a\|f\|_{L^p}^2 -\log a }.
\end{align}
Observe from the triangle inequality that $\|f\|_{L^p} \leq \|f-\bar f\|_{L^p} + |\bar f|$. If $d\geq 3$, then we choose $p=\frac{2d}{d-2}$ and use Sobolev embedding to obtain that the right-hand side is $\leq$
\begin{align}
\frac{d}{4}\paren*{2a C_{Sob,\frac{2d}{d-2}}^2 \|\nabla f\|_{L^2}^2 + 2a|\bar f|^2 - \log  a}.
\end{align}
If $d\leq 2$, then we choose any $2<p<\infty$ and use Sobolev embedding plus interpolation to instead obtain
\begin{align}
\frac{p}{2(p-2)}\log(\|f\|_{L^p}^2) &= \frac{p}{2(p-2)}\cdot \paren*{\frac{2p}{d(p-2)}}^{-1}\log\paren*{\|f\|_{L^p}^{\frac{4p}{d(p-2)}}}\nn\\
&\leq \frac{d}{4}\paren*{a\paren*{\|f-\bar f\|_{L^p} + |\bar f| }^{\frac{4p}{d(p-2)}} - \log a} \nn\\
&\leq \frac{d}{4}\paren*{a\paren*{C_{Sob,p}\|f-\bar f\|_{L^2}^{1-d(\frac12-\frac1p)}\|f-\bar f\|_{\dot{H}^1}^{d(\frac12-\frac1p)} + |\bar f|}^{\frac{4p}{d(p-2)}} - \log a}\nn\\
&\leq \frac{d}{4}\paren*{2^{\frac{4p}{d(p-2)}-1}a\paren*{|\bar f|^{\frac{4p}{d(p-2)}} + C_{Sob,p}^{\frac{4p}{d(p-2)}} 2^{\frac{4p}{d(p-2)}\left(1-d(\frac12-\frac1p)\right)} \|\nabla f\|_{L^2}^{2}} - \log a} \nn\\
&\leq \frac{d}{4}\paren*{2^{\frac{4p}{d(p-2)}-1}a\paren*{|\bar f|^2 + C_{Sob,p}^{\frac{4p}{d(p-2)}} 2^{\frac{4p}{d(p-2)}\left(1-d(\frac12-\frac1p)\right)} \|\nabla f\|_{L^2}^{2}} - \log a},
\end{align}
where we have implicitly used above that $|\bar f|\leq \|f\|_{L^2}=1$ and the convexity of $|\cdot|^{\frac{4p}{d(p-2)}}$. To remove the assumption $\int_{\T^d} f^2=1$, we apply the preceding argument to $g\coloneqq f/\|f\|_{L^2}$, which satisfies $\int_{\T^d}g^2 dx = 1$. This then gives the inequality in the statement of the lemma.
\end{proof}

Next, we show that for any time $t>0$, the $L^p$ norm of $\mu^t$ controls the $L^q$ norm of $\mu^t$, for $1\leq p\leq q\leq\infty$, at the cost of a factor blowing up like $(\sigma t)^{-\frac{d}{2}\left(\frac{1}{p}-\frac{1}{q}\right)}$ as $t\rightarrow 0$ and like $e^{C\sigma t}$ as $t\rightarrow\infty$. In other words, this gain of integrability, sometimes called \emph{hypercontractivity}, is only useful for short positive times. This is in contrast to the setting of $\R^d$, where one has this $L^p$-$L^q$ control with only a factor of $(\sigma t)^{-\frac{d}{2}\left(\frac{1}{p}-\frac{1}{q}\right)}$, which yields the optimal decay of solutions as $t\rightarrow\infty$  (cf. \cite[Proposition 3.8]{RS2021}).

\begin{lemma}\label{lem:CLdcy}
Let $d\geq 1$ and $d-2 \leq s<d$. If $\mu$ is a solution to equation \eqref{eq:lim}, then for $1\leq p\leq q\leq\infty$,
\begin{equation}
\forall t> 0, \qquad \|\mu^t\|_{L^q} \le C_{p,q,d} (\sigma t)^{-\frac{d}{2}(\frac{1}{p} - \frac{1}{q})}e^{C_{p,q,d}\sigma t}\|\mu^0\|_{L^p},
\end{equation}
where $C_{p,q,d}>0$ depends only on $p,q,d$. 
\end{lemma}
\begin{proof}
We have already seen that we may assume without loss of generality that $\mu$ is spatially smooth on its lifespan and $\mu$ is $C^\infty$ in time. Therefore, there are no issues of regularity or decay in justifying the computations to follow. The proof is based on an adaptation of an argument, originally due to Carlen and Loss in \cite{CL1995} and extended in \cite{RS2021}.

For given $p,q$ as above, let $r:[0,T]\rightarrow [p,q]$ be a $C^1$ increasing function to be specified momentarily. Replacing the absolute value $|\cdot|$ with $(\vep^2+|\cdot|^2)^{1/2}$, differentiating, then sending $\vep\rightarrow 0^+$, we find that
\begin{multline}
r(t)^2\|\mu^t\|_{L^{r(t)}}^{r(t)-1}\frac{d}{dt}\|\mu^t\|_{L^{r(t)}} = \dot{r}(t)\int_{\T^d}|\mu^t|^{r(t)} \log \paren*{\frac{|\mu^t|^{r(t)}}{\|\mu^t\|_{L^{r(t)}}^{r(t)}}} dx  \\
+  {\sigma}r(t)^2\int_{\T^d}|\mu^t|^{r(t)-1}\sgn(\mu^t)\D\mu^tdx - {r(t)^2}\int_{\T^d}|\mu^t|^{r(t)-1}\sgn(\mu^t)\div(\mu^t \M\nabla\g\ast\mu^t)dx . \label{eq:CLrhs}
\end{multline}
Above, we have used the calculus identity
\begin{equation}
\frac{d}{dt} x(t)^{y(t)} = \dot{y}(t)x(t)^{y(t)}\log x(t) + y(t)\dot{x}(t)x(t)^{y(t)-1}
\end{equation}
for $C^1$ functions $x(t)>0$ and $y(t)$. As shown in the proof of \cref{lem:Lpcon},
\begin{multline}
\sigma r(t)^2\int_{\T^d}|\mu^t|^{r(t)-1}\sgn(\mu^t)\D\mu^tdx  - {r(t)^2}\int_{\T^d}|\mu^t|^{r(t)-1}\sgn(\mu^t)\div(\mu^t \M\nabla\g\ast\mu^t)dx  \\
\leq - 4\sigma(r(t)-1)\int_{\T^d} |\nabla |\mu^t|^{r(t)/2} |^2 dx.
\end{multline}
Hence,
\begin{multline}\label{eq:LSIpreapp}
r(t)^2\|\mu^t\|_{L^{r(t)}}^{r(t)-1}\frac{d}{dt}\|\mu^t\|_{L^{r(t)}} \leq \dot{r}(t)\int_{\T^d}|\mu^t|^{r(t)} \log \paren*{\frac{|\mu^t|^{r(t)}}{\|\mu^t\|_{L^{r(t)}}^{r(t)}}} dx \\
- 4\sigma(r(t)-1)\int_{\T^d} |\nabla |\mu^t|^{r(t)/2} |^2 dx.
\end{multline}

We apply \cref{lemma:logSobolev} to the right-hand side of inequality \eqref{eq:LSIpreapp} with choice $a = \frac{8\sigma (r(t)-1)}{\dot{r}(t)d C_{LS,1}}$ and $f= |\mu^t|^{r(t)/2}$ to obtain that
\begin{multline}\label{eq:LSIpostapp}
r(t)^2\|\mu^t\|_{L^{r(t)}}^{r(t)-1}\frac{d}{dt}\|\mu^t\|_{L^{r(t)}} \leq -\dot{r}(t)\frac{d}{2}\log\paren*{ \frac{8\sigma(r(t)-1)}{\dot{r}(t)d C_{LS,1}}}\|\mu^t\|_{L^{r(t)}}^{r(t)} \\
 + \frac{4\sigma C_{LS,2}(r(t)-1)}{C_{LS,1}}\|\mu^t\|_{L^{r(t)/2}}^{r(t)},
\end{multline}
with the implicit understanding that $\dot{r}(t)>0$ (i.e., $r$ is strictly increasing). Define
\begin{equation}
G(t) \coloneqq \log\paren*{\|\mu^t\|_{L^{r(t)}}}.
\end{equation}
Then it follows from \eqref{eq:LSIpostapp} that
\begin{align}
\frac{d}{dt}G(t) = \frac{1}{ \|\mu^t\|_{L^{r(t)}}}\frac{d}{dt} \|\mu^t\|_{L^{r(t)}} &\leq -\frac{\dot{r}(t)}{r(t)^2}\frac{d}{2}\log\paren*{ \frac{8\sigma(r(t)-1)}{\dot{r}(t)d C_{LS,1}}} \nn\\
&\ph + \frac{4\sigma C_{LS,2}(r(t)-1)}{C_{LS,1}r(t)^2}\frac{\|\mu^t\|_{L^{r(t)/2}}^{r(t)}}{\|\mu^t\|_{L^{r(t)}}^{r(t)}}. \label{eq:ddtG}
\end{align}
By H\"older's inequality, $\frac{\|\mu^t\|_{L^{r(t)/2}}}{\|\mu^t\|_{L^{r(t)}}}\leq 1$. Seting $s(t)\coloneqq 1/r(t)$ and writing $\frac{r-1}{\dot{r}}= -\frac{s(1-s)}{\dot{s}}$, we find from \eqref{eq:ddtG} that
\begin{equation}
\frac{d}{dt}G(t) \leq \dot{s}(t)\paren*{\frac{d}{2}\log\paren*{\frac{8\sigma}{dC_{LS,1}}s(t)(1-s(t))}} + \frac{d}{2}(-\dot{s}(t))\log(-\dot{s}(t)) + \frac{4\sigma C_{LS,2}}{C_{LS,1}}\paren*{s(t)-s(t)^2}.
\end{equation}
So by the fundamental theorem of calculus,
\begin{multline}\label{eq:Gexp}
G(T)-G(0)  \leq \int_0^T\dot{s}(t)\frac{d}{2}\log\paren*{\frac{8\sigma}{dC_{LS,1}} s(t)(1-s(t))}dt \\
 - \frac{d}{2}\int_0^T \dot{s}(t)\log\paren*{-\dot{s}(t)}dt +\frac{4\sigma C_{LS,2}}{C_{LS,1}}\int_0^T\paren*{s(t)-s(t)^2}dt.
\end{multline}
We require that $s(0) = 1/p$ and $s(T) = 1/q$, so by the fundamental theorem of calculus,
\begin{equation}
\int_0^T\dot{s}(t)\frac{d}{2}\log(\frac{8\sigma}{dC_{LS,1}} s(t)(1-s(t)))dt = \frac{d}{2}\paren*{\log(\frac{8\sigma}{dC_{LS,1}})s + \log\paren*{s^s(1-s)^{-(1-s)}} -2s }|_{s=1/p}^{s=1/q}.
\end{equation}
Using the convexity of $a \mapsto a\log a$, we minimize the second integral in the right-hand side of \eqref{eq:Gexp} by choosing $s(t)$ to linearly interpolate between $s(0)=1/p$ and $s(T) = 1/q$, i.e.
\begin{equation}
\dot{s}(t) = \frac{1}{T}\paren*{\frac{1}{q}-\frac{1}{p}}, \qquad 0\leq t\leq T.
\end{equation}
Thus by fundamental theorem of calculus,
\begin{equation}
-\frac{d}{2}\int_0^T \dot{s}(t)\log\paren*{-\dot{s}(t)}dt = -\frac{d}{2}\paren*{\frac{1}{p}-\frac{1}{q}}\log\paren*{\frac{T}{1/p-1/q}}
\end{equation}
and
\begin{align}
\frac{4\sigma C_{LS,2}}{C_{LS,1}}\int_0^T\paren*{s(t)-s(t)^2}dt &= \paren*{\frac{1}{T}\left(\frac1q-\frac1p\right) }^{-1}\frac{4\sigma C_{LS,2}}{C_{LS,1}}\int_0^T \dot{s}(t)\paren*{s(t)-s(t)^2}dt \nn\\
&=\paren*{\frac{1}{T}\left(\frac1q-\frac1p\right) }^{-1}\frac{4\sigma C_{LS,2}}{C_{LS,1}}\left(\frac{s^2}{2} - \frac{s^3}{3}\right)\vert_{s=1/p}^{s=1/q}.
\end{align}
The desired conclusion now follows from a little bookkeeping and exponentiating both sides of the inequality. 
\end{proof}

Combining \cref{lem:CLdcy} with \cref{lem:Lpcon}, we obtain the following corollary.

\begin{cor}\label{cor:CLdcy}
Let $\mu$ be a solution to equation \eqref{eq:lim}. Then for any $1\leq p\leq q\leq\infty$, there exists a constant $C_1>0$ depending on $d,p,q$ such that
\begin{align}\label{eq:CLdcyhypgf}
\forall t>0, \qquad \|\mu^t\|_{L^q} \leq C_1\paren*{\min(\sigma t,1)}^{-\frac{d}{2}\left(\frac1p-\frac1q\right)}\|\mu^0\|_{L^p}.
\end{align}
Additionally, if $\M$ is antisymmetric, and $1\leq p<\infty$, then there is a constant $C_2$ depending on $d,p$ such that
\begin{align}\label{eq:CLdcycon}
\forall t>0, \qquad \|\mu^t-1\|_{L^q} \leq C_1\paren*{\min(\sigma t,1)}^{-\frac{d}{2}\left(\frac1p-\frac1q\right)}e^{-C_2\sigma t}\|\mu^0-1\|_{L^p}.
\end{align}
\end{cor}
\begin{proof}
If $\sigma t\leq 1$, then by \cref{lem:CLdcy},
\begin{align}
\|\mu^t\|_{L^q} &\leq C_{p,q,d} (\sigma(t/2))^{-\frac{d}{2}(\frac{1}{p} - \frac{1}{q})}e^{C_{p,q,d}\sigma(t/2)}\|\mu^{t/2}\|_{L^p}\nn\\
&\leq C_{p,q,d} (\sigma(t/2))^{-\frac{d}{2}\left(\frac1p-\frac1q\right)} e^{C_{p,q,d}\sigma(t/2)} \|\mu^0\|_{L^p},
\end{align}
where the final inequality is by \cref{lem:Lpcon}.  If $\sigma t>1$, then by time translation invariance of solutions and \cref{lem:CLdcy} again,
\begin{align}
\|\mu^t\|_{L^q} \leq C_{p,q,d}(\sigma(1/2\sigma))^{-\frac{d}{2}\left(\frac1p-\frac1q\right)}e^{C_{p,q,d}\sigma(1/2\sigma)}\|\mu^{t-\frac{1}{2\sigma}}\|_{L^p} \leq C_{p,q,d}' \|\mu^0\|_{L^p}.
\end{align}

In the conservative case, we may also obtain from combining \Cref{lem:Lpcon,lem:CLdcy},
\begin{align}
\|\mu^t-1\|_{L^q} &\leq C_{p,q,d}\paren*{\sigma\min(t/2,\frac{1}{2\sigma})}^{-\frac{d}{2}\left(\frac1p-\frac1q\right)} e^{C_{p,q,d}\sigma\min(t/2,\frac{1}{2\sigma})}\|\mu^{\max(\frac{t}{2}, t-\frac{1}{2\sigma})}-1\|_{L^p} \nn\\
&\leq C _{p,q,d}'\paren*{\min(\sigma t,1)}^{-\frac{d}{2}\left(\frac1p-\frac1q\right)} e^{-C_{p,d}\sigma\max(\frac{t}{2},t-\frac{1}{2\sigma})}\|\mu^0-1\|_{L^p},
\end{align}
provided that $p<\infty$. This then completes the proof of the lemma.
\end{proof}

The preceding corollary does not yield a rate of decay for $\|\mu^t-1\|_{L^p}$ in the gradient flow case $\M=-\I$ and \cref{lem:entFEdcay} by itself does not give a rate of decay when $p=\infty$. However, by combining \cref{lem:entFEdcay} with \cref{lem:CLdcy}, we can obtain such a rate of decay, but only under the restriction $d-2\leq s\leq d-1$. The reason for this restriction is that $\M\nabla\g\ast\mu^t$ loses derivatives compared to $\mu^t$ if $d-1<s<d$, an issue we have already seen---and will continue to see---in this paper.

\begin{cor}\label{cor:CLdcy2}
Suppose that $d\geq 2$, $d-2\leq s\leq d-1$ and that $\M=-\I$. Then there exist constants $C_1,C_2>0$ depending only on the dimension $d$, such that
\begin{align}\label{eq:CLdcygf}
\forall t>0, \qquad \|\mu^t-1\|_{L^\infty} \leq C_1 \sigma^{-d-\frac32}(\sigma t)^{-\frac{(d^2+d-1)}{2}} e^{-C_2\sigma t}\sqrt{\Fc_{\sigma}(\mu^0)}.
\end{align}
\end{cor}
\begin{proof}
The argument exploits the mild formulation of the equation \eqref{eq:mild} together with the smoothing properties of the heat kernel. Arguing similar to in the proof of \cref{prop:lwp}, for any $1\leq q\leq\infty$, we have
\begin{align}
\|\mu^t-1\|_{L^q} &\leq \|e^{t\sigma\D}(\mu^0-1)\|_{L^q} + \int_0^t\left\| e^{\sigma(t-\tau)\D}\div\paren*{\mu^\tau\M\nabla\g\ast\mu^\tau}\right\|_{L^q}d\tau \nn\\
&\leq C_{d,p,q}(\sigma t)^{-\frac{d}{2}\left(\frac1p-\frac1q\right)}\|\mu^0-1\|_{L^p} \nn\\
&\ph+ C_{d,p,q}\int_0^t (\sigma(t-\tau))^{-\frac12-\frac{d}{2}\left(\frac{1}{p}-\frac{1}{q}\right)} \|\mu^\tau\|_{L^{p_1}} \|\M\nabla\g\ast(\mu^\tau-1)\|_{L^{p_2}} d\tau,
\end{align}
where $\frac{d}{2}\left(\frac1p-\frac1q\right)<\frac12$ and $\frac{1}{p_1}+\frac{1}{p_2} = \frac1p$. Using \cref{cor:CLdcy} on $\|\mu^\tau\|_{L^{p_1}}$ and Young's inequality/boundedness of Riesz transforms (assuming $p_2<\infty$) on $\|\M\nabla\g\ast(\mu^\tau-1)\|_{L^{p_2}}$, we find
\begin{align}
\int_0^t (\sigma(t-\tau))^{-\frac12-\frac{d}{2}\left(\frac{1}{p}-\frac{1}{q}\right)} \|\mu^\tau\|_{L^{p_1}} \|\M\nabla\g\ast(\mu^\tau-1)\|_{L^{p_2}} d\tau \nn\\
\lesssim \int_0^t (\sigma(t-\tau))^{-\frac12-\frac{d}{2}\left(\frac{1}{p}-\frac{1}{q}\right)}\min(\sigma\tau,1)^{-\frac{d}{2}\left(\frac{1}{r_1}-\frac{1}{p_1}\right)}\|\mu^0\|_{L^{r_1}}\sup_{0<\tau\leq t}\|\mu^\tau-1\|_{L^{p_2}}d\tau,
\end{align}
where $r_1$ is chosen so that $\frac{d}{2}\left(\frac1{r_1}-\frac1p\right) \leq \frac12$. Choose $p_1=\infty$ and $p_2=p$. If $\sigma t\leq 1$, then by rescaling time, the integral in the last line becomes
\begin{equation}
\sigma^{-1}(\sigma t)^{-\frac{d}2\left(\frac1p-\frac1q\right)}\|\mu^0\|_{L^d}\sup_{0<\tau\leq t}\|\mu^\tau-1\|_{L^p}\int_0^1 (1-\tau)^{-\frac12-\frac{d}{2}\left(\frac1p-\frac1q\right)} \tau^{-\frac12}d\tau,
\end{equation}
which implies that
\begin{align}\label{eq:mu1LqLp}
\|\mu^t-1\|_{L^q} \leq C_{p,q,d}\sigma^{-1}(\sigma t)^{-\frac{d}2\left(\frac1p-\frac1q\right)}\|\mu^0\|_{L^d}\sup_{0<\tau\leq t}\|\mu^\tau-1\|_{L^p}.
\end{align}
Using time translation and an iteration argument, the $L^p$ norm in the right-hand side can be reduced to an $L^1$ norm, to which we can applying \cref{lem:entFEdcay}.

More precisely, suppose that $q=\infty$. Fix a time $t_0$, and fix a step size ${\sigma\ka}\leq \min(\frac{1}{2(d+1)}, \frac{{ \sigma t_0}}{2(d+1)})$. We choose a sequence of exponents
\begin{align}
\infty=p_0 > p_1>\cdots > p_{d} > p_{d+1}=1,
\end{align}
such that $\frac{d}{2}\left(\frac{1}{p_{i+1}}-\frac{1}{p_i}\right) < \frac12$. By translation the initial time, we observe from \eqref{eq:mu1LqLp} that for any $t_0-i\ka \leq t\leq t_0$,
\begin{align}
\|\mu^t-1\|_{L^{p_i}} \leq C_d \sigma^{-1}(\sigma\kappa)^{-\frac{d}{2}\left(\frac1{p_{i+1}}-\frac{1}{p_i}\right)} \|\mu^{t-(d+1)\ka}\|_{L^d}\sup_{t-(i+1)\ka< \tau\leq t} \|\mu^\tau-1\|_{L^{p_{i+1}}}.
\end{align}
Implicitly, we have used that $\|\mu^\tau\|_{L^d}$ is nonincreasing. This implies
\begin{align}
\|\mu^t-1\|_{L^\infty} \leq C_d \sigma^{-d-1} (\sigma\ka)^{-\frac{d}{2}} \|\mu^{t-(d+1)\ka}\|_{L^d}^{d+1}\sup_{t-(d+1)\ka< \tau\leq t} \|\mu^\tau-1\|_{L^{1}}.
\end{align}
Applying \cref{lem:CLdcy} to $\|\mu^{t-(d+1)\ka}\|_{L^d}$ to go from $L^d$ to $L^1$ and applying \cref{lem:entFEdcay} to $\|\mu^\tau-1\|_{L^1}$, the preceding right-hand side is $\leq$
\begin{align}
C_d \sigma^{-d-\frac32}(\min(\sigma t,1))^{-\frac{(d^2+d-1)}{2}} e^{-c_d\sigma t}\sqrt{\Fc_{\sigma}(\mu^0)}.
\end{align}
for constants $C_d,c_d>0$ depending on $d$. Note we have implicitly used that $\|\mu^0\|_{L^1}=1$. This now completes the proof.
\end{proof}

\section{Derivative decay estimates for the mean-field equation}\label{sec:Rlx}
In this section, we prove an exponential rate of decay as $t\rightarrow\infty$ for the $L^q$ norms of derivatives (of arbitrarily large order) of solutions to equation \eqref{eq:lim}, which we know are global by \cref{prop:gwp}. In particular, we show that solutions are smooth for $t>0$. We assume throughout this section that $\int_{\T^d}\mu^0 = 1$ and that $\mu^0\geq 0$ if $\M=-\I$. 

\begin{prop}\label{prop:LpLq}
Let $\mu^t$ be a solution to \eqref{eq:lim} with $\int_{\T^d}\mu^0 = 1$. If $\M=-\I$, further assume that $\mu^0\geq 0$. Let $\al >0$, $\ep>0$, and $n\in\N$.

\textbullet \ When $\max(0,d-2) \le s\le d-1$, there exist constants $C, C_\ep>0$, for any $\ep>0$, and functions $\W_{n,q}, \W_{\al,q}: [0,\infty)^4\rightarrow [0,\infty)$, continuous, nondecreasing, and polynomial in their arguments, such that for every $t>0$,\footnote{The notation $\indic_{a\wedge b}$ denotes the indicator function which is one if both the conditions $a$ and $b$ hold, and zero otherwise.}
\begin{equation}\label{eq:propnabals<}
\|\Dm^{ \al}\mu^t\|_{L^q} \le \W_{\al,q}(\|\mu^0\|_{L^{\infty}},\sigma^{-1}, \|\mu^0-1\|_{L^q}, \Fc_{\sigma}(\mu^0)) (\sigma t)^{-\frac{\al}{2}}\paren*{1 + C_\ep (\sigma t)^{-\ep}\indic_{s=d-1 \wedge q=1}}e^{-C\sigma t}
\end{equation}
and
 \begin{equation}\label{eq:propnabns<}
 \|\nabla^{\otimes n}\mu^t\|_{L^q} \le \W_{n,q}(\|\mu^0\|_{L^{\infty}},\sigma^{-1}, \|\mu^0-1\|_{L^q}, \Fc_{\sigma}(\mu^0)) (\sigma t)^{-\frac{n}{2}}\|\mu^0\|_{L^q}\paren*{1 + C_\ep (\sigma t)^{-\ep}\indic_{s=d-1 \wedge q=1}}e^{-C\sigma t},
\end{equation}

\textbullet \ When $d-1<s<d$, there exist a constant $C>0$ and functions $\W_{\al,q},\W_{n,q}: [0,\infty)^5\rightarrow [0,\infty)$, which are continuous, nondecreasing, and polynomial in their arguments, such that for every $t>0$,
\begin{equation}\label{eq:propnabals>}
 \||\nabla|^{\al}\mu^t\|_{L^q} \le \W_{\al,q}(\|\mu^0\|_{L^\infty},\|\mu^0\|_{\dot{H}^{\la_2}},\sigma^{-1}, \|\mu^0-1\|_{L^{\frac{2}{d-s}}}, \Fc_{\sigma}(\mu^0)) (\sigma t)^{-\frac{\al}{2}}\paren*{1+(\sigma t)^{-\ep}\indic_{q=\infty}} e^{-C\sigma t}
\end{equation}
and
 \begin{equation}\label{eq:propnabns>}
  \|\nabla^{\otimes n}\mu^t\|_{L^q} \le \W_{n,q}(\|\mu^0\|_{L^\infty},\|\mu^0\|_{\dot{H}^{\la_2}},\sigma^{-1}, \|\mu^0-1\|_{L^{\frac{2}{d-s}}}, \Fc_{\sigma}(\mu^0)) (\sigma t)^{-\frac{n}{2}}\paren*{1+(\sigma t)^{-\ep}\indic_{q=\infty}} e^{-C\sigma t},
\end{equation}
where $\la_2 \coloneqq 1+s-d$. 
\end{prop}

\begin{remark}
The constants and functions in the statement of the proposition additionally may depend on $d,s,\M$. One may extract a more explicit dependence of $\W_{\al,q}, \W_{n,q}$ on their arguments from the proof of \cref{prop:LpLq}; but we do not find it enlightening and so do not present it. We only remark that $\W_{\al,q}, \W_{n,q}$ do not depend on the argument $\Fc_{\sigma}(\mu^0)$ if $\M$ is antisymmetric (conservative case).
\end{remark}

\begin{remark}
The exponent $\la_2$ is chosen so that $\|\nabla\g\ast\mu\|_{L^2} \lesssim \|\mu\|_{\dot{H}^{\la_2}}$. See \cref{rem:intL2} below for further motivation.
\end{remark}

\begin{remark}
By using the time translation trick, one can apply \cref{prop:LpLq}, going from $\tau=t$ to $\tau=t_0:=\min(\frac{t}{2},t-\frac{1}{2\sigma})$, leading to the argument $\|\mu^{t_0}\|_{L^\infty}$ in $\W_{\al,q}, \W_{n,q}$. One can then use \cref{cor:CLdcy} to eliminate the norm $\|\mu^{t_0}\|_{L^\infty}$ at the cost of additional factors of $(\sigma t)^{-1}$. Since this does not help us---and we know that $\mu^t\in L^\infty$ for any $t>0$ automatically---we have chosen to keep the $L^\infty$ dependence for simplicity. 
\end{remark}

So as to make the presentation easier to digest, we break the proof of \cref{prop:LpLq} into a series of lemmas, which are proved in the upcoming two subsections. In \cref{ssec:Rlxs<}, we treat the range $d-2\leq s\leq d-1$, showing \eqref{eq:propnabals<}, \eqref{eq:propnabns<}. Then in \cref{ssec:Rlxs>}, we treat the harder, remaining range $d-1<s<d$, showing \eqref{eq:propnabals>}, \eqref{eq:propnabns>}. This then completes the proof of \cref{prop:LpLq} and, together with the results of \Cref{ssec:WPfe,ssec:WPLpLq}, establishes assertion \eqref{eq:mainrlxbnd} from \cref{thm:mainWP}.

\subsection{The case $d-2\leq s\leq d-1$}\label{ssec:Rlxs<}
We begin with the temporal decay estimates for the $L^p$ norms of the derivatives of $\mu^t$ (note $\nabla^{\otimes n}\mu^t$ has zero average for $n\geq 1$ and similarly for $\Dm^\al\mu^t$ if $\al>0$) in the easier case $d-2\leq s\leq d-1$. The first result is a short-time gain of regularity.

\begin{lemma}\label{lem:derdcys<}
Let $d\geq 2$ and $d-2\leq s\leq d-1$. For each $n\in\N$ and $1\leq p\leq q\leq\infty$, there exists a function $\W_{n,p,q}: [0,\infty)^4 \rightarrow [0,\infty)$, continuous, nondecreasing in its arguments, and vanishing if any of its arguments is zero such that following holds. If $\mu$ is a smooth solution to equation \eqref{eq:lim} on $[0,T]$, then
\begin{multline}
\forall t\in (0,\sigma^{-1}], \qquad \|\nabla^{\otimes n} \mu^t\|_{L^q} \leq \W_{n,p,q}(\|\mu^0\|_{L^{\infty}},\sigma^{-1}, \|\mu^0-1\|_{L^p},\Fc_{\sigma}(\mu^0)) \\
\times\left(\|\mu^0-1\|_{L^p} + \left( \|\mu^0\|_{L^\infty}^{1-\frac1p}\left(\Fc_{\sigma}(\mu^0)/\sigma\right)^{\frac{1}{2p}}\indic_{p<\infty} + C_{\epsilon}\|\mu^0\|_{L^\infty}^{1-\epsilon}\left(\Fc_{\sigma}(\mu^0)/\sigma\right)^{\frac{\epsilon}{2}}\indic_{p=\infty}\right) \indic_{\M=-\I} \right)\\
\times(\sigma t)^{- \frac{n}{2}-\frac{d}{2}\left(\frac1p-\frac1q\right)}\paren*{1 + C_\ep (\sigma t)^{-\ep}\indic_{s=d-1 \wedge q=1}},
\end{multline}
where $C>0$ depends on $n,d,p,q$, $C_\epsilon>0$ depends on $\epsilon \in (0,d^{-1})$, $C_{\ep}>0$ depends on $\ep>0$, which can be made arbitrarily small. The function $\W_{n,p,q}$ additionally depends on $d,s,\M$. Moreover, it is independent of $\Fc_{\sigma}(\mu^0)$ if $\M$ is antisymmetric and is independent of $\|\mu^0-1\|_{L^p}$ if $\M=-\I$.
\end{lemma}

\begin{proof}
Our starting point is the following identity, which follows from commutativity of Fourier multipliers,
\begin{equation}\label{eq:mildLeibniz}
\p_{\al}\mu^t = e^{t\sigma\D}\p_{\al}\mu^0 -\int^t_0 e^{\sigma(t-\tau)\D} \div \p_\al \left(\mu^\tau \M\nabla\g\ast\mu^\tau  \right)d\tau,
\end{equation}
where $\al = (\al_1,\ldots,\al_d) \in \N_0^d$ is a multi-index of order $|\al| = 1$. The general case $|\al| \geq 1$ will be handled by induction. As the heat kernel is singular at $\tau=t$, we divide the integration over $[0,t]$ into $[0,(1-\ep)t]$ and $[(1-\ep)t,t]$, for some $\ep \in (0,1)$ to be determined. Applying triangle and Minkowski's inequalities to the right hand side of \eqref{eq:mildLeibniz} leads us to 
\begin{multline}\label{eq:defJ123}
    \|\p_\al\mu^t\|_{L^q} \le \|e^{t\sigma\D}\p_\al\mu^0\|_{L^q} + \int_0^{t(1-\ep)} \left\| e^{\sigma(t-\tau)\D} \div \p_\al \left(\mu^\tau \M\nabla\g\ast\mu^\tau  \right)\right\|_{L^q} d\tau \\ 
    + \int_{t(1-\ep)}^t \left\| e^{\sigma(t-\tau)\D} \div \p_\al \left(\mu^\tau \M\nabla\g\ast\mu^\tau  \right)\right\|_{L^q} d\tau.
\end{multline}
We respectively denote by $J_1(t), J_2(t), J_3(t)$ the three terms in the right hand side of the previous inequality and proceed to estimate each of them individually. 

$J_1(t)$ is straightforward consequence of heat kernel estimate \eqref{eq:nabnK} and Young's inequality:
\begin{equation}
\label{eq:J1estim}
    J_1(t) \lesssim \paren*{\min(\sigma t,1)}^{-\frac{d}{2}\left( \frac{1}{p}-\frac{1}{q}\right) - \frac{1}{2}}e^{-C\sigma t}\|\mu^0-1\|_{L^p},
\end{equation}
for any $1\leq p\leq q\leq\infty$.

Consider now $J_2(t)$. By \eqref{eq:mDhk} and H\"older's inequalities, we have for any $1\le p \le q\le\infty$,
\begin{align}
&\left\| \p_\al e^{\sigma(t-\tau)\D} \div \left(\mu^\tau \M\nabla\g\ast\mu^\tau  \right)\right\|_{L^q} \nn\\
&\lesssim e^{-C\sigma(t-\tau)}	\min\paren*{\sigma(t-\tau),1}^{-\frac{d}{2}\left( \frac{1}{p}-\frac{1}{q}\right) - 1} \|\mu^\tau \M\nabla\g\ast\mu^\tau \|_{L^p} \nn \\
&\lesssim e^{-C\sigma(t-\tau)}	\min\paren*{\sigma(t-\tau),1}^{-\frac{d}{2}\left( \frac{1}{p}-\frac{1}{q}\right) - 1}  \|\mu^\tau\|_{L^\infty} \| \M\nabla\g\ast\mu^\tau \|_{L^{p}} \nn \\
&\lesssim e^{-C\sigma(t-\tau)}	\min\paren*{\sigma(t-\tau),1}^{-\frac{d}{2}\left( \frac{1}{p}-\frac{1}{q}\right) - 1} \paren*{\min(\sigma\tau,1)}^{-\frac12}\|\mu^0\|_{L^d} \| \M\nabla\g\ast\mu^\tau \|_{L^p}, \label{eq:J2t1}
\end{align}
where we applied \cref{cor:CLdcy} to $\|\mu^\tau\|_{L^\infty}$ to obtain the last line. If $p=\infty$ (and so $q=\infty$ as well) and $s=d-1$, the Riesz transform $\nabla\g\ast$ is not bounded on $L^p$ and so we will be out of luck in trying to estimate $\|\M\nabla\g\ast\mu^t\|_{L^p}$ in terms of $\mu^t$. Instead, we modify the preceding argument to obtain that for any $d<r<\infty$,
\begin{multline}
\left\| \p_\al e^{\sigma(t-\tau)\D} \div \left(\mu^\tau \M\nabla\g\ast\mu^\tau  \right)\right\|_{L^\infty} \\
\lesssim e^{-C\sigma(t-\tau)}\min\paren*{\sigma(t-\tau),1}^{-\frac{d}{2r}- 1} \paren*{\min(\sigma\tau,1)}^{-\hal+ \frac{d}{2r}} \|\mu^0\|_{L^\frac{rd}{r-d}} \| \M\nabla\g\ast\mu^\tau \|_{L^r}. \label{eq:J2t2}
\end{multline}
Combining the estimates \eqref{eq:J2t1} and \eqref{eq:J2t2}, we have shown that for any $1\le p\le q\le \infty$ and $d<r<\infty$, 
\begin{multline}
\left\| \p_\al e^{\sigma(t-\tau)\D} \div \left(\mu^\tau \M\nabla\g\ast\mu^\tau  \right)\right\|_{L^q} \lesssim \min\paren*{\sigma(t-\tau),1}^{-\frac{d}{2}\left( \frac{1}{p}-\frac{1}{q}\right) - 1}e^{-C\sigma(t-\tau)}\paren*{\min(\sigma\tau,1)}^{-\hal} \\
    \times \Big(\|\mu^0\|_{L^d}   \| \M\nabla\g\ast\mu^\tau \|_{L^p}\indic_{\substack{s<d-1 \\ s=d-1 \wedge p<\infty}}\\
     +\min\paren*{\sigma(t-\tau),1}^{-\frac{d}{2r}}\paren*{\min(\sigma\tau,1)}^{ \frac{d}{2r}} \|\mu^0\|_{L^\frac{rd}{r-d}}\| \M\nabla\g\ast\mu^\tau \|_{L^r} \indic_{s=d-1 \wedge p=\infty}\Big) .
\end{multline}
We then use either \eqref{eq:CLdcycon} from \cref{cor:CLdcy}, if $\M$ is antisymmetric, or \eqref{eq:entFEdcaygf} from \cref{lem:entFEdcay}, if $\M=-\I$, to bound for $1\leq  r<\infty$ (same for $r$ replaced by $p$)
\begin{align}
\| \M\nabla\g\ast\mu^\tau \|_{L^r} \lesssim \|\mu^\tau-1\|_{L^r} \lesssim e^{-C'\sigma \tau}\|\mu^0-1\|_{L^r} \indic_{\M \ \text{a.s.}} \\
+ (1+\|\mu^0\|_{L^\infty})^{1-\frac1r}\paren*{e^{-C''\sigma \tau}\sqrt{\Fc_{\sigma}(\mu^0)/\sigma}}^{\frac1r}\indic_{\M=-\I}.
\end{align}
Since $\tau\leq t$ and $\sigma t\leq 1$ by assumption, we can drop the $\min(\cdot)$ and exponential factors above. It now follows that
\begin{multline}
    \left\| \p_\al e^{\sigma(t-\tau)\D} \div \left(\mu^\tau \M\nabla\g\ast\mu^\tau  \right)\right\|_{L^q} \lesssim \paren*{\sigma(t-\tau)}^{-\frac{d}{2}\left( \frac{1}{p}-\frac{1}{q}\right) -1}(\sigma\tau)^{-\frac12} \\
    \times \Bigg(\|\mu^0\|_{L^d}\Big(\|\mu^0-1\|_{L^p} \indic_{\M \ \text{a.s.}} + (1+\|\mu^0\|_{L^\infty})^{1-\frac1p}\paren*{\sqrt{\Fc_{\sigma}(\mu^0)/\sigma}}^{\frac1p}\indic_{\M=-\I} \Big)\indic_{p<\infty}\\
     +\|\mu^0\|_{L^{d+}}\Big(\|\mu^0-1\|_{L^r} \indic_{\M \ \text{a.s.}} + (1+\|\mu^0\|_{L^\infty})^{1-\frac1r}\paren*{ \sqrt{\Fc_{\sigma}(\mu^0)/\sigma}}^{\frac1r}\indic_{\M=-\I} \Big)\paren*{\frac{\tau}{(t-\tau)}}^{\frac{d}{2r}} \indic_{ p=\infty} \Bigg).
\end{multline}
Adjusting $C,C',\frac{C''}{r},\frac{C''}{p}$ if necessary, we may assume without loss of generality that $C=C'=\frac{C''}{r}=\frac{C''}{p}$. Recalling the definition of $J_2(t)$ from \eqref{eq:defJ123} and using the dilation invariance of Lebesgue measure, we arrive at
\begin{multline}
\label{eq:J2estim}
J_2(t) \lesssim \frac{A_\ep}{\sigma} \|\mu^0\|_{L^{d^*}} (\sigma t)^{-\frac{d}{2} \left( \frac{1}{p}-\frac{1}{q}\right) - \frac{1}{2}}\Bigg(\Big(\|\mu^0-1\|_{L^p}\indic_{\M \ \text{a.s.}} + \|\mu^0\|_{L^\infty}^{1-\frac1p} \left(\Fc_{\sigma}(\mu^0)/\sigma\right)^{\frac{1}{2p}}\indic_{\M=-\I}\Big)\indic_{\substack{p<\infty}} \\
+ \Big(\|\mu^t-1\|_{L^r}\indic_{\M \ \text{a.s.}} + \|\mu^0\|_{L^\infty}^{1-\frac1r}\left(\Fc_{\sigma}(\mu^0)/\sigma\right)^{\frac{1}{2r}}\indic_{\M=-\I}\Big)\indic_{p =\infty} \Bigg),
\end{multline}
where $d^* \coloneqq d^+\indic_{s=d-1\wedge p=\infty} + d(1-\indic_{s=d-1\wedge p=\infty})$, $d<r<\infty$ is arbitrary, and
\begin{equation}\label{eq:Aepdefs<bc}
    A_\ep \coloneqq \int_0^{1-\ep} (1-\tau)^{-\frac{d}{2}\left( \frac{1}{p}-\frac{1}{q}\right) -1}  \tau^{-\hal} \left( \indic_{\substack{ p<\infty}} +\paren*{\frac{\tau}{(1-\tau)}}^{\frac{d}{2r}}  \indic_{p=\infty} \right)d\tau.
\end{equation}

Finally, for $J_3(t)$, we have by product rule and triangle inequality,
\begin{multline}
    \left\| \p_\al e^{\sigma(t-\tau)\D} \div \left(\mu^\tau \M\nabla\g\ast\mu^\tau  \right)\right\|_{L^q} \\
    \lesssim  e^{-C\sigma(t-\tau)}\min\paren*{\sigma(t-\tau),1}^{-\hal}\paren*{\|\p_\al\mu^\tau \M\nabla\g\ast\mu^\tau\|_{L^q} + \|\mu^\tau \M\nabla\g\ast\p_\al\mu^\tau  \|_{L^{q}}}.
\end{multline}
If $s<d-1$, then since $\nabla\g\in L^1$, we obtain from application of estimate \eqref{eq:CLdcyhypgf} of \cref{cor:CLdcy} to $\|\mu^\tau\|_{L^\infty}$,
\begin{align}
    \left\| \p_\al e^{\sigma(t-\tau)\D} \div \left(\mu^\tau \M\nabla\g\ast\mu^\tau  \right)\right\|_{L^q} 
    &\lesssim  e^{-C\sigma(t-\tau)}\min\paren*{\sigma(t-\tau),1}^{-\hal}\|\mu^\tau\|_{L^\infty} \|\p_\al\mu^\tau \|_{L^q} \nn \\
    &\lesssim e^{-C\sigma(t-\tau)}\min\paren*{\sigma(t-\tau),1}^{-\hal}\min(\sigma\tau,1)^{-\hal} \|\mu^0\|_{L^d} \|\p_\al\mu^\tau \|_{L^q}.\label{eq:J3s<}
\end{align}
If $s=d-1$ and $q>1$, we modify the argument (to account for $\nabla\g\notin L^1$) to obtain, for any $1< r< q\le \infty$, 
\begin{multline}
    \left\| \p_\al e^{\sigma(t-\tau)\D} \div \left(\mu^\tau \M\nabla\g\ast\mu^\tau  \right)\right\|_{L^q} \\
    \lesssim  e^{-C\sigma(t-\tau)} \min(\sigma(t-\tau),1)^{-\hal -\frac{d}{2}\left(\frac{1}{r}-\frac{1}{q} \right)} \paren*{\|\mu^\tau \M\nabla\g\ast\p_\al\mu^\tau\|_{L^r} + \|\p_\al\mu^\tau \M\nabla\g\ast\mu^\tau\|_{L^r}} .
\end{multline}
We have by H\"older's and Young's inequalities,
\begin{align}
    \| \mu^\tau \M\nabla\g\ast\p_\al\mu^\tau  \|_{L^r} &\lesssim \|  \M\nabla\g\ast\p_\al\mu^\tau \|_{L^r} \| \mu^\tau \|_{L^\infty} \nn \\
    &\lesssim \|\p_\al\mu^\tau \|_{L^r}\| \mu^\tau \|_{L^\infty} \nn \\
    &\leq \|\p_\al\mu^\tau \|_{L^q}\| \mu^\tau \|_{L^\infty}. \label{eq:J3t1}
\end{align} 
Again, by H\"older's inequality plus the boundedness of the Riesz transform, 
\begin{align}
    \| \p_\al\mu^\tau \M\nabla\g\ast\mu^\tau  \|_{L^r} &\le \| \p_\al\mu^\tau\|_{L^q} \|\M\nabla\g\ast\mu^\tau  \|_{L^{\frac{rq}{q-r}}} \nn \\
    &\lesssim \| \p_\al\mu^\tau\|_{L^q} \|\mu^\tau -1\|_{L^{\frac{rq}{q-r}}}.\label{eq:J3t2}
\end{align}
Taking $r$ close enough to $q$ so that $\frac{d}{r}-\frac{d}{q}<1$, it follows from \eqref{eq:J3t1}, \eqref{eq:J3t2} and application of \eqref{eq:CLdcyhypgf} to $\|\mu^\tau\|_{L^\infty}$,
\begin{multline}
    \left\| \p_\al e^{\sigma(t-\tau)\D} \div \left(\mu^\tau \M\nabla\g\ast\mu^\tau  \right)\right\|_{L^q} \\
    \lesssim  e^{-C\sigma(t-\tau)}\min(\sigma(t-\tau),1)^{-\hal -\frac{d}{2}\left(\frac{1}{r}-\frac{1}{q} \right)} \min(\sigma\tau,1)^{-\hal + \frac{d}{2}\left(\frac{1}{r}-\frac{1}{q} \right)} \|\mu^0\|_{L^\frac{d}{1-d(\frac{1}{r}-\frac{1}{q})}} \| \p_\al\mu^\tau\|_{L^q}.\label{eq:J3s>}
\end{multline}

Combining the estimates \eqref{eq:J3s<} and \eqref{eq:J3s>} and dropping the $\min(\cdot)$ and exponential factors using the assumption $\sigma t\leq 1$, we arrive at
\begin{multline}\label{eq:J3estim}
    J_3(t) \lesssim {\|\mu^0\|_{L^{d^*}}} \int_{t(1-\ep)}^t  (\sigma(t-\tau))^{-\hal} (\sigma\tau)^{-\hal} \|\p_\al\mu^\tau \|_{L^q} \\
    \times \left( \indic_{s<d-1} + \paren*{\frac{\tau}{(t-\tau)}}^{0+} \indic_{s=d-1, q>1}\right) d\tau.
\end{multline}
Combining the estimates \eqref{eq:J1estim}, \eqref{eq:J2estim}, \eqref{eq:J3estim}, we obtain
\begin{multline}\label{eq:prelemma}
	 \| \p_\al\mu^t \|_{L^q} \lesssim \|\mu^0-1\|_{L^p}(\sigma t)^{-\frac{d}{2}\left(\frac1p-\frac1q\right)-\frac12} \\
	 +\frac{A_\ep}{\sigma} \|\mu^0\|_{L^{d^*}} (\sigma t)^{-\frac{d}{2} \left( \frac{1}{p}-\frac{1}{q}\right) - \frac{1}{2}}\Bigg(\Big(\|\mu^0-1\|_{L^p}\indic_{\M \ \text{a.s.}}	  + \|\mu^0\|_{L^\infty}^{1-\frac1p}\left(\Fc_{\sigma}(\mu^0)/\sigma\right)^{\frac{1}{2p}}\indic_{\M=-\I}\Big)\indic_{\substack{p<\infty}}\\
	  + \Big(\|\mu^0-1\|_{L^r}\indic_{\M \ \text{a.s.}} +\|\mu^0\|_{L^\infty}^{1-\frac1r}\left(\Fc_{\sigma}(\mu^0)/\sigma\right)^{\frac{1}{2r}}\indic_{\M=-\I}\Big)\indic_{p =\infty} \Bigg) \\
	 + {\|\mu^0\|_{L^{d^*}}}\int_{t(1-\ep)}^t  (\sigma(t-\tau))^{-\hal} (\sigma\tau)^{-\hal} \|\p_\al\mu^\tau \|_{L^q}  \left( \indic_{s<d-1} + \paren*{\frac{\tau}{(t-\tau)}}^{0+}\indic_{s=d-1, q>1}\right) d\tau.
\end{multline}

To close the estimate for $\|\p_\al\mu^t\|_{L^q}$, we define the function (for $0<t\leq \sigma^{-1}$)
\begin{equation}
\phi(t)\coloneqq \sup_{0<\tau\le t} (\sigma\tau)^{\frac{d}{2} \left( \frac{1}{p}-\frac{1}{q}\right) + \frac{1}{2}} \| \p_\al\mu^\tau \|_{L^q}.
\end{equation}
Using this notation, we rearrange \eqref{eq:prelemma}, using the dilation invariance of Lebesgue measure, to obtain the inequality
\begin{multline}\label{eq:taddfac}
\phi(t) \le C\|\mu^0-1\|_{L^p} + \frac{CA_\ep\|\mu^0\|_{L^{d^*}}}{\sigma}  \Bigg(\Big(\|\mu^0-1\|_{L^p}\indic_{\M \ \text{a.s.}}	  + \|\mu^0\|_{L^\infty}^{1-\frac1p}\left(\Fc_{\sigma}(\mu^0)/\sigma\right)^{\frac{1}{2p}}\indic_{\M=-\I}\Big)\indic_{\substack{p<\infty}}\\
	  + \Big(\|\mu^0-1\|_{L^r}\indic_{\M \ \text{a.s.}} + \|\mu^0\|_{L^\infty}^{1-\frac1r}\left(\Fc_{\sigma}(\mu^0)/\sigma\right)^{\frac{1}{2r}}\indic_{\M=-\I}\Big)\indic_{p =\infty} \Bigg)  + \frac{C B_\ep \|\mu^0\|_{L^{d^*}}}{\sigma}\phi(t),
\end{multline}
where $C>0$ depends only on $d,s,p,q$ and
\begin{equation}
B_\ep \coloneqq \int_{1-\ep}^1 (1-\tau)^{-\hal} \tau^{-\frac{d}{2} \left(\frac{1}{p}-\frac{1}{q}\right)-1}  \left( \indic_{s<d-1} + \paren*{\frac{\tau}{1-\tau}}^{0+} \indic_{s=d-1, q>1}\right) d\tau.
\end{equation}
The fact that we did not pick up any factors of $t$ in \eqref{eq:taddfac} precisely explains our choice of the exponents in the factors $(t-\tau)$ and $\tau$ above. Since the integral in the definition of $B_\ep$ decreases monotonically to zero as $\ep\rightarrow 1^{-}$, we may choose $\ep$ sufficiently small so that $C \|\mu^0\|_{L^{d^*}} B_\ep \le \frac\sigma2$. 
Thus, 
\begin{multline}\label{eq:estimstep2}
    \| \p_\al\mu^t\|_{L^q} \le \W_{1,p,q}(\|\mu^0\|_{L^{\infty}},\sigma^{-1},\|\mu^0-1\|_{L^p}, \Fc_{\sigma}(\mu^0)) (\sigma t)^{-\frac{d}{2} \left( \frac{1}{p}-\frac{1}{q}\right) - \frac{1}{2}} \\
    \times \left(\|\mu^0-1\|_{L^p} + \left( \|\mu^0\|_{L^\infty}^{1-\frac1p}\left(\Fc_{\sigma}(\mu^0)/\sigma\right)^{\frac{1}{2p}}\indic_{p<\infty} + C_{\epsilon}\|\mu^0\|_{L^\infty}^{1-\epsilon}\left(\Fc_{\sigma}(\mu^0)/\sigma\right)^{\frac{\epsilon}{2}}\indic_{p=\infty}\right) \indic_{\M=-\I} \right),
\end{multline}
for any $\epsilon \in (0,d^{-1})$, where $\W_{1,p,q}$ is a continuous, nondecreasing, polynomial function of its arguments. Furthermore, $\W_{1,p,q}$ does not depend on $\Fc_{\sigma}(\mu^0)$ if $\M$ is antisymmetric.

\medskip 

Let us now bootstrap from the case $|\alpha|=1$ to the general case $n=|\al|\geq 1$. As our induction hypothesis, assume that
\begin{multline}\label{eq:s<ih}
    \forall |\be|\le n-1, \ t\in (0,\sigma^{-1}], \  1\le p\le q\le\infty, \\
     \| \p_\be\mu^t\|_{L^q} \le \W_{|\beta|,p,q}(\|\mu^0\|_{L^{\infty}},\sigma^{-1},\|\mu^0-1\|_{L^p},\Fc_{\sigma}(\mu^0)) (\sigma t)^{-\frac{d}{2} \left( \frac{1}{p}-\frac{1}{q}\right) - \frac{|\be|}{2}} \\
     \times \left(\|\mu^0-1\|_{L^p} + \left( \|\mu^0\|_{L^\infty}^{1-\frac1p}\left(\Fc_{\sigma}(\mu^0)/\sigma\right)^{\frac{1}{2p}}\indic_{p<\infty} + C_{\epsilon}\|\mu^0\|_{L^\infty}^{1-\epsilon}\left(\Fc_{\sigma}(\mu^0)/\sigma\right)^{\frac{\epsilon}{2}}\indic_{p=\infty}\right) \indic_{\M=-\I} \right),
\end{multline}
where $\ep \in (0,d^{-1})$ and $\W_{|\beta|,p,q}$ is a continuous, nondecreasing, polynomial function of its arguments. Analogous to \eqref{eq:defJ123}, we have
\begin{multline}
    \|\p_\al\mu^t\|_{L^q} \le \|e^{\sigma t\D}\p_\al\mu^0\|_{L^q} + \int_0^{t(1-\ep)} \left\| e^{\sigma (t-\tau)\D} \div \p_\al \left(\mu^\tau \M\nabla\g\ast\mu^\tau  \right)\right\|_{L^q} d\tau \\ 
    + \int_{t(1-\ep)}^t \left\| e^{\sigma(t-\tau)\D} \div \p_\al \left(\mu^\tau \M\nabla\g\ast\mu^\tau  \right)\right\|_{L^q} d\tau.
\end{multline}
Repeating the arguments for $J_1(t)$ and $J_2(t)$ above, we have 
\begin{align}
    J_1(t) \lesssim \paren*{\sigma t}^{-\frac{d}{2}\left( \frac{1}{p}-\frac{1}{q}\right) - \frac{n}{2}}\|\mu^0-1\|_{L^p},\label{eq:J1estimboot}
\end{align}
\begin{multline}
J_2(t) \lesssim  \frac{A_\ep\|\mu^0\|_{L^{d^*}}}{\sigma}  (\sigma t)^{-\frac{d}{2} \left( \frac{1}{p}-\frac{1}{q}\right) - \frac{n}{2}}\\
\times\Bigg(\Big(\|\mu^0-1\|_{L^p}\indic_{\M \ \text{a.s.}} + \|\mu^0\|_{L^\infty}^{1-\frac1p}\left(\Fc_{\sigma}(\mu^0)/\sigma\right)^{\frac{1}{2p}}\indic_{\M=-\I}\Big)\indic_{\substack{p<\infty}} \\
+ \Big(\|\mu^0-1\|_{L^r}\indic_{\M \ \text{a.s.}} + \|\mu^0\|_{L^\infty}^{1-\frac1r}\left(\Fc_{\sigma}(\mu^0)/\sigma\right)^{\frac{1}{2r}}\indic_{\M=-\I}\Big)\indic_{p =\infty} \Bigg) , \label{eq:J2estimboot}
\end{multline}
where now
\begin{equation}
A_\ep \coloneqq\int_0^{1-\ep} (1-\tau)^{-\frac{d}{2}\left( \frac{1}{p}-\frac{1}{q}\right) -\frac{|\al|+1}{2}}  \tau^{-\hal} \left( \indic_{\substack{ p<\infty}} +\paren*{\frac{\tau}{1-\tau}}^{0+}  \indic_{p=\infty} \right)d\tau.
\end{equation}
For $J_3$, we apply the Leibniz rule,
\begin{equation}
    \p_\al(\mu\M\nab\g\ast\mu) = \sum_{\be\le\al} {{\al}\choose{\be}} \p_\be\mu \M\nab\g\ast\p_{\al-\be}\mu,
\end{equation}
and note that estimates \eqref{eq:J3s<}, \eqref{eq:J3t1}, \eqref{eq:J3t2} also hold for the $\beta=0,\beta=\alpha$ terms. For the terms with $\beta\notin \{0,\alpha\}$, we use the induction hypothesis \eqref{eq:s<ih}. If $s<d-1$ or $s=d-1$ and $q\notin\{1,\infty\}$,
\begin{align}
    &\int_{t(1-\ep)}^t \left\| e^{\sigma(t-\tau)\D} \div \p_\be\mu^\tau \M\nabla\g\ast\p_{\al-\be}\mu^\tau  \right\|_{L^q} d\tau \nn\\
    &\lesssim \int_{t(1-\ep)}^t (\sigma(t-\tau))^{-\hal} \|\p_\be\mu^\tau\|_{L^\infty}\| \p_{\al-\be}\mu^\tau\|_{L^q} d\tau\nn \\
    &\lesssim \W_{|\beta|,d,\infty }(\|\mu^0\|_{L^\infty},\sigma^{-1}, \|\mu^0-1\|_{L^d},\Fc_{\sigma}(\mu^0)) \W_{n-|\be|,p,q}(\|\mu^0\|_{L^\infty},\sigma^{-1}, \|\mu^0-1\|_{L^p},\Fc_{\sigma}(\mu^0)) \nn\\
    &\ph\qquad\times \left(\|\mu^0-1\|_{L^p} + \left( \|\mu^0\|_{L^\infty}^{1-\frac1p}\left(\Fc_{\sigma}(\mu^0)/\sigma\right)^{\frac{1}{2p}}\indic_{p<\infty} + C_{\epsilon}\|\mu^0\|_{L^\infty}^{1-\epsilon}\left(\Fc_{\sigma}(\mu^0)/\sigma\right)^{\frac{\epsilon}{2}}\indic_{p=\infty}\right) \indic_{\M=-\I} \right) \nn\\
    &\ph\qquad \times \left(\|\mu^0-1\|_{L^p} + \left( \|\mu^0\|_{L^\infty}^{1-\frac1p}\left(\Fc_{\sigma}(\mu^0)/\sigma\right)^{\frac{1}{2p}}\indic_{p<\infty} + C_{\epsilon'}\|\mu^0\|_{L^\infty}^{1-\epsilon'}\left(\Fc_{\sigma}(\mu^0)/\sigma\right)^{\frac{\epsilon'}{2}}\indic_{p=\infty}\right) \indic_{\M=-\I} \right)\nn\\
    &\ph\qquad\times \int_{t(1-\ep)}^t (\sigma(t-\tau))^{-\hal} (\sigma\tau)^{-\frac{d}{2}\left( \frac{1}{p}-\frac{1}{q}\right)-\frac{|\al|+1}{2}} d\tau \nn \\ 
    &\lesssim  \sigma^{-1}\W_{|\beta|,d,\infty }(\|\mu^0\|_{L^\infty},\sigma^{-1},\|\mu^0-1\|_{L^d},\Fc_{\sigma}(\mu^0)) \W_{n-|\be|,p,q}(\|\mu^0\|_{L^\infty},\sigma^{-1},\|\mu^0-1\|_{L^p},\Fc_{\sigma}(\mu^0)) \nn\\
    &\ph\qquad\times \left(\|\mu^0-1\|_{L^p} + \left( \|\mu^0\|_{L^\infty}^{1-\frac1p}\left(\Fc_{\sigma}(\mu^0)/\sigma\right)^{\frac{1}{2p}}\indic_{p<\infty} + C_{\epsilon}\|\mu^0\|_{L^\infty}^{1-\epsilon}\left(\Fc_{\sigma}(\mu^0)/\sigma\right)^{\frac{\epsilon}{2}}\indic_{p=\infty}\right) \indic_{\M=-\I} \right) \nn\\
    &\ph\qquad \times \left(\|\mu^0-1\|_{L^p} + \left( \|\mu^0\|_{L^\infty}^{1-\frac1p}\left(\Fc_{\sigma}(\mu^0)/\sigma\right)^{\frac{1}{2p}}\indic_{p<\infty} + C_{\epsilon'}\|\mu^0\|_{L^\infty}^{1-\epsilon'}\left(\Fc_{\sigma}(\mu^0)/\sigma\right)^{\frac{\epsilon'}{2}}\indic_{p=\infty}\right) \indic_{\M=-\I} \right)\nn\\
    &\ph\qquad\times(\sigma t)^{-\frac{d}{2}\left( \frac{1}{p}-\frac{1}{q}\right)-\frac{(n+1)}{2}} B_\ep,
\end{align}
where $\ep,\ep' \in (0,d^{-1})$ and
\begin{equation}
B_\ep \coloneqq \int_{1-\ep}^1 (1-\tau)^{-\frac{1}{2}}\tau^{-\frac{d}{2}\left(\frac{1}{p}-\frac{1}{q}\right)-\frac{(n+1)}{2}} d\tau.
\end{equation}
If $s=d-1$ and $q=\infty$, then similar to before, we argue
\begin{align}
&\left\| e^{\sigma(t-\tau)\D} \div \left(\p_{\be}\mu^\tau \M\nabla\g\ast\p_{\al-\be}\mu^\tau  \right)\right\|_{L^q} \nn\\
&\lesssim (\sigma(t-\tau))^{-\hal -\frac{d}{2}\left(\frac{1}{r}-\frac{1}{q} \right)} \|\p_{\be}\mu^\tau\|_{L^q} \|\p_{\al-\be}\mu^\tau\|_{L^\infty} \nn\\
    &\lesssim \W_{|\be|,p,q}(\|\mu^0\|_{L^{\infty}},\sigma^{-1},\|\mu^0-1\|_{L^p},\Fc_{\sigma}(\mu^0)) \W_{n-|\be|,d+,\infty}(\|\mu^0\|_{L^{\infty}}, \sigma^{-1},\|\mu^0-1\|_{L^{d+}},\Fc_{\sigma}(\mu^0)) \nn\\
    &\ph\qquad\times \left(\|\mu^0-1\|_{L^p} + \left( \|\mu^0\|_{L^\infty}^{1-\frac1p}\left(\Fc_{\sigma}(\mu^0)/\sigma\right)^{\frac{1}{2p}}\indic_{p<\infty} + C_{\epsilon}\|\mu^0\|_{L^\infty}^{1-\epsilon}\left(\Fc_{\sigma}(\mu^0)/\sigma\right)^{\frac{\epsilon}{2}}\indic_{p=\infty}\right) \indic_{\M=-\I} \right) \nn\\
    &\ph\qquad \times \left(\|\mu^0-1\|_{L^p} + \left( \|\mu^0\|_{L^\infty}^{1-\frac1p}\left(\Fc_{\sigma}(\mu^0)/\sigma\right)^{\frac{1}{2p}}\indic_{p<\infty} + C_{\epsilon'}\|\mu^0\|_{L^\infty}^{1-\epsilon'}\left(\Fc_{\sigma}(\mu^0)/\sigma\right)^{\frac{\epsilon'}{2}}\indic_{p=\infty}\right) \indic_{\M=-\I} \right)\nn\\
    &\ph\qquad\times(\sigma(t-\tau))^{-\hal} (\sigma\tau)^{-\frac{|\al|+1}{2} -\frac{d}{2}\left(\frac{1}{p}-\frac{1}{q} \right)}\paren*{\frac{\tau}{t-\tau}}^{0+}.
\end{align}

Above, we have neglected the case $s=d-1$ and $q=1$, as our arguments do not work. However, we have by H\"older's inequality that, for any $r>1$, 
\begin{align}
    \|\p_\al\mu^t\|_{L^1} &\le \|\p_\al \mu^t\|_{L^r}  \nn\\
       &\lesssim \W_{|\al|,p,r}(\|\mu^0\|_{L^\infty},\sigma^{-1},\|\mu^0-1\|_{L^r},\Fc_{\sigma}(\mu^0))(\sigma t)^{-\frac{d}{2}\left(1-\frac{1}{r}\right) -\frac{|\al|}{2}}  \nn\\
       &\ph\times \left(\|\mu^0-1\|_{L^p} + \left( \|\mu^0\|_{L^\infty}^{1-\frac1p}\left(\Fc_{\sigma}(\mu^0)/\sigma\right)^{\frac{1}{2p}}\indic_{p<\infty} + C_{\epsilon}\|\mu^0\|_{L^\infty}^{1-\epsilon}\left(\Fc_{\sigma}(\mu^0)/\sigma\right)^{\frac{\epsilon}{2}}\indic_{p=\infty}\right) \indic_{\M=-\I} \right)  \nn\\
        &\lesssim \W_{|\al|,p,r}(\|\mu^0\|_{L^\infty},\sigma^{-1},\|\mu^0-1\|_{L^1}^{\frac{1}{r}}\left(1+\|\mu^0\|_{L^\infty}\right)^{1-\frac1r},\Fc_{\sigma}(\mu^0))(\sigma t)^{-\frac{d}{2}\left(1-\frac{1}{r}\right) -\frac{|\al|}{2}}  \nn\\
       &\ph\times \left(\|\mu^0-1\|_{L^p} + \left( \|\mu^0\|_{L^\infty}^{1-\frac1p}\left(\Fc_{\sigma}(\mu^0)/\sigma\right)^{\frac{1}{2p}}\indic_{p<\infty} + C_{\epsilon}\|\mu^0\|_{L^\infty}^{1-\epsilon}\left(\Fc_{\sigma}(\mu^0)/\sigma\right)^{\frac{\epsilon}{2}}\indic_{p=\infty}\right) \indic_{\M=-\I} \right) .
\end{align}

Combining the estimates above and following the same reasoning used to obtain \eqref{eq:estimstep2} in the case $|\al|=1$, one completes the proof of the induction step. Therefore, the proof of \cref{lem:derdcys<} is complete.
\end{proof}

We now combine \cref{lem:derdcys<} with \Cref{lem:Lpcon,lem:entFEdcay} to get long-time exponential decay of $\|\nabla^{\otimes n}\mu^t\|_{L^q}$. This then establishes estimate \eqref{eq:propnabns<} of \cref{prop:LpLq}.

\begin{lemma}\label{lem:'derdcys<}
Let $d\geq 2$ and $d-2\leq s\leq d-1$. For each $n\in\N$ and $1\leq p\leq q\leq\infty$, there exists a function $\W_{n,p,q}: [0,\infty)^4 \rightarrow [0,\infty)$, continuous, nondecreasing, and polynomial in its arguments, such that following holds. If $\mu$ is a solution to equation \eqref{eq:lim}, then
\begin{multline}
\forall t>0, \qquad \|\nabla^{\otimes n} \mu^t\|_{L^q} \leq \W_{n,p,q}(\|\mu^0\|_{L^{\infty}},\sigma^{-1}, \|\mu^0-1\|_{L^p},\Fc_{\sigma}(\mu^0))  e^{-C\sigma t}\\
\times\min(\sigma t,1)^{- \frac{n}{2}-\frac{d}{2}\left(\frac1p-\frac1q\right)}\paren*{1 + C_\ep\min(\sigma t,1)^{-\ep}\indic_{s=d-1 \wedge q=1}},
\end{multline}
where $C>0$ depends on $n,d,p,q$ and $C_\ep>0$ depends on $\ep>0$, which can be made arbitrarily small. The function $\W_{n,p,q}$ additionally depends on $d,s,\M$. Moreover, it is independent of $\Fc_{\sigma}(\mu^0)$ if $\M$ is antisymmetric.
\end{lemma}
\begin{proof}
Fix $t>0$ and assume that $\sigma t> 1$ (otherwise, the desired result is already handled by \cref{lem:derdcys<}). Let $\sigma t_0 =  \sigma t-\frac12$. Translating time, we may apply \cref{lem:derdcys<} to obtain
\begin{align}
\|\nabla^{\otimes n} \mu^t\|_{L^q} &\leq \W_{n,p,q}(\|\mu^{t_0}\|_{L^{\infty}},\sigma^{-1}, \|\mu^{t_0}-1\|_{L^p},\Fc_{\sigma}(\mu^{t_0})) (\sigma(t-t_0))^{- \frac{n}{2}-\frac{d}{2}\left(\frac1p-\frac1q\right)}\nn\\
&\ph\qquad \times \Bigg(\|\mu^{t_0}-1\|_{L^p} + \left( \|\mu^{t_0}\|_{L^\infty}^{1-\frac1p}\Big(\Fc_{\sigma}(\mu^{t_0})/\sigma\right)^{\frac{1}{2p}}\indic_{p<\infty} \nn\\
&\ph\qquad + C_{\epsilon}\|\mu^{t_0}\|_{L^\infty}^{1-\epsilon}\left(\Fc_{\sigma}(\mu^{t_0})/\sigma\right)^{\frac{\epsilon}{2}}\indic_{p=\infty}\Big) \indic_{\M=-\I} \Bigg)\paren*{1 + C_\ep (\sigma(t-t_0))^{-\ep}\indic_{s=d-1 \wedge q=1}} \nn\\
&\leq C\W_{n,p,q}(\|\mu^{t_0}\|_{L^{\infty}},\sigma^{-1}, \|\mu^{t_0}-1\|_{L^p},\Fc_{\sigma}(\mu^{t_0}))\paren*{1 + C_\ep'\indic_{s=d-1 \wedge q=1}} \nn\\
&\ph\qquad\times \Bigg(\|\mu^{t_0}-1\|_{L^p} + \Big( \|\mu^{t_0}\|_{L^\infty}^{1-\frac1p}\left(\Fc_{\sigma}(\mu^{t_0})/\sigma\right)^{\frac{1}{2p}}\indic_{p<\infty} \nn\\
&\ph\qquad + C_{\epsilon}\|\mu^{t_0}\|_{L^\infty}^{1-\epsilon}\left(\Fc_{\sigma}(\mu^{t_0})/\sigma\right)^{\frac{\epsilon}{2}}\indic_{p=\infty}\Big) \indic_{\M=-\I} \Bigg),
\end{align}
where the second inequality follows from $\sigma(t-t_0)>\frac{1}{2}$. We have $\|\mu^{t_0}\|_{L^\infty} \leq \|\mu^0\|_{L^\infty}$. Applying estimate \eqref{eq:muLPexpcon} from \cref{lem:Lpcon} to $\|\mu^{t_0}-1\|_{L^p}$ and estimate \eqref{eq:entFEdcay} from \cref{lem:entFEdcay} to $\Fc_{\sigma}(\mu^0)$, then using that $\W_{n,p,q}$ is nondecreasing in its arguments, we find
\begin{multline}
\W_{n,p,q}(\|\mu^{t_0}\|_{L^{\infty}},\sigma^{-1}, \|\mu^{t_0}-1\|_{L^p},\Fc_{\sigma}(\mu^{t_0}))\Bigg(\|\mu^{t_0}-1\|_{L^p} + \Big( \|\mu^{t_0}\|_{L^\infty}^{1-\frac1p}\left(\Fc_{\sigma}(\mu^{t_0})/\sigma\right)^{\frac{1}{2p}}\indic_{p<\infty} \nn\\
+ C_{\epsilon}\|\mu^{t_0}\|_{L^\infty}^{1-\epsilon}\left(\Fc_{\sigma}(\mu^{t_0})/\sigma\right)^{\frac{\epsilon}{2}}\indic_{p=\infty}\Big) \indic_{\M=-\I} \Bigg)\\
\leq \tl\W_{n,p,q}(\|\mu^{0}\|_{L^{\infty}},\sigma^{-1}, \|\mu^{0}-1\|_{L^p},\Fc_{\sigma}(\mu^{0}))e^{-C'\sigma t},
\end{multline}
where $C'>0$ and $\tl\W_{n,p,q}: [0,\infty)^4\rightarrow [0,\infty)$ is a continuous, nondecreasing polynomial function of its arguments, which is independent of $\Fc_{\sigma}(\mu^0)$ if $\M$ is antisymmetric. This completes the proof.
\end{proof}

\begin{remark}\label{rem:derdcys<fd}
Using the fractional Leibniz rule in place of the ordinary Leibniz rule, one can adapt the proof of \cref{lem:derdcys<}, then use the same argument as in the proof of \cref{lem:'derdcys<}, to also obtain, for any $\al>0$ and $\ep>0$, 
\begin{multline}\label{eq:derdcys<fd}
\forall t>0, \qquad \|\Dm^\al \mu^t\|_{L^q} \leq \W_{\al,p,q}(\|\mu^0\|_{\infty},\sigma^{-1}, \|\mu^0-1\|_{L^p}, \Fc_{\sigma}(\mu^0))\\
\times e^{-C\sigma t} \min(\sigma t,1)^{ - \frac{\al}{2}-\frac{d}{2}\left(\frac1p-\frac1q\right)}\paren*{1 + C_\ep \min(\sigma t,1)^{-\ep}\indic_{s=d-1 \wedge q=1}},
\end{multline}
where $\W_{\al,p,q}: [0,\infty)^4 \rightarrow [0,\infty)$ is a function with the same properties as $\W_{n,p,q}$ above. This establishes estimate \eqref{eq:propnabals<} of \cref{prop:LpLq}. Alternatively, one can obtain the decay estimate \eqref{eq:derdcys<fd} following the proof of \cref{lem:derdcys>} presented in the next subsection.
\end{remark}

\subsection{The case $d-1<s<d$}\label{ssec:Rlxs>}
Next, we establish the analogue of \cref{lem:derdcys<} in the more difficult case $d-1<s<d$. Recall from above that the difficulty stems from the loss of regularity in the vector field $\M\nabla\g\ast\mu$, an issue we already saw in the proof of \cref{prop:lwp} for local well-posedness.

As an intermediate step, we first prove a uniform-in-time bound for $L^2$ norms of (fractional) derivatives of $\mu^t$, which by Sobolev embedding, will yield uniform-in-time control of the quantity $\||\nab|^{s-d+1} \mu^\tau\|_{L^p}$, for any $1\leq p\leq \infty$, arising in estimation of the vector field $\M\nabla\g\ast\mu^\tau$.

\begin{lemma}\label{lem:intmd} 
Let $d\geq 1$ and $d-1<s<d$. Let $\mu $ be a solution to \eqref{eq:lim} with unit mass and such that $\mu^0\geq 0$ if $\M=-\I$. For $\al>0$, there exists a $C>0$ depending only on $d,s,\al,\sigma,\M$, such that
\begin{equation}\label{eq:intmd}
\forall t>0, \qquad \|\mu^t\|_{\dot{H}^\al}^2 \leq \|\mu^0\|_{\dot{H}^{\al}}^2 + \sigma^{-1}\tl\W_{\al}(\|\mu^0\|_{L^\infty}, \|\mu^0-1\|_{L^{\frac{2(\al+d-s)}{d-s}}}, \Fc_{\sigma}(\mu^0)/\sigma).
\end{equation}
where $\tl\W_{\al}: [0,\infty)^3\rightarrow [0,\infty)$ is a continuous, nondecreasing, polynomial function of its arguments, which does not depend on its third argument if $\M$ is antisymmetric. Additionally, there exists a $T_*>0$ such that $\|\mu^t\|_{\dot{H}^{\al}}^2$ is decreasing on $[T_*,\infty)$. 
\end{lemma}
\begin{proof}
We may assume that $\mu^0\neq 1$; otherwise, the left-hand side of \eqref{eq:intmd} is identically zero and there is nothing to prove. By \cref{rem:clssol}, we may assume without loss of generality that $\mu$ is a classical solution.
We have for $\al>0$,
\begin{align}
    \frac{d}{dt} \hal \|\Dm^\al\mu^t\|_{L^2}^2 &= \int_{\T^d} \Dm^\al\mu^t \left(\sigma\D\Dm^\al\mu^t -\div \Dm^\al (\mu^t\M\nab\g\ast\mu^t)\right) dx \nn \\
    &= -\sigma\int_{\T^d} |\nab\Dm^\al\mu^t|^2 dx + \int_{\T^d} \nab\Dm^\al\mu^t \cdot \Dm^\al(\mu^t\M\nab\g\ast\mu^t) dx ,\label{eq:ubndpre}
\end{align}
where the ultimate line follows from integration by parts. 

Consider the second term in \eqref{eq:ubndpre}. By Cauchy-Schwarz and the fractional Leibniz rule (e.g., see \cite[Theorem 7.6.1]{Grafakos2014m}), we have, for any exponent $2\leq p\leq\infty$,
\begin{align}
    \left|\int_{\T^d} \nab\Dm^\al\mu^t \cdot \Dm^\al(\mu^t\M\nab\g\ast\mu^t) dx \right| &\le \|\nab\Dm^\al\mu^t\|_{L^2} \| \Dm^\al(\mu^t\M\nab\g\ast\mu^t)\|_{L^2} \nn \\
    &\lesssim  \|\nab\Dm^\al\mu^t\|_{L^2} \Big( \|\Dm^\al\mu^t\|_{L^p}\|\M\nab\g\ast\mu^t\|_{L^\frac{2p}{p-2}}  \nn \\
    &\qquad + \|\mu^t\|_{L^\infty} \|\M\nab\g\ast\Dm^\al\mu^t\|_{L^2}\Big). \label{eq:FLbndapp}
\end{align}
We choose $p=\frac{2(1+\al)}{\al}$. Then by the fractional Gagliardo-Nirenberg interpolation inequalities (e.g., see \cite[Theorem 2.44]{BCD2011}), 
\begin{align}
    \|\Dm^\al\mu^t\|_{L^p} &\lesssim \|\mu^t\|_{\dot{H}^{1+\al}}^{\frac{\al}{1+\al}}\|\mu^t-1\|_{L^\infty}^{\frac{1}{1+\al}}, \\
\label{eq:GN}
    \|\M\nab\g\ast\mu^t\|_{L^\frac{2p}{p-2}} &\lesssim \|\mu^t\|_{\dot{H}^{1+\al}}^{\frac{s+1-d}{1+\al}} \|\mu^t-1\|_{L^{\frac{2(\al+d-s)}{d-s}}}^{\frac{\al-s+d}{1+\al}},
\end{align}
which allows to handle the first product inside the parentheses in \eqref{eq:FLbndapp}. For the second product, we trivially estimate
\begin{equation}
    \|\M\nab\g\ast\Dm^\al\mu^t\|_{L^2} \lesssim \|\mu^t\|_{\dot{H}^{1+\al+s-d}} \leq \|\mu^t-1\|_{L^2}^{\frac{d-s}{1+\al}} \|\mu^t\|_{\dot{H}^{1+\al}}^{\frac{1+\al+s-d}{1+\al}} .
\end{equation}
Combining the above estimates, we obtain
\begin{multline}\label{eq:estimdt}
    \frac{d}{dt} \|\mu^t\|_{\dot{H}^{\al}}^2 \leq -  \sigma\|\mu^t\|_{\dot{H}^{1+\al}}^2  +C\|\mu^t\|_{\dot{H}^{1+\al}}^{2+\frac{s-d}{1+\al}} \|\mu^t-1\|_{L^{\frac{2(\al+d-s)}{d-s}}}^{\frac{\al-s+d}{1+\al}} \|\mu^t-1\|_{L^\infty}^{\frac{1}{1+\al}} \\
     + C\|\mu^t\|_{\dot{H}^{1+\al}}^{2+\frac{s-d}{1+\al}}\|\mu^t\|_{L^\infty} \|\mu^t-1\|_{L^2}^{\frac{d-s}{1+\al}}, 
\end{multline}
for some constant $C>0$ depending only on $d,s,\al,\M$. By Plancherel's theorem,
\begin{align}
\|\mu^t\|_{\dot{H}^{1+\al}}^{2} \geq \|\mu^t\|_{\dot{H}^{1+\al}}^{2+\frac{s-d}{1+\al}} \paren*{\|\mu^t\|_{L^2}^2 - 1}^{\frac{d-s}{2(1+\al)}} \geq \|\mu^t\|_{\dot{H}^{1+\al}}^{2+\frac{s-d}{1+\al}}\paren*{\|\mu^0\|_{L^2}^2-1}^{\frac{d-s}{2(1+\al)}},
\end{align}
where the final inequality follows from $\|\mu^t\|_{L^2}\geq \|\mu^0\|_{L^2}>1$, since the $L^2$ norm is nonincreasing. Thus, using that $\|\mu^t\|_{L^\infty}$ is nonincreasing and triangle inequality, it follows from \eqref{eq:estimdt} that
\begin{multline}\label{eq:rhssd}
 \frac{d}{dt} \|\mu^t\|_{\dot{H}^{\al}}^2 \leq \|\mu^t\|_{\dot{H}^{1+\al}}^{2+\frac{s-d}{1+\al}}\Big(-\sigma\paren*{\|\mu^0\|_{L^2}^2-1}^{\frac{d-s}{2(1+\al)}} + C \|\mu^t-1\|_{L^{\frac{2(\al+d-s)}{d-s}}}^{\frac{\al-s+d}{1+\al}}\paren*{1+\|\mu^0\|_{L^\infty}}^{\frac{1}{1+\al}} \\
+ C\|\mu^0\|_{L^\infty} \|\mu^t-1\|_{L^2}^{\frac{d-s}{1+\al}} \Big).
\end{multline}
Applying the exponential decay of $\|\mu^t-1\|_{L^{\frac{2(\al+d-s)}{d-s}}}, \|\mu^t-1\|_{L^2}$ given by estimate \eqref{eq:muLPexpcon} of \cref{lem:Lpcon}, in the conservative case, or estimate \eqref{eq:entFEdcaygf} of \cref{lem:entFEdcay}, in the dissipative case, we see that there is a $T_*>0$, a lower bound for which is explicitly computable, such that the right-hand side of \eqref{eq:rhssd} is $<0$ for all $t>T_*$. Hence, $\|\mu^t\|_{\dot{H}^{\al}}^2$ is strictly decreasing on $(T_*,\infty)$.

Using 
Young's product inequality, we see that for any $\ep>0$, the right-hand side of \eqref{eq:estimdt} is $\leq$
\begin{multline}
\paren*{-\sigma  +\left(2+\frac{s-d}{1+\al}\right)\ep} \|\mu^t\|_{\dot{H}^{1+\al}}^2 + \frac{(d-s)}{2(1+\al)}\paren*{C\ep^{-\frac{2+\frac{s-d}{1+\al}}{2}}\|\mu^t-1\|_{L^\infty}^{\frac{1}{1+\al}}\|\mu^t-1\|_{L^{\frac{2(\al+d-s)}{d-s}}}^{\frac{\al+d-s}{1+\al}}}^{\frac{2(1+\al)}{d-s}} \\
+ \frac{(d-s)}{2(1+\al)}\paren*{C\ep^{-\frac{2+\frac{s-d}{1+\al}}{2}} \|\mu^t\|_{L^\infty} \|\mu^t-1\|_{L^2}^{\frac{d-s}{1+\al}}}^{\frac{2(1+\al)}{d-s}}.
\end{multline}
Choosing $\ep$ sufficiently small depending on $d,s,\al,\sigma$, we see that the first term is nonpositive. Using that $\|\mu^t\|_{L^\infty}$ is nonincreasing, we now conclude from the fundamental theorem of calculus that
\begin{equation}\label{eq:HalGron}
\|\mu^t\|_{\dot{H}^{\al}}^2 \leq \|\mu^0\|_{\dot{H}^\al}^2+ C_\ep\int_0^t \left((1+\|\mu^0\|_{L^\infty})^{\frac{2}{d-s}}\|\mu^\tau-1\|_{L^{\frac{2(\al+d-s)}{d-s}}}^{\frac{2(\al+d-s)}{d-s}} + \|\mu^0\|_{L^\infty}^{\frac{2(1+\al)}{d-s}} \|\mu^\tau-1\|_{L^2}^2  \right)d\tau  .
\end{equation}
Using estimate \eqref{eq:muLPexpcon} from \cref{lem:Lpcon} in the conservative case and \eqref{eq:entFEdcaygf} from \cref{lem:entFEdcay} in the dissipative case, the preceding right-hand side is controlled by
\begin{align}
&\|\mu^0\|_{\dot{H}^{\al}}^2 + C_\ep\int_0^t e^{-C\sigma \tau}\Bigg( (1+\|\mu^0\|_{L^\infty})^{\frac{2}{d-s}}\Big( \|\mu^0-1\|_{L^\frac{2(\al+d-s)}{d-s}}^{\frac{2(\al+d-s)}{d-s}}\indic_{\M \ \text{a.s.}} \nn\\
&\ph\qquad + (1+\|\mu^0\|_{L^\infty})^{\frac{2(\al+d-s)}{d-s}-1} \sqrt{\Fc_{\sigma}(\mu^0)/\sigma}\indic_{\M=-\I}\Big) +\|\mu^0\|_{L^\infty}^{\frac{2(1+\al)}{d-s}}\Big(\|\mu^0-1\|_{L^2}^2\indic_{\M \ \text{a.s.}} \nn\\
&\ph\qquad + (1+\|\mu^0\|_{L^\infty}) \sqrt{\Fc_{\sigma}(\mu^0)/\sigma}\indic_{\M=-\I} \Big) \Bigg)d\tau \nn\\
&\leq \|\mu^0\|_{\dot{H}^{\al}}^2 + \frac{C_\ep C}{\sigma}\tl\W_{\al}(\|\mu^0\|_{L^\infty}, \|\mu^0-1\|_{L^{\frac{2(\al+d-s)}{d-s}}}, \Fc_{\sigma}(\mu^0)/\sigma),
\end{align}
where $\tl\W_\al$ is a continuous, nondecreasing function of its arguments, vanishing if any of its arguments is zero. Also, $\tl\W_{\al}$ does not depend on its third argument if $\M$ is antisymmetric and does not depend on its second argument if $\M=-\I$. Implicitly, we have used above that $\|\cdot\|_{L^2}\leq \|\cdot\|_{L^{\frac{2(\al+d-s)}{d-s}}}$ in arriving at the final inequality. With this final estimate, the proof of the lemma is complete.
\end{proof}

\begin{remark}\label{rem:intL2}
If $1\le p \le 2$, then by H\"older's inequality and \cref{lem:intmd} with $\al=s-d+1$,
\begin{align}
\|\Dm^{s-d+1}\mu^t\|_{L^p} &\leq\|\mu^t\|_{\dot{H}^{s-d+1}} \nn\\
&\leq 2\left(\|\mu^0\|_{\dot{H}^{s-d+1}} + \sqrt{\sigma^{-1}\tl\W_{s-d+1}(\|\mu^0\|_{L^\infty}, \|\mu^0-1\|_{L^{\frac{2}{d-s}}}, \sqrt{\Fc_{\sigma}(\mu^0)/\sigma})} \right).
\end{align}
If $2< p <\infty$, then by Sobolev embedding and \cref{lem:intmd} with $\alpha=1+s-d+d(\frac12-\frac1p)$,
\begin{multline}
\|\Dm^{s-d+1}\mu^t\|_{L^p} \lesssim \|\mu^t\|_{\dot{H}^{1+s-d+d\left(\frac{1}{2}-\frac{1}{p}\right)}} \leq \Bigg((\|\mu^0\|_{\dot{H}^{s+1-d\left(\frac{1}{2}+\frac{1}{p}\right)}}  \\
+ \sqrt{\sigma^{-1}\tl\W_{1+d(\frac12-\frac1p)+s-d}(\|\mu^0\|_{L^\infty}, \|\mu^0-1\|_{L^{\frac{2+2d(\frac12-\frac1p)}{d-s}}}, \sqrt{\Fc_{\sigma}(\mu^0)/\sigma})} \Bigg).
\end{multline}
If $p=\infty$, due to the failure of endpoint Sobolev embedding, we instead have the preceding bound with an arbitrarily small $\vep$ added to $1+s-d+d(\frac12-\frac1p)$. In all cases, there exist $\la_p>0$ defined by 
\begin{equation}
\la_p \coloneqq 
    \begin{cases}
    1+s-d, & \ 1\le p\le 2 \\
    1+s-d +d\left(\frac12-\frac1p\right), & \ 2<p<\infty \\
     (1+s- \frac{d}{2})+, & \ p=\infty,
    \end{cases}
\end{equation}
such that for any $1\leq p\leq\infty$,
\begin{equation}
\forall t\geq 0, \qquad \|\nabla\g\ast\mu^t\|_{L^p} \leq C\Big(\|\mu^0\|_{\dot{H}^{\la_p}}  + \sqrt{\sigma^{-1}\tl\W_{\la_p}(\|\mu^0\|_{L^\infty}, \|\mu^0-1\|_{L^{\frac{2(\la_p+d-s)}{d-s}}}, \Fc_{\sigma}(\mu^0)/\sigma)} \Big).
\end{equation}
\end{remark}

With \cref{lem:intmd} in hand, we are now ready to prove the analogue of \cref{lem:derdcys<} in the case $d-1<s<d$.

\begin{lemma}\label{lem:derdcys>}
Let $d\geq 1$ and $d-1<s<d$. For every $\al>0$, and $1\leq  q\leq\infty$, there exists a function $\W_{\al,q} : [0,\infty)^5 \rightarrow [0,\infty)$, which is continuous, nondecreasing, polynomial in its arguments, such that for any solution $\mu$ to \eqref{eq:lim}, it holds that
\begin{multline}\label{eq:derdcys>}
\forall t\in (0,\sigma^{-1}], \qquad \|\Dm^{\al}\mu^t\|_{L^q} \leq (\sigma t)^{-\frac{\al}{2}}\paren*{1+C_\ep (\sigma t)^{-\ep}\indic_{q=\infty}} \\
\times\W_{\al,q}(\|\mu^0\|_{L^\infty}, \|\mu^0\|_{\dot{H}^{\la_2}},\sigma^{-1}, \|\mu^0-1\|_{L^{\frac{2}{d-s}}}, \Fc_{\sigma}(\mu^0)),
\end{multline}
where $C>0$ depends on $d,\al,s,q,\M$, $\ep>0$ is arbitrary, and $C_\ep>0$ depends only on $d,s,\ep$. The function $\W_{\al,q}$ additionally depends on $d,s,\M$ and is independent of its fifth argument if $\M$ is antisymmetric.
\end{lemma}
\begin{proof}
We first prove the assertion \eqref{eq:derdcys>} in the case $q=2$. We will then treat general $L^q$ norms by H\"older's inequality ($q\leq 2$) and Gagliardo-Nirenberg interpolation ($q>2$). The reason for this approach is that we need an \textit{a priori} uniform-in-time bound on $\|\Dm^{{ s+1-d}}\mu^\tau\|_{L^q}$ if we try to directly start with general $q$, and the only way we know how to obtain such a bound is through an intermediate $L^2$ estimate (i.e., \cref{rem:intL2}) and Sobolev embedding as commented above.

Starting from the \eqref{eq:defJ123} with $\p_\al$ replaced by $\Dm^\al$ and recycling notation, we define 
\begin{align}
    J_1(t) &\coloneqq \| e^{ \sigma t\D}\Dm^\al\mu^0 \|_{L^2}, \\
    J_2(t) &\coloneqq \int_0^{t(1-\ep)} \| e^{\sigma(t-\tau)\D} \div \Dm^\al\left( \mu^\tau \M\nab\g\ast\mu^\tau \right) \|_{L^2}d\tau, \\
    J_3(t) &\coloneqq \int_{t(1-\ep)}^t \| e^{\sigma(t-\tau)\D} \div \Dm^\al\left( \mu^\tau \M\nab\g\ast\mu^\tau \right) \|_{L^2}d\tau.
\end{align}
for $\ep\in (0,1)$ to be determined. Analogous to \eqref{eq:J1estim}, heat kernel estimates give
\begin{equation}\label{eq:J1s>}
J_1(t) \lesssim e^{-C\sigma t}\min(\sigma t,1)^{-\frac{\al}{2}}\|\mu^0-1\|_{L^2}.
\end{equation}
For $J_2(t)$, we also have
\begin{align}
&\left\| \Dm^\al e^{\sigma(t-\tau)\D} \div \left(\mu^\tau \M\nabla\g\ast\mu^\tau  \right)\right\|_{L^2} \nn\\
&\lesssim e^{-C\sigma(t-\tau)} \min(\sigma(t-\tau),1)^{- \frac{1+\al}{2}} \min(\sigma\tau,1)^{-\hal} \|\mu^0\|_{L^d} \| \M\nabla\g\ast\mu^\tau \|_{L^2}\nn \\ 
&\lesssim e^{-C\sigma(t-\tau)} \min(\sigma(t-\tau),1)^{- \frac{1+\al}{2}} \min(\sigma\tau,1)^{-\hal}\|\mu^0\|_{L^d}\Bigg(\|\mu^0\|_{\dot{H}^{\la_2}} \nn\\
&\ph\qquad+ \sqrt{\sigma^{-1}\tl\W_{\la_2}(\|\mu^0\|_{L^\infty}, \|\mu^0-1\|_{L^{\frac{2}{d-s}}}, \sqrt{\Fc_{\sigma}(\mu^0)/\sigma})}\Bigg),
\end{align}
where we have used \eqref{eq:CLdcyhypgf} from \cref{cor:CLdcy} in the first inequality and \cref{lem:intmd} in the second. Thus,
\begin{multline}\label{eq:J2s>}
    J_2(t) \lesssim \frac{A_{\ep,\al}}{\sigma} (\sigma t)^{-\frac{\al}{2}} \|\mu^0\|_{L^d}\Bigg(\|\mu^0\|_{\dot{H}^{\la_2}} \\
+ \sqrt{\sigma^{-1}\tl\W_{\la_2}(\|\mu^0\|_{L^\infty}, \|\mu^0-1\|_{L^{\frac{2}{d-s}}}, \sqrt{\Fc_{\sigma}(\mu^0)/\sigma})}\Bigg)  ,
\end{multline}
where, similar to \eqref{eq:Aepdefs<bc}, 
\begin{equation}
    A_{\ep,\al} \coloneqq \int_0^{1-\ep} (1-\tau)^{-\frac{d}{2}\left( \frac{1}{p}-\frac{1}{q}\right) - \frac{1+\al}{2}} \tau^{-\hal} d\tau.
\end{equation}

For $J_3(t)$, we choose $\d'\in (1+s-d,1)$ so that
\begin{equation}\label{eq:dd'}
\al + 1+s-d-\d' < \al.
\end{equation}
Using the fractional Leibniz rule (see \cite[Theorem 7.6.1]{Grafakos2014m}), we find that
\begin{align}
\left\|\Dm^{\al}e^{\sigma(t-\tau)\D}\div(\mu^\tau \M\nabla\g\ast\mu^\tau) \right\|_{L^2} &\lesssim \min(\sigma(t-\tau),1)^{-\frac{1+\d'}{2}} \left\|\Dm^{\al-\d'}(\mu^\tau \M\nabla\g\ast\mu^\tau)\right\|_{L^2} \nn\\
&\lesssim \min(\sigma(t-\tau),1)^{-\frac{1+\d'}{2}} \Bigg(\|\Dm^{\al-\d'}\mu^\tau\|_{L^{p_1}} \|\M\nabla\g\ast\mu^\tau\|_{L^{p_2}} \nn\\
&\ph \quad + \|\mu^\tau\|_{L^{\tl{p}_1}} \|\Dm^{\al-\d'}\M\nabla\g\ast\mu^\tau\|_{L^{\tl{p}_2}}\Bigg),\label{eq:isrhsJ3s>}
\end{align}
where $\frac{1}{p_1}+\frac{1}{p_2}=\frac{1}{\tl{p}_1}+\frac{1}{\tl{p}_2} = \frac{1}{2}$. Choose $(\tl{p}_1,\tl{p}_2) = (\infty,2)$, so that by the condition \eqref{eq:dd'},
\begin{align}
\|\mu^\tau\|_{L^{\tl{p}_1}} \|\Dm^{\al-\d'}\M\nabla\g\ast\mu^\tau\|_{L^{\tl{p}_2}} \lesssim \|\mu^\tau\|_{L^\infty}\|\Dm^{\al}\mu^\tau\|_{L^2}.
\end{align}
Note that $\al-\d'<\al-(s+1-d)$, by choice of $\d'$. So, using Gagliardo-Nirenberg interpolation, we have for the choice $(p_1,p_2) = (\frac{2\al}{\al-(s+1-d)}, \frac{2\al}{s+1-d})$ (which the reader may check is H\"older conjugate to $2$)
\begin{align}
\|\Dm^{\al-\d'}\mu^\tau\|_{L^{p_1}} \lesssim \|\Dm^{\al-(s+1-d)}\mu^\tau\|_{L^{p_1}} \lesssim  \|\mu^\tau-1\|_{L^\infty}^{\frac{s+1-d}{\al}} \|\mu^\tau\|_{\dot{H}^{\al}}^{1-\frac{s+1-d}{\al}} , \\
\|\M\nabla\g\ast\mu^\tau\|_{L^{p_2}} \lesssim \|\Dm^{s+1-d}\mu^\tau\|_{L^{p_2}} \lesssim \|\mu^\tau-1\|_{L^\infty}^{1-\frac{s+1-d}{\al}}  \|\mu^\tau\|_{\dot{H}^{\al}}^{\frac{s+1-d}{\al}} .
\end{align}
Evidently, the preceding implies
\begin{align}
\|\Dm^{\al-\d'}\mu^\tau\|_{L^{p_1}} \|\M\nabla\g\ast\mu^\tau\|_{L^{p_2}} \lesssim \|\mu^\tau-1\|_{L^\infty} \|\mu^\tau\|_{\dot{H}^\al} 
\end{align}
and in turn that the right-hand side of \eqref{eq:isrhsJ3s>} is $\lesssim$
\begin{align}
&\min(\sigma(t-\tau),1)^{-\frac{1+\d'}{2}}\|\mu^\tau\|_{L^\infty} \|\Dm^{\al}\mu^\tau\|_{L^2} \nn\\
&\lesssim \min(\sigma(t-\tau),1)^{-\frac{1+\d'}{2}}\min(\sigma\tau,1)^{-\frac{1-\d'}{2}} \|\mu^0\|_{L^{\frac{d}{1-\d'}}} \|\Dm^{\al}\mu^\tau\|_{L^2},
\end{align}
where the second line is by \eqref{eq:CLdcyhypgf} from \cref{cor:CLdcy} applied to $\|\mu^\tau\|_{L^\infty}$. Hence, defining $\phi(t)\coloneqq \sup_{t\geq\tau>0}(\sigma \tau)^{\frac{\al}{2}}\|\Dm^{\al}\mu^\tau\|_{L^2}$ and using dilation invariance of Lebesgue measure, we obtain the estimate 
\begin{equation}
J_3(t) \leq \frac{C B_{\ep,\al} \|\mu^0\|_{L^{\frac{d}{1-\d'}}}}{\sigma (\sigma t)^{\frac{\al}{2}}} \phi(t),
\end{equation}
where
\begin{equation}
B_{\ep,\al} \coloneqq \int_{1-\ep}^1 (1-\tau)^{-\frac{1+\d'}{2}} \tau^{-\frac{\al +(1-\d')}{2}}d\tau.
\end{equation}
Note that $\d'$ may be chosen independently of $\al$, hence we have omitted the dependence on it from our notation. Choosing $\ep$ sufficiently close to $1$ so that $CB_{\ep,\al}\|\mu^0\|_{L^\frac{d}{1-\d'}}<\frac{\sigma}{2}$, we arrive at 
\begin{multline}\label{eq:isphibndL2}
\phi(t) \leq C_{\al,2}\|\mu^0\|_{L^2} + \frac{A_{\ep,\al}\|\mu^0\|_{L^d}}{\sigma}\Bigg(\|\mu^0\|_{\dot{H}^{\la_2}} \\
\\
+ \sqrt{\sigma^{-1}\tl\W_{\la_2}(\|\mu^0\|_{L^\infty}, \|\mu^0-1\|_{L^{\frac{2}{d-s}}}, \sqrt{\Fc_{\sigma}(\mu^0)/\sigma})}\Bigg).
\end{multline}

From these $L^2$ estimates, we now obtain general $L^q$ estimates. If $1\leq q\leq 2$, then H\"older's inequality implies that $\sup_{0<t\leq \sigma^{-1}}(\sigma t)^{\frac{\al}{2}}\|\Dm^{\al}\mu^t\|_{L^q}$ is controlled by the right-hand side of \eqref{eq:isphibndL2}. If $2<q<\infty$, then choosing $\be = \frac{q\al}{2}$, Gagliardo-Nirenberg interpolation gives
\begin{align}
\|\Dm^{\al}\mu^t\|_{L^q} &\lesssim \|\mu^t-1\|_{L^\infty}^{1-\frac{2}{q}} \|\mu^t\|_{\dot{H}^{\be}}^{\frac{2}{q}}\nn\\
& \leq (\sigma t)^{-\frac{\al}{2}}\|\mu^0\|_{L^\infty}^{1-\frac{2}{q}}\Bigg(C_{\be,2}\|\mu^0-1\|_{L^2} + \frac{A_{\ep,\be}\|\mu^0\|_{L^d}}{\sigma}\Bigg(\|\mu^0\|_{\dot{H}^{\la_2}} \nn\\
&\ph\qquad+ \sqrt{\sigma^{-1}\tl\W_{\la_2}(\|\mu^0\|_{L^\infty}, \|\mu^0-1\|_{L^{\frac{2}{d-s}}}, \sqrt{\Fc_{\sigma}(\mu^0)/\sigma})}\Bigg)\Bigg)^{\frac2q}. \label{eq:DmalLq}
\end{align}
If $q=\infty$, then for $\vep>0$ and $1<r<\infty$, we have by Sobolev embedding and \eqref{eq:DmalLq},
\begin{align}
\|\Dm^{\al}\mu^t\|_{L^\infty} &\lesssim \|\Dm^{\al+\frac{d}{r}+\vep}\mu^t\|_{L^r}  \nn\\
&\leq (\sigma t)^{-\frac{\al+\frac{d}{r}+\vep}{2}}\|\mu^0\|_{L^\infty}^{1-\frac{2}{r}}\Bigg(C_{\be,2}\|\mu^0-1\|_{L^2} + \frac{A_{\ep,\be}\|\mu^0\|_{L^d}}{\sigma}\Bigg(\|\mu^0\|_{\dot{H}^{\la_2}} \nn\\
&\ph\qquad+ \sqrt{\sigma^{-1}\tl\W_{\la_2}(\|\mu^0\|_{L^\infty}, \|\mu^0-1\|_{L^{\frac{2}{d-s}}}, \sqrt{\Fc_{\sigma}(\mu^0)/\sigma})}\Bigg)\Bigg)^{\frac2r}, \label{eq:DmalLinf}
\end{align}
where $\be\coloneqq \al+\frac{d}{r}+\vep$. Choosing $r$ arbitrarily large, this completes the proof of the lemma.

\end{proof}

\begin{remark}
\emph{A posteriori}, one can infer from \cref{lem:derdcys>} that for all $n\in\N$ and $1\leq q\leq\infty$, there exists a function $\W_{n,q}$ with the same properties as $\W_{\al,q}$, such that
\begin{multline}
\forall t\in (0,\sigma^{-1}], \qquad \|\nabla^{\otimes n}\mu^t\|_{L^q} \leq (\sigma t)^{-\frac{\al}{2}}\paren*{1+C_\ep (\sigma t)^{-\ep}\indic_{q=\infty}} \\
\times\W_{\al,q}(\|\mu^0\|_{L^\infty}, \|\mu^0\|_{\dot{H}^{\la_2}},\sigma^{-1}, \|\mu^0-1\|_{L^{\frac{2}{d-s}}}, \Fc_{\sigma}(\mu^0)).
\end{multline}
Indeed, the case $q<\infty$ follows from \eqref{eq:derdcys>} using the $L^q$ boundedness of the Fourier multiplier $\frac{\nabla}{\Dm}$. The case $q=\infty$ follows from $\frac{\nabla}{\Dm^{1+\ep}}$ being bounded on $L^\infty$, for $\ep>0$.
\end{remark}

Similar to \cref{lem:'derdcys<}, we now combine \cref{lem:derdcys>} with \Cref{lem:Lpcon,lem:entFEdcay} to obtain estimate \eqref{eq:propnabals>} of \cref{prop:LpLq}. Estimate \eqref{eq:propnabns>} follows then from the preceding remark.

\begin{lemma}\label{lem:'derdcys>}
Let $d\geq 1$ and $d-1<s<d$. For every $\al>0$, and $1\leq  q\leq\infty$, there exists a function $\W_{\al,q} : [0,\infty)^5 \rightarrow [0,\infty)$, which is continuous, nondecreasing, polynomial in its arguments, such that for a solution $\mu$ to \eqref{eq:lim}, it holds that
\begin{multline}\label{eq:'derdcys>}
\forall t>0, \qquad \|\Dm^{\al}\mu^t\|_{L^q} \leq \min(\sigma t,1)^{-\frac{\al}{2}}\paren*{1+C_\ep \min(\sigma t,1)^{-\ep}\indic_{q=\infty}}e^{-C\sigma t} \\
\times\W_{\al,q}(\|\mu^0\|_{L^\infty}, \|\mu^0\|_{\dot{H}^{\la_2}},\sigma^{-1}, \|\mu^0-1\|_{L^{\frac{2}{d-s}}}, \Fc_{\sigma}(\mu^0)),
\end{multline}
where $C>0$ depends on $d,\al,s,q,\M$, $\ep>0$ is arbitrary, and $C_\ep>0$ depends only on $d,s,\ep$. The function $\W_{\al,q}$ additionally depends on $d,s,\M$, is independent of its fifth argument if $\M$ is antisymmetric.
\end{lemma}
\begin{proof}
Fix $t>0$ and assume that $\sigma t>1$ (otherwise, there is nothing to prove). Let $\sigma t_0 = \sigma t-\frac12 > \frac{\sigma t}{2}$. Then translating time and applying \cref{lem:derdcys>}, we obtain for any $\be>0$,
\begin{align}\label{eq:DbeL2}
\|\Dm^\be\mu^t\|_{L^2} &\leq  (\sigma(t-t_0))^{-\frac{\be}{2}}\paren*{1+C_\ep (\sigma(t-t_0))^{-\ep}\indic_{q=\infty}} \\
&\ph\qquad \times\W_{\be,2}(\|\mu^{t_0}\|_{L^\infty}, \|\mu^{t_0}\|_{\dot{H}^{\la_2}},\sigma^{-1}, \|\mu^{t_0}-1\|_{L^{\frac{2}{d-s}}}, \Fc_{\sigma}(\mu^{t_0})) \nn\\
&\leq C(1+C_\ep'\indic_{q=\infty})\W_{\be,2}(\|\mu^{t_0}\|_{L^\infty}, \|\mu^{t_0}\|_{\dot{H}^{\la_2}},\sigma^{-1}, \|\mu^{t_0}-1\|_{L^{\frac{2}{d-s}}}, \Fc_{\sigma}(\mu^{t_0})).
\end{align}
By \cref{lem:intmd}, \cref{lem:Lpcon}, \cref{lem:entFEdcay}, and the nondecreasing property of $\W_{\be,2}$,
\begin{multline}\label{eq:Wbe2}
\W_{\be,2}(\|\mu^{t_0}\|_{L^\infty}, \|\mu^{t_0}\|_{\dot{H}^{\la_2}},\sigma^{-1}, \|\mu^{t_0}-1\|_{L^{\frac{2}{d-s}}}, \Fc_{\sigma}(\mu^{t_0})) \leq \W_{\be,2}\Bigg(\|\mu^{0}\|_{L^\infty}, \\
\|\mu^0\|_{\dot{H}^{\la_2}}^2 + \sigma^{-1}\tl\W_{\la_2}(\|\mu^0\|_{L^\infty}, \|\mu^0-1\|_{L^{\frac{2}{d-s}}}, \Fc_{\sigma}(\mu^0)/\sigma), \sigma^{-1}, e^{-\frac{C\sigma t}{2}}\|\mu^0-1\|_{L^{\frac{2}{d-s}}}\indic_{\M \ \text{a.s.}} \\
+ (1+\|\mu^0\|_{L^\infty})^{1-\frac{d-s}{2}}\left(e^{-2\pi^2\sigma t}\sqrt{2\Fc_\sigma(\mu^0)/\sigma}\right)^{\frac{d-s}{2}}\indic_{\M =- \I} , e^{-2\pi^2\sigma t}\Fc_{\sigma}(\mu^0)\Bigg).
\end{multline}
Let us denote the right-hand side by $\tl\W_{\be,2}(\|\mu^0\|_{L^\infty}, \|\mu^0\|_{\dot{H}^{\la_2}}, \sigma^{-1},\|\mu^0-1\|_{L^{\frac{2}{d-s}}}, \Fc_{\sigma}(\mu^0))$. Now for $\al>0$, by interpolation, then combining \eqref{eq:DbeL2} and \eqref{eq:Wbe2} for $\be=2\al$,
\begin{align}
\|\Dm^\al \mu^t\|_{L^2} &\leq \|\mu^t-1\|_{L^2}^{\frac12} \|\Dm^{2\al}\mu^t\|_{L^2}^{\frac12} \nn\\
&\leq  \paren*{e^{-C\sigma t}\|\mu^0-1\|_{L^2}\indic_{\M \ \text{a.s.}} + (1+\|\mu^0\|_{L^\infty})^{\frac12} \left(e^{-4\pi^2\sigma t}\sqrt{2\Fc_{\sigma}(\mu^0)/\sigma}\right)^{\frac12}\indic_{\M=-\I} }^{\frac12}\nn\\
&\ph\qquad\times\tl\W_{\be,2}(\|\mu^0\|_{L^\infty}, \|\mu^0\|_{\dot{H}^{\la_2}}, \|\mu^0-1\|_{L^{\frac{2}{d-s}}}, \Fc_{\sigma}(\mu^0))^{\frac12},
\end{align}
where in the last line we have also used \Cref{lem:Lpcon,lem:entFEdcay}. Upon relabeling, this yields \eqref{eq:'derdcys>} for $q=2$. For $1\leq q< 2$, we may simply appeal to H\"older's inequality. For $2<q\leq\infty$, we appeal to Gagliardo-Nirenberg interpolation similar to \eqref{eq:DmalLq}, \eqref{eq:DmalLinf}.
\end{proof}

With the proof of \cref{lem:derdcys>} complete, the reader will see, after a little bookkeeping that the proof of \cref{prop:LpLq} is also complete.

\section{The modulated free energy approach}\label{sec:ME}
In this section, we explain how to prove uniform-in-time propagation of chaos for the system \eqref{eq:SDE} (in the gradient-flow case) using the modulated free energy approach. This will then complete the proof of our main result \cref{thm:main}.

\subsection{The modulated free energy inequality}\label{ssec:MEmfeineq}
The first step is to prove the following modulated free energy inequality. The very interesting phenomenon is that, compared to a pure modulated energy approach such as in \cite{RS2021}, the modulated free energy yields crucial cancellations in the dissipative case at positive temperature.

First, we must clarify what we mean by a solution to the Liouville equation \eqref{eq:liou}, since the kernel $\nabla\g$ is singular. We recall from \cite{JW2018, BJW2020} (e.g., see Definition 2.1 in the last reference) the definition of an entropy solution to the Liouville equation \eqref{eq:liou}. The proof of existence of an entropy solution to \eqref{eq:liou} is sketched in \cite[Section 4.2]{BJW2020} for the (attractive) case $s=0$. Following a similar argument, we sketch a proof of existence for the Riesz case \eqref{eq:gmod} in \cref{app}. In principle, entropy solutions need not be unique, though this is immaterial for our purposes.

\begin{mydef}\label{def:entsol}
Let $T>0$. We say that $f_N\in L^\infty([0,T], L^1((\T^d)^N))$, with $f_N^t\geq 0$ and $\int_{(\T^d)^N}df_N^t=1$, is an entropy solution to equation \eqref{eq:liou} on the interval $[0,T]$ if it solves \eqref{eq:liou} in the sense of distributions and for $0\leq t\leq T$, 
\begin{equation}\label{eq:entbnd}
\int_{(\T^d)^N}\log\paren*{\frac{f_N^t}{G_N}} df_N^t + \sigma\sum_{i=1}^N \int_0^t \int_{(\T^d)^N}\left|\nabla_{x_i}\log\paren*{\frac{f_N^\tau}{G_N}} \right|^2df_N^\tau \leq \int_{(\T^d)^N}\log\paren*{\frac{f_N^0}{G_N}}df_N^0,
\end{equation}
where $G_N \coloneqq \exp\paren*{-\frac{1}{2N\sigma}\sum_{1\leq i\neq j\leq N}\g(x_i-x_j)}$. We say that the entropy solution is global if the above holds on $[0,\infty)$.
\end{mydef}

\begin{lemma}\label{lem:entsol}
If $f_N^0$ is a probability density on $(\T^d)^N$ such that
\begin{equation}
\int_{(\T^d)^N}\log\paren*{\frac{f_N^0}{G_N}}df_N^0<\infty,
\end{equation}
then there exists a global entropy solution to \eqref{eq:liou} with initial datum $f_N^0$.
\end{lemma}

\begin{remark}
Given that the only entropy solutions we show exist are limits of sequences of smooth solutions to a regularized problem, there seems no harm in taking as part of \cref{def:entsol} that $f_N$ can be expressed as such a limit.
\end{remark}

Having dispensed with this technicality, one can rigorously establish an inequality for the evolution of the modulated free energy, as stated in \cref{prop:MFEid} below. We refer to \cite[Proposition 2.3]{BJW2020} for a proof.

\begin{prop}\label{prop:MFEid}
Assume that $f_N$ is an entropy solution to the Liouville equation \eqref{eq:liou} and that $\mu \in C([0,\infty), W^{2,\infty}(\T^d))$ solves equation \eqref{eq:lim}. Then the modulated free energy defined by \eqref{def:modulatedfreenrj} satisfies that
\begin{multline}\label{gronwall}
E_N(f_N^t, \mu^t) \leq E_N(f_N^0,\mu^0) -\frac{{\sigma^2}}{N}\int_0^t\int_{(\T^d)^N}\left|\nabla\log\paren*{\frac{f_N^\tau}{(\mu^\tau)^{\otimes N}}} - \nabla\log\paren*{\frac{G_N}{G_{(\mu^\tau)^{\otimes N}}}}\right|^2df_N^\tau\\
 -\frac{1}{2} \int_0^t\int_{(\T^d)^N}\int_{(\T^d)^2\setminus\triangle} (u^\tau(x)-u^\tau(y))\cdot \nabla\g(x-y) d\left(\mu_N^\tau - \mu^\tau\right)^{\otimes 2}(x,y)d f_N^\tau ,
\end{multline}
where $u^t \coloneqq {\sigma}\nabla\log\mu^t + \nabla \g \ast \mu^t$ and
\begin{equation}
G_{(\mu^\tau)^{\otimes N}}(\ux_N) \coloneqq \exp\paren*{-\frac1\sigma\sum_{i=1}^N \g\ast\mu^\tau(x_i)+\frac{N}{2{\sigma}}\int_{(\T^d)^2}\g(x-y)d(\mu^\tau)^{\otimes 2}}.
\end{equation}
\end{prop}

\begin{remark}
The second term in the right-hand side of \eqref{gronwall} is obviously nonpositive, and we will discard it in the sequel.
\end{remark}

\subsection{The modulated energy and functional inequalities on the torus}\label{ssec:MEme}
We now want to control the right-hand side of \eqref{gronwall} by the modulated free energy itself and conclude by application of the Gr\"{o}nwall-Bellman lemma. The control of the right-hand side is done in several manners in the literature. In \cite{NRS2021}, the authors use commutator estimates together with a renormalization procedure implemented through a smearing of the Dirac masses, which allows for Riesz-like potentials that are not exactly our potential $\g$. In particular, this method works for full range $0 \le s < d$.  Instead of smearing of the Dirac masses, Bresch et al. \cite{BJW2019edp} employed a regularization of the kernel $\nabla\g$ and a functional inequality (of the same kind as in \cite{Serfaty2020}) for this regularized kernel, which may be understood as a commutator estimate, though this connection is not made in \cite{BJW2019edp}. 

When considering the exact Riesz potential in the Coulomb/super-Coulomb case $d-2 \le s < d$, one can prove functional inequalities \eqref{eq:SerFI} using integration by parts. More precisely, one uses the Caffarelli-Silvestre extension procedure to replace $\g$ by the kernel of a local operator\footnote{A degenerate elliptic operator with an $A_2$ weight, for which there is a good theory \cite{FKS1982}.} in the extended space $\R^{d+k}$ and then exploits a \emph{stress-energy tensor} structure and integration parts. Combining this with the smearing procedure mentioned in the preceding paragraph, this method alows for \emph{sharp} estimates. This is an advantage over the approaches of \cite{NRS2021, BJW2019edp}, in the particular the latter work which encounters an inefficiency in the kernel regularization. Of course, the cost to the stress-tensor approach is the rigidity of the interaction---it must be exactly Riesz.

In the forthcoming article \cite{RS2022}, a proof of the sharp version of \eqref{eq:SerFI} with right-hand side \eqref{eq:RSFI} is given in Euclidean space using the above described stress-tensor approach. As there is a version of the extension the Caffarelli-Silvestre extension procedure for the torus, the proof can be straightforwardly adapted to the setting of this paper. In an effort to make the present article self-contained, we sketch below (see \cref{prop:FI}) the proof of this functional inequality on $\T^d$.

\medskip
As a first step, we need to discuss properties of the modulated energy, in particular the electric formulation as a renormalized energy following \cite{PS2017, Serfaty2020, Rosenzweig2021ne, RS2022}.\footnote{Strictly speaking, only the first and third cited works consider the periodic setting; but the arguments are adaptations of the Euclidean case anyway.}

The distribution $\frac{1}{\cd}\g_E$ is the kernel of the nonlocal operator $\Dm^{d-s}$ in $\R^d$. However, as popularized by Caffarelli and Silvestre \cite{CS2007}, $\g_E$ is the restriction to $\R^d\times\{0\}$ of the kernel
\begin{equation}
\G_E(X) \coloneqq |X|^{-s}, \qquad \forall X=(x,z) \in \R^d\times\R^k
\end{equation}
which satisfies (in the sense of distributions)
\begin{equation}\label{eq:divGE}
-\frac{1}{\bar{\mathsf{c}}_{d,s}}\div(|z|^\gamma\nabla\G_E) = \d_{0}, \qquad \R^d\times\R^k
\end{equation}
for $\gamma=s+1-d$ and $k=0$ if $s=d-2$ and $k=1$ if $d-2<s<d$.\footnote{The constant $\bar{c}_{d,s}$ should not be confused with the constant $\cd$ in \eqref{eq:gmod}.} We generally use capital letters (e.g., $X$) to denote points of the extended space $\R^d\times\R^k$. Such a representation also holds on $\T^d$, as shown in \cite{RS2014,RS2016fl}. Namely, let $\G$ denote the unique solution of
\begin{equation}\label{eq:divG}
-\frac{1}{\bar{\mathsf{c}}_{d,s}}\div(|z|^\gamma\nabla\G) = \d_0-\d_{\T^d\times\{0\}}, \qquad \T^d\times\R^k
\end{equation}
with $\int_{\T^d\times\R^k}\G d\d_{\T^d\times\{0\}}=0$. Here, $\d_{\T^d\times\{0\}}$ denotes the restriction to $\T^d$ viewed as a subspace of $\T^d\times\R^k$.

Following \cite{PS2017}, we also will use in \cref{sec:ME} the following truncation of the extended potential $\G$. For $0<\eta<\frac{1}{4}$, we let
\begin{equation}\label{eq:Getadef}
\G_{\eta} \coloneqq  \min\paren*{\G_{E}(\eta),\G_{E}} + \G-\G_{E} - \Cs_{\eta}, \qquad \T^d\times\R^k,
\end{equation}
where
\begin{equation}
\Cs_{\eta} \coloneqq \int_{\T^d\times\R^k}\paren*{\min\paren*{\G_{E}(\eta),\G_{E}} + \G-\G_{E}}d\d_{\T^d\times\{0\}}(X).
\end{equation}
The constant $\Cs_{\eta}$ is to enforce that $\G_{\eta}$ has zero average on $\T^d\times\{0\}$. The reader may check that
\begin{equation}\label{eq:divGeta}
-\frac{1}{\bar{\mathsf{c}}_{d,s}}\div\paren*{|z|^\gamma\nabla\G_\eta} = \d_0^{(\eta)}-\d_{\T^d\times\{0\}}, \qquad \T^d\times\R^k,
\end{equation}
where $\d_0^{(\eta)}$ is the positive measure supported on the sphere $\p B(0,\eta)\subset \T^d\times\R^k$ defined by
\begin{equation}\label{eq:d0eta}
\int_{\T^d\times\R^k}\varphi d\d_0^{(\eta)} = -\frac{1}{\bar{\mathsf{c}}_{d,s}}\int_{\p B(0,\eta)}\varphi(X)|z|^\gamma \g_{E}'(\eta) ,\qquad \forall \varphi\in C(\T^d\times\R^k),
\end{equation}
where $\g_E$ is viewed as a function on $\R$ (through radial symmetry) with an abuse of notation. Given $X\in\T^d\times\R^k$, we let $\d_X^{(\eta)} \coloneqq \d_0^{(\eta)}(\cdot-X)$ denote the translate by $X$.

We introduce the notation
\begin{align}
H_N \coloneqq \G\ast \paren*{\frac{1}{N}\sum_{i=1}^N \d_{X_i}-\tl\mu}, \label{eq:HNdef}\\
H_{N,\vec\eta} \coloneqq \frac{1}{N}\sum_{i=1}^N \G_{\eta_i}(\cdot-X_i) - \G\ast\tl\mu,
\end{align}
where $\tl\mu \coloneqq \mu\d_{\T^d\times\{0\}}$ is the identification of $\mu$ as a probability measure on $\T^d\times\R^k$ and $\vec\eta = (\eta_1,\ldots,\eta_N)$ is an $N$-tuple of smearing lengthscales. We use the notation $X_i=(x_i,0)$ to denote points $x_i$ embedded in the extended space $\T^d\times\R^k$. We also let $H_N^i(X) \coloneqq H_N(X) - \frac{1}{N}\G(X-X_i)$. Observe from \eqref{eq:divG}, \eqref{eq:divGeta} that
\begin{equation}\label{eq:divHNeta}
-\frac{1}{\bar{\mathsf{c}}_{d,s}}\div\paren*{|z|^{\gamma}\nabla H_{N,\vec\eta}} = \frac{1}{N}\sum_{i=1}^N \d_{X_i}^{(\eta_i)} - \tl\mu.
\end{equation}

Consider the quantity
\begin{equation}\label{eq:Fcdef}
\mathcal{F}^{\vec\eta} \coloneqq \frac{1}{2\bar{\mathsf{c}}_{d,s}}\paren*{\int_{\T^d\times\R^k} |z|^\gamma |\nabla H_{N,\vec\eta}|^2 dX - \frac{\bar{\mathsf{c}}_{d,s}}{N^2}\sum_{i=1}^N \int_{\T^d\times\R^k}\G_{\eta_i}d\d_{0}^{(\eta_i)} - \frac{2\bar{\mathsf{c}}_{d,s}}{N}\sum_{i=1}^N \int_{\T^d\times\R^k}{\Fs_{\eta_i}(x-x_i)d\tl\mu(x)}},
\end{equation}
where { $\Fs_{\eta_i} \coloneqq \G-\G_{\eta_i}$.} Using the identity \eqref{eq:divHNeta} and integration by parts, it is straightforward that $\Fc^{\vec\eta}$ converges to $\Fr_N(\ux_N,\mu)$ as $\max_{i}\eta_i \rightarrow 0$. One can say more: the expression $\Fc^{\vec\eta}$ is monotonically decreasing with respect to the parameters $\eta_i$ and becomes equal to the modulated energy $\Fr_N(\ux_N,\mu)$ when the $\eta_i$ are sufficiently small so that the balls $B(X_i,\eta_i)$ are disjoint.

\begin{prop}\label{prop:MElb}
Assume $d\geq 1$ and $d-2\leq s<d$. Let $\eta_i,\al_i\in (0,\frac14)$ such that $\eta_i\geq \al_i$. Given a pairwise distinct configuration $\ux_N\in (\T^d)^N$ and a density $\mu\in L^\infty(\T^d)$ with $\int_{\T^d}\mu=1$, then
\begin{equation}
\Fc^{\vec\eta} \geq \Fc^{\vec\al}.
\end{equation}
Defining the nearest-neighbor type length scale\footnote{The idea to have a length scale which depends on each point originates in \cite{LS2018, Serfaty2020}. The recognition of the importance, in particular for proving uniform-in-time convergence results, of weighting the typical inter-particle distance $N^{-1/d}$ by the maximum density of the points is due to \cite{RS2021}.} 
\begin{equation}
\rs_i \coloneqq \frac14\min\paren*{\min_{1\leq j\leq N : j\neq i} |x_i-x_j|, (N\|\mu\|_{L^\infty})^{-\frac{1}{d}}},\qquad \forall 1\leq i\leq N,
\end{equation}
then
\begin{equation}
\Fr_N(\ux_N,\mu) = \Fc^{\vec\eta}, \qquad \text{if $\eta_i\leq \rs_i$ for every $1\leq i\leq N$}.
\end{equation}
From this relation, it follows that there is a constant $C>0$ depending only $d,s$ such that
\begin{align}
\frac{1}{2\bar{\mathsf{c}}_{d,s}}\int_{\T^d\times\R^k} \zg|\nabla H_{N,\vec\eta}|^2dX \leq C\paren*{\Fr_N(\ux_N,\mu) {+} \frac{1}{2N^2}\sum_{i=1}^N \int_{\T^d\times\R^k}\G_{\eta_i}d\d_{0}^{(\eta_i)} + C\|\mu\|_{L^\infty}^{\frac{s}{d}}N^{\frac{s}{d}-1}}, \label{eq:gradHNeta}\\
\frac{1}{N^2}\sum_{i=1}^N \g_{E}(\rs_i) \leq  C\paren*{\Fr_N(\ux_N,\mu) {+} \frac{1}{2N^2}\sum_{i=1}^N \int_{\T^d\times\R^k}\G_{\eta_i}d\d_{0}^{(\eta_i)} + C\|\mu\|_{L^\infty}^{\frac{s}{d}}N^{\frac{s}{d}-1}}. \label{eq:gsum}
\end{align}
\end{prop}

\begin{proof}
We sketch the proof, which originates in \cite{PS2017} and has been adapted to \cite[Lemma 3.2]{Serfaty2020}, \cite[Lemma B.1]{AS2021}. Given $\al_i\leq \eta_i$, write
\begin{equation}
H_{N,\vec\eta} = H_{N,\vec\al} + \frac1N\sum_{i=1}^N \paren*{\G_{\eta_i}-\G_{\al_i}}(X-X_i)
\end{equation}
and expand the first term in $\Fc^{\vec\eta}$. Using integration by parts and the identities \eqref{eq:divGeta}, \eqref{eq:divHNeta}, we obtain
\begin{equation}\label{eq:Fcdiff}
\Fc^{\vec\al}-\Fc^{\vec\eta} = \frac{1}{2N^2}\sum_{1\leq i\neq j\leq N} \int_{\T^d\times\R^k} \paren*{\G_{\al_i}-\G_{\eta_i}}(X-X_i)d\paren*{\d_{X_j}^{(\al_j)}+\d_{X_j}^{(\eta_j)}}(X). 
\end{equation}
From the definition \eqref{eq:Getadef} of $\G_{\eta}$, we see that $\G_{\al_i}-\G_{\eta_i}\geq 0$ with support in the closed ball $\ol{B(0,\eta_i)}$. Since $\d_{X_j}^{(\al_j)},\d_{X_j}^{(\eta_j)}$ are positive measures, it follows that the integral in \eqref{eq:Fcdiff} is nonnegative and vanishes if the balls $B(X_i,\eta_i), B(X_j,\eta_j)$ are disjoint, for $i\neq j$. Letting $\max_i\al_i\rightarrow 0$ now yields $\Fc^{\vec\eta} = \Fr_N(\ux_N,\mu)$.

The relation \eqref{eq:gradHNeta} is an immediate consequence of the end result of the preceding paragraph and H\"older's inequality. The relation \eqref{eq:gsum} follows by the same reasoning as in the proof of \cite[Lemma B.1]{AS2021} and using \eqref{eq:ggE}.
\end{proof}

\begin{remark}\label{rem:MFEcoer}
Since one may directly estimate the self-interaction term, the relation \eqref{eq:gradHNeta} implies that the modulated energy is nonnegative up to a term vanishing as $N\rightarrow\infty$:
\begin{equation}\label{eq:MElbN}
    \Fr_N(\ux_N,\mu) \geq -\frac{\log(N\|\mu\|_{L^\infty})}{2dN}\indic_{s=0}-\mathsf{C}\|\mu\|_{L^\infty}^{\frac{s}{d}}N^{-1+\frac{s}{d}},
\end{equation}
where $\mathsf{C}>0$ depends only on $d,s$. We use a special font for the constant $\mathsf{C}$ to distinguish it in later computations. As previously commented, the order of the term $N^{-1+\frac{s}{d}}$ is sharp. Furthermore, it is known (e.g., see \cite[Proposition 3.6]{Serfaty2020}) that the modulated energy is coercive in the sense that it controls a negative-order Sobolev distance between the empirical measure $\mu_N \coloneqq \frac{1}{N}\sum_i \d_{x_i}$ and the density $\mu$: for any $\zeta>\frac{d}{2}+d-s$, 
\begin{equation}
\|\mu_N - \mu\|_{\dot{H}^{-\zeta}} \leq C\|\mu\|_{L^\infty}^{\frac{s}{d}} N^{\frac{s}{d}-1} + C\paren*{\Fr_N(\ux_N,\mu) + \frac{\log(N\|\mu\|_{L^\infty})}{2dN}\indic_{s=0} + C\|\mu\|_{L^\infty}^{\frac{s}{d}}N^{\frac{s}{d}-1}}^{1/2},
\end{equation}
where $C>0$ depends only on $d,s$.  From this relation, one can deduce that if the $N$-point configuration $\ux_N$ is regarded as a random vector in $(\T^d)^N$ with law $f_N$, so that $\mu_N$ is a random element in $\mathcal{P}(\T^d)$, then
\begin{multline}
\mathbb{E}_{f_N}\paren*{\|\mu_N - \mu\|_{\dot{H}^{-\zeta}}^2} \leq C\mathbb{E}_{f_N}\paren*{\Fr_N(\ux_N,\mu) +\frac{\log(N\|\mu\|_{L^\infty})}{2dN}\indic_{s=0}} \\
 + C\|\mu\|_{L^\infty}^{\frac{s}{d}}N^{\frac{s}{d}-1}\paren*{1+ \|\mu\|_{L^\infty}^{\frac{s}{d}}N^{\frac{s}{d}-1}}. 
\end{multline}
This yields a bound for the difference $f_{N;k}- (\mu)^{\otimes k}$ in a negative-order Sobolev space (see \cite[Remark 1.5]{RS2021}).

The relative entropy is obviously nonnegative by Jensen's inequality. Moreover, by sub-additivity, the total $N$-particle relative entropy controls the relative entropy of the $k$-point marginals. Using Pinsker's inequality, it follows that
\begin{equation}
\|f_{N;k}-\mu^{\otimes k}\|_{L^1} \leq \sqrt{2k H_{k}(f_{N;k} \vert \mu^{\otimes k})} \leq \sqrt{2k H_N(f_N \vert \mu^{\otimes N})}.
\end{equation}
The implied rate $O(k/N)$ for the relative entropy between $f_{N;k}$ and $\mu^{\otimes k}$ is in general \emph{not sharp}, as recently demonstrated by Lacker \cite{Lacker2021}, who shows that $O(k^2/N^2)$ is the sharp rate. We note, however, that this cited work is limited to interactions less singular than Riesz (e.g., bounded).

In any case, we conclude that the modulated free energy metrizes both propagation of chaos in the sense of convergence of marginals in the $L^1$ norm and convergence of the empirical measure in expected Sobolev distance. It is therefore a good quantity for quantitatively proving mean-field convergence.
\end{remark}

We now come to the periodic version of the new, sharp functional inequality from \cite{RS2022}.

\begin{prop}\label{prop:FI}
Assume $\mu \in L^\infty(\T^d)$ is such that $\int_{\T^d} \mu = 1$. For any pairwise distinct configuration $\ux_N \in (\T^d)^N$ and any Lipschitz map
$v: \T^{d}\to \R^{d}$, we have 
\begin{multline}\label{controlRHS}
\left|\int_{(\T^d)^2\setminus\triangle} \left( v(x)- v(y)\right)\cdot  \nab \g(x-y) d \left( \frac{1}{N}\sum_{i=1}^N \delta_{x_i}  - \mu\right)^{\otimes 2}(x,y)\right| 
\\ \le C \|\nab v\|_{L^\infty}\left(  F_N(\ux_N, \mu) +  \frac{\log(N\|\mu\|_{L^\infty})}{2dN} \indic_{s=0} + C\|\mu\|_{L^\infty}^{\frac{s}{d}}  N^{-1+\frac{s}{d}} \right) 
\end{multline} where $C$ depends only on $s$, $d$.
\end{prop}
\begin{proof}
We only sketch the proof. For more details, we refer to the upcoming work \cite{RS2022}.

First, the reader may check using \eqref{eq:divGeta}, \eqref{eq:divG} that if $W=(w,0), Y=(y,0) \in \T^d\times\R^k$ and $\eta\in (0,\frac14)$, then 
\begin{equation}\label{eq:lemme}
\int_{\T^d\times\R^k} \G(X-W) d\delta_Y^{(\eta)} (X) = \G_\eta(W-Y).
\end{equation}

We identify the vector field $v$ as a vector field on $\T^d\times\R^k$ by defining $v(X) \coloneqq (v(x),0)$. Desymmetrizing and breaking up the measure,
\begin{align}
&\int_{ (\T^d\times\R^k)^2\setminus\triangle }( v(X)-v(Y) )\cdot \nabla\G(X-Y) d\left( \frac{1}{N}\sum_{i=1}^N\delta_{X_i} - \tl\mu\right)^{\otimes 2} (X,Y) \nn\\
&= \sum_{i=1}^N \frac{2}{N} \int_{\T^d\times\R^k}  v(X_i) \cdot \nabla \G(X_i - Y) d\left( \frac{1}{N} \sum_{j\neq i} \delta_{X_j} - \tl\mu\right)(Y)  \nn\\
&\ph - 2 \int_{ (\T^d\times\R^k)^2\setminus\triangle } v(X) \cdot \nabla \G(X-Y) d\tl\mu(X)d\left( \frac{1}{N}\sum_{i=1}^N \delta_{X_i} - \tl\mu\right)(Y) \nn\\
&=  \frac2N \sum_{i=1}^N  v(X_i) \cdot \nabla H_N^i (X_i) - 2\int_{\T^d} v\cdot \nabla H_N d\mu, \label{eq:FIpre}
\end{align}
where the reader will recall the definitions of $H_{N},H_{N}^i$ from \eqref{eq:HNdef}. Using the identities
\begin{align}
H_N^i(X) = H_{N, \vec{\eta} } (X) - \frac1N\G_{\eta_i} (X-X_i)\quad \text{in } \ B(X_i, \eta_i), \\
H_N(X) = H_{N, \vec{\eta}}(X)  +\frac1N \sum_{i=1}^N(\G- \G_{\eta_i}) (X-X_i),
\end{align}
we rewrite the expression \eqref{eq:FIpre} as the sum $\Te_1 + \Te_2 +\Te_3$, where 
\begin{align}\label{eq:FIT1def}
\Te_1= 2 \int_{\T^d\times\R^k} v\cdot \nabla H_{N, \vec{\eta}} \ d \left( \frac 1N \sum_{i=1}^N \delta_{X_i}^{(\eta_i)} - \mu\right),
\end{align}
\begin{multline}\label{eq:FIT2def}
\Te_2 =  \frac2N \sum_{i=1}^N \int_{\T^d\times\R^k} ( v(X_i) - v(X) ) \cdot \nabla H_N^i(X) d \delta_{X_i}^{(\eta_i)} (X)
\\ - \frac2{N^2} \sum_{i=1}^N \int_{\T^d\times\R^k}(v(X)-v(X_i)) \cdot \nabla \G_{\eta_i} (X-X_i) d\delta_{X_i}^{(\eta_i)} (X) \\
 + \frac2{N}\sum_{i=1}^N \int_{\T^d\times\R^k} (v(X)-v(X_i)  ) \cdot \nabla (\G_{\eta_i}- \G) (X-X_i) d\tl\mu(X),
\end{multline}
\begin{multline}\label{eq:FIT3def}
\Te_3= \frac2N \sum_{i=1}^N \int_{\T^d\times\R^k} v(X_i) \cdot \nabla H_N^i d( \delta_{X_i} - \delta_{X_i}^{(\eta_i)} ) 
 - \frac2{N^2} \sum_{i=1}^N\int_{\T^d\times\R^k} v(X_i) \cdot \nabla \G_{\eta_i} (X -X_i) d \delta_{X_i}^{(\eta_i)} (X) 
\\ 
+ \frac2N  \sum_{i=1}^N\int_{\T^d\times\R^k} v(X_i) \cdot \nabla (\G_{\eta_i}-\G) (X-X_i) d\tl\mu(X).
\end{multline}

First, we claim $\Te_3=0$. Indeed, unpacking the definition of $H_N^i$,
\begin{multline}\label{eq:T3FIrw}
\Te_3= \frac{2}{N^2} \sum_{1\leq i\neq j\leq N} \int_{\T^d\times\R^k} v(X_i) \cdot \nabla \G(X-X_j) d ( \delta_{X_i} - \delta_{X_i}^{(\eta_i)} ) (X)  \\
- \frac{2}{N}  \sum_{i=1}^N \int_{(\T^d\times\R^k)^2} v(X_i) \cdot \nabla \G(X-Y) d\tl\mu(Y) d ( \delta_{X_i} - \delta_{X_i}^{(\eta_i)} ) (X) \\
- \frac2{N^2} \sum_{i=1}^N\int_{\T^d\times\R^k} v(X_i) \cdot \nabla \G_{\eta_i} (X -X_i) d \delta_{X_i}^{(\eta_i)} (X) 
+ \frac2N  \sum_{i=1}^N\int_{\T^d\times\R^k} v(X_i) \cdot \nabla (\G_{\eta_i}-\G) (X-X_i) d\tl\mu(X).
\end{multline}
Thanks to \eqref{eq:lemme}, we have 
\begin{equation}
\int_{\T^d\times\R^k} \nabla \G(X-X_j) d ( \delta_{X_i} - \delta_{X_i}^{(\eta_i)} ) (X) 
=  \nabla \G(X_i-X_j)-  \nabla \G_{\eta_i} (X_i-X_j),
\end{equation}
which vanishes since $\eta_i\leq \rs_i$ by assumption and $\G_{\eta_i}=\G$ outside of $B(0, \eta_i) \subset \T^d\times\R^k$.  Thus, the first line of \eqref{eq:T3FIrw} vanishes.  By the same reasoning, the second line of \eqref{eq:T3FIrw} equals 
\begin{equation}
-\frac2N \sum_{i=1}^N \int_{\T^d\times\R^k} v(X_i) \cdot \nabla( \G- \G_{\eta_i}) (X_i-Y)  d\tl\mu(Y),
\end{equation}
and therefore the second line cancels with the second term on the last line of \eqref{eq:T3FIrw}. It remains to show that 
\begin{equation}
\int_{\T^d\times\R^k} v(X_i) \cdot \nabla \G_{\eta_i} (X-X_i) d \delta_{X_i}^{(\eta_i)}(X) =0.
\end{equation}
This is a consequence of the fundamental theorem of calculus, the observation
\begin{align}
\nabla \G_{\eta_i} (X-X_i)\paren*{\delta_{X_i}^{(\eta_i)}-\d_{X_i}}(X) &=-\frac{1}{\cd}  \nabla \G_{\eta_i} (X-X_i) \div \left(|z|^\gamma \nabla \G_{\eta_i} (X-X_i) \right)\nn\\
&= -\frac{1}{\cd} \div\comm{\G_{\eta_i} (X-X_i)}{\G_{\eta_i}(X-X_i)},
\end{align}
and that the last $k$ components of $v$ vanish and the trace of $\nabla_x\G_{\eta_i}$ to $\T^d\times\{0\}$ has zero average. Above, $\comm{\cdot}{\cdot}$ denotes the stress-energy tensor, which is the $(d+k) \times (d+k)$ tensor defined by
\begin{align}
\comm{\varphi}{\psi}^{ij} \coloneqq |z|^\ga\left(\p_i\varphi\p_j\psi +\p_j\varphi\p_i\psi\right) - |z|^\ga\nabla\varphi\cdot\nabla\psi \d_{ij}, \qquad 1\leq i,j\leq d+k,
\end{align}
for test functions $\varphi,\psi$ on $\T^d\times\R^k$.

We write $\Te_1$ in terms of the divergence of the stress-energy tensor as in \cite{Serfaty2020} and integrate by parts to obtain
\begin{equation}\label{eq:T1FIfin}
|\Te_1| \leq  C\|\nabla v\|_{L^\infty} \int_{\T^d\times\R^k} \zg |\nabla H_{N,\vec\eta}|^2 dX.
\end{equation}

Finally, consider $\Te_2$. Since $\eta_i\leq \rs_i$, $\supp(\nabla\Fs_{\eta_i}) \subset \ol{B(0,\eta_i)}$ implies that the second and third lines simplify to
\begin{multline}\label{eq:T2Fssimp}
-\frac{2}{N}\sum_{i=1}^N \frac{1}{N}\Bigg( \frac{1}{N}\int_{\T^d\times\R^{k}}\nabla\Fs_{\eta_i}(X-X_i)\cdot \paren*{ v(X)- v(X_i)} d \delta_{X_i}^{(\eta_i)}(X) \\
-\int_{\T^d} \nabla\fs_{\eta_i}(x-x_i)\cdot \paren*{v(x)-v(x_i)}d\mu(x)\Bigg),
\end{multline}
where $\fs_{\eta_i}$ is the trace of $\Fs_{\eta_i}$ to $\T^d\times\{0\}$. Using $|\nabla\Fs_{\eta}| = \eta^{-s-1}$ on the support of $\d_{0}^{(\eta)}$, we may bound the first term inside the parentheses by $\frac{1}{N} \eta_i^{-s}\|\nabla v\|_{L^\infty}$; and using $|\nabla \fs_\eta|\leq |\nabla \g|$, we may bound the second term by $C \eta_i^{d-s} \|\mu\|_{L^\infty} \|\nabla v\|_{L^\infty}$. Using the explicit form \eqref{eq:d0eta} of $\delta_{X_i}^{(\eta_i)}$ and mean value theorem, we bound each summand in the first line of \eqref{eq:FIT2def} by
\begin{equation}\label{eq:preavg}
C \eta_i \|\nabla  v\|_{L^\infty} \int_{\partial B(X_i, \eta_i)} |\nabla H_N^i| \frac{|z|^\gamma}{\eta_i^{s+1}}d\mathcal{H}^{d+k-1}(X),
\end{equation}
where $\mathcal{H}^{d+k-1}$ denotes the $(d+k-1)$-dimensional Hausdorff measure in $\T^d\times\R^{k}$ (equivalent to surface measure). We set $\eta_i = t\rs_i$ and average (with respect to Lebesgue measure) over $t\in [\frac12,1]$. After using Cauchy-Schwarz, it follows that the average of \eqref{eq:preavg} is $\leq$
\begin{equation}\label{eq:T3FIfin}
\frac{C}{N}\|\nabla v\|_{L^\infty}\paren*{\frac{1}{N}\sum_{i=1}^N \rs_i^{-s} + \int_{\T^d\times\R^k}\zg|\nabla H_{N,\vec{\rs}}|^2dX}.
\end{equation}

After a little bookkeeping and using the relations \eqref{eq:gradHNeta}, \eqref{eq:gsum} from \cref{prop:MElb} to bound  the right-hand sides of \eqref{eq:T1FIfin} and \eqref{eq:T3FIfin}, we arrive at the statement of the proposition.

\end{proof}

\subsection{Conclusion of Gr\"onwall argument}\label{ssec:MEconc}
We now have all the ingredients necessary to show a uniform-in-time bound for the modulated free energy. This then completes the proof of \cref{thm:main}.

We perform a Gr\"onwall-type argument on the quantity
\begin{equation}\label{eq:EcNdef}
\mathcal{E}_N^t \coloneqq E_N(f_N^t,\mu^t) + \frac{\log(N\|\mu^t\|_{L^\infty})}{2Nd}\indic_{s=0} + \Cs\|\mu^t\|_{L^\infty}^{\frac{s}{d}}N^{\frac{s}{d}-1},
\end{equation}
where $\Cs>0$ is the constant in \eqref{eq:MElbN} and the reader will recall the definition of the modulated free energy from \eqref{def:modulatedfreenrj}. The inclusion of the last two terms is to obtain a nonnegative quantity. Using \eqref{gronwall} and \eqref{controlRHS}, averaging with respect to $f_N^t$, one obtains the inequality

\begin{equation}\label{eq:ENtineq}
\mathcal{E}_N^t \le \mathcal{E}_N^0 + C\int_0^t \|\nab u^\tau \|_{L^\infty}\left(\Fr_N(\ux_N^\tau,\mu^\tau) + \frac{\log(N\|\mu^\tau\|_{L^\infty})}{2Nd}\indic_{s=0} + C\|\mu^\tau\|_{L^\infty}^{\frac{s}{d}}N^{\frac{s}{d}-1} \right)d\tau,
\end{equation}
where the vector field $u^\tau$ is given by\footnote{Note that this is \emph{not} the same vector field as in the pure modulated energy approach due the presence of the first term coming from the relative entropy.}
\begin{equation}
u^\tau\coloneqq  {\sigma}\nabla\log(\mu^\tau) +\nabla\g\ast\mu^\tau.
\end{equation}
In arriving at \eqref{eq:ENtineq}, we have implicitly used that $\|\mu^t\|_{L^\infty}$ is nonincreasing. The constant $C$ depends only on $d,s$ and {taking $\Cs$ above larger if necessary, we may assume that $\Cs\geq C$}. Since the relative entropy is nonnegative, it may be added to the expression inside the parentheses and the inequality still holds. Applying the Gr\"{o}nwall-Bellman lemma, we obtain
\begin{equation}
\Ec_N(f_N^t,\mu^t) \le {\Ec_N(f_N^0,\mu^0)} \exp \left(C \int_0^t \| \nab u^\tau\|_{L^\infty}d\tau \right),
\end{equation}

We need to exhibit decay of the Lipschitz seminorm $\|\nabla u^\tau\|_{L^\infty}$. By the triangle inequality,
\begin{equation}
\|\nabla u^\tau\|_{L^\infty} \leq {\sigma}\|\nabla^{\otimes 2}\log(\mu^\tau)\|_{L^\infty} + \|\nab^{\otimes 2}\g\ast\mu^\tau\|_{L^\infty} = {\sigma}\|\nabla^{\otimes 2}\log(\mu^\tau)\|_{L^\infty} + \|\g\ast\nab^{\otimes 2}\mu^\tau\|_{L^\infty}.
\end{equation}
Assume $\mu^0$ is bounded from below, i.e. $\ka \coloneqq \inf_{\T^d} \mu^0 >0$. Note that $\int_{\T^d}\mu^0=1$ implies $\ka\leq 1$, since $\T^d$ has unit volume. Then $\mu^\tau\geq\ka$ uniformly in $\tau$ by \cref{remark:positiveness}, and we have by the chain rule that
\begin{align}
\|\nabla^{\otimes 2}\log(\mu^\tau)\|_{L^\infty} \leq \left\|\frac{\nabla^{\otimes 2}\mu^\tau}{\mu^\tau}\right\|_{L^\infty} + \left\| \frac{(\nabla\mu^\tau)^{\otimes 2}}{(\mu^\tau)^2}\right\|_{L^\infty} \lesssim \frac{\|\nabla^{\otimes 2}\mu^\tau\|_{L^\infty}}{\ka} + \frac{\|\nabla\mu^\tau\|_{L^\infty}^2}{\ka^2}. \label{eq:Gronvf1}
\end{align}
Since $\g$ is in $L^1$,
\begin{align}
	\|\g\ast\nab^{\otimes 2}\mu^t\|_{L^\infty} \lesssim \|\nabla^{\otimes 2}\mu^t\|_{L^\infty}. \label{eq:Gronvf2}
\end{align}
Combining \eqref{eq:Gronvf1}, \eqref{eq:Gronvf2}, we obtain
\begin{align}
\|\nab u^t\|_{L^\infty} \lesssim \paren*{1+\frac{{\sigma}}{\ka}}\|\nabla^{\otimes 2}\mu^\tau\|_{L^\infty} + \frac{{\sigma}\|\nabla\mu^\tau\|_{L^\infty}^2}{\ka^2}.
\end{align}
For $1\leq n\leq 2$, we may use estimate \eqref{eq:propnabns<} from \cref{prop:LpLq} to find
\begin{multline}\label{eq:nabus<}
\|\nab u^t\|_{L^\infty} \le  \paren*{1+\frac{\sigma}{\ka}}{\W_{2,\infty}(\|\mu^0\|_{L^{\infty}},\sigma^{-1}, \|\mu^0-1\|_{L^\infty}, \Fc_{\sigma}(\mu^0)) e^{-C\sigma t} \min(\sigma t,1)^{-1}} \\
+ \frac{\sigma}{\ka^2}\paren*{\W_{1,\infty}(\|\mu^0\|_{L^{\infty}},\sigma^{-1}, \|\mu^0-1\|_{L^\infty}, \Fc_{\sigma}(\mu^0)) e^{-C\sigma t} \min(\sigma t,1)^{-\frac12}}^2,
\end{multline}
if $d-2\leq s\leq d-1$, and estimate \eqref{eq:propnabns>} from \cref{prop:LpLq} to find
\begin{multline}\label{eq:nabus>}
\|\nab u^t\|_{L^\infty} \leq\paren*{1+\frac{\sigma}{\ka}}e^{-C\sigma t}\min(\sigma t,1)^{-1}\\
\times \paren*{1+C_\ep \min(\sigma t,1)^{-\ep}} \W_{2,\infty}(\|\mu^0\|_{L^\infty}, \|\mu^0\|_{\dot{H}^{\la_2}},\sigma^{-1}, \|\mu^0-1\|_{L^\infty}, \Fc_{\sigma}(\mu^0)) \\
 + \frac{\sigma}{\ka^2}\Bigg(e^{-C\sigma t}\min(\sigma t,1)^{-\frac12}\paren*{1+C_\ep \min(\sigma t,1)^{-\ep}} \W_{1,\infty}(\|\mu^0\|_{L^\infty}, \|\mu^0\|_{\dot{H}^{\la_2}},\sigma^{-1}, \|\mu^0-1\|_{L^\infty}, \Fc_{\sigma}(\mu^0))\Bigg)^2,
\end{multline}
if $d-1<s<d$ and where $\ep>0$ is arbitrary. By \cref{prop:lwp}, there is a time $T_0>0$, comparable to
\begin{equation}
\begin{cases}
\displaystyle\frac{\sigma}{\|\mu^0\|_{L^\infty}^2}, & {s<d-1} \\ 
\displaystyle \paren*{\frac{\sigma^{\frac{d}{2p}+\frac12}}{C_p \|\mu^0\|_{L^\infty}} }^{\frac{2p}{p-d}}, & {s=d-1} \\
\displaystyle\paren*{\frac{\sigma^{\frac{1+\d}{2}}}{C_\d\|\mu^0\|_{{W}^{2,\infty}}}}^{\frac{2}{1-\d}}, & {d-1<s<d},
\end{cases}
\end{equation}
for some $p\in (d,\infty)$ and $\d\in (s+d-1,1)$, such that $\|\mu\|_{C([0,T_0], W^{2,\infty})} \leq 2\|\mu^0\|_{W^{2,\infty}}$. We then divide the integration over the subintervals $[0,T_0]$ and $[T_0,t]$, assuming that $t\geq T_0$ without loss of generality. On $[0,T_0]$, we use the trivial estimate given by choice of $T_0$. On $[T_0,t]$, we use the decay estimates \eqref{eq:nabus<}, \eqref{eq:nabus>} and we obtain
\begin{multline}\label{eq:Antbnd}
    \int_0^\infty \|\nabla u^\tau\|_{L^\infty}d\tau \leq   2C_1 T_0\left(\left(1+\frac{\sigma}{\ka}\right)\|\mu^0\|_{W^{2,\infty}} + \frac{\sigma\|\mu^0\|_{W^{2,\infty}}^2}{\ka^2}\right)\\
   +C_1\int_{T_0}^\infty \Bigg(\Bigg(\paren*{1+\frac{\sigma}{\ka}}{\W_{2,\infty}(\|\mu^0\|_{L^{\infty}},\sigma^{-1}, \|\mu^0-1\|_{L^\infty}, \Fc_{\sigma}(\mu^0)) e^{-C\sigma \tau} \min(\sigma \tau,1)^{-1}} \\
+ \frac{\sigma}{\ka^2}\paren*{\W_{1,\infty}(\|\mu^0\|_{L^{\infty}},\sigma^{-1}, \|\mu^0-1\|_{L^\infty}, \Fc_{\sigma}(\mu^0)) e^{-C\sigma t} \min(\sigma \tau,1)^{-\frac12}}^2\Bigg)\indic_{d-2\leq s\leq d-1}\\
   +\Bigg(\paren*{1+\frac{\sigma}{\ka}}e^{-C\sigma \tau}\min(\sigma \tau,1)^{-1}\paren*{1+C_\ep \min(\sigma\tau,1)^{-\ep}}\\
\times  \W_{2,\infty}(\|\mu^0\|_{L^\infty}, \|\mu^0\|_{\dot{H}^{\la_2}},\sigma^{-1}, \|\mu^0-1\|_{L^\infty}, \Fc_{\sigma}(\mu^0)) \\
 + \frac{\sigma}{\ka^2}\Bigg(e^{-C\sigma \tau}\min(\sigma \tau,1)^{-\frac12}\paren*{1+C_\ep \min(\sigma \tau,1)^{-\ep}} \\
 \times\W_{1,\infty}(\|\mu^0\|_{L^\infty}, \|\mu^0\|_{\dot{H}^{\la_2}},\sigma^{-1}, \|\mu^0-1\|_{L^\infty}, \Fc_{\sigma}(\mu^0))\Bigg)^2\Bigg)\indic_{d-1<s<d}\Bigg) d\tau.
\end{multline}
Evidently, the integral over $[T_0,\infty)$ is finite. Applying this bound, we obtain the uniform-in-time estimate (written in compact form)
\begin{multline}
\sup_{t\geq 0}\Ec_N(f_N^t,\mu^t) \le \Ec_N(f_N^0,\mu^0)\Big(\W_1(\|\mu^0\|_{W^{2,\infty}},\sigma^{-1},\ka^{-1}, \|\mu^0-1\|_{L^\infty}, \Fc_{\sigma}(\mu^0))\indic_{d-2\leq s\leq d-1}\\
+ \W_2(\|\mu^0\|_{W^{2,\infty}}, \sigma^{-1},\ka^{-1},\|\mu^0-1\|_{L^\infty}, \Fc_{\sigma}(\mu^0))\indic_{d-1<s<d}\Big),
\end{multline}
where $\W_1$ and $\W_2$ are continuous, nondecreasing in their arguments. This completes the proof of \cref{thm:main}.

\section{Application to $\dot{W}^{-1,\infty}$ kernels}\label{sec:JW}
In this final section of the paper, we show how our decay estimates approach can be combined with the relative entropy method of \cite{JW2018} to give a simple proof of uniform-in-time propagation of chaos for mean-field McKean-Vlasov systems
\begin{equation}\label{eq:JWSDE}
\begin{cases}
dx_{i}^t = \displaystyle\frac{1}{N}\sum_{1\leq j\leq N : j\neq i} \k(x_i^t-x_j^t)dt + \sqrt{2\sigma}dW_i^t\\
x_i^t|_{t=0} = x_i^0
\end{cases}\qquad i\in\{1,\ldots,N\}.
\end{equation}
The vector field/kernel $\k:\T^d\setminus\{0\}\rightarrow \T^d$ is assumed to satisfy the following conditions:
\begin{enumerate}[(i)]
\item\label{ass:kL1} $\k\in L^1(\T^d)$,
\item\label{ass:kdiv} $\div \k = 0$ in the sense of distributions,\footnote{This condition is not strictly necessary; it would suffice to have have $\div\k \in \dot{W}^{-1,\infty}$. But we impose it to simplify the presentation.}
\item\label{ass:kW-1inf} $\k^{\al} = \p_{\be}V^{\al \be}$ for an $L^\infty(\T^d)$ matrix-valued field $(V^{\al\be})_{\al,\be=1}^d$.
\end{enumerate}
Note that $\k$ is no longer assumed to be potential (i.e., $\k=\nabla\g$ for some $\g$). We remark that the last condition (\ref{ass:kW-1inf}) amounts to requiring that $\k$ belong to the negative-order Sobolev space $\dot{W}^{-1,\infty}(\T^d)$. A sufficient condition is that $\k\in L^{d,\infty}(\T^d)$ (the weak $L^d$ space). A model example is the 2D periodic Biot-Savart kernel (i.e., $\k=\M\nabla\g$ for a $90^\circ$ rotation matrix $\M$ and $\g$ the periodic Coulomb potential), as explained in \cite[p. 531]{JW2018}.

Analogous to \eqref{eq:lim}, the mean-field limit of the system \eqref{eq:JWSDE} is determined by the solution of the Cauchy problem
\begin{equation}\label{eq:MFlimk}
\begin{cases}
\p_t\mu = -\div(\mu\k\ast\mu) + \sigma\D\mu\\
\mu|_{t=0} = \mu^0
\end{cases}
\qquad (t,x) \in [0,\infty)\times\T^d.
\end{equation}
By assumption (\ref{ass:kdiv}), the vector field $u\coloneqq \k\ast\mu$ is divergence-free. Hence, equation \eqref{eq:MFlimk} is a transport-diffusion equation with a divergence-free vector field. Such an equation fits into the framework of \cref{sec:WP} for the range $d-2\leq s<d-1$. Indeed, an examination of the proofs of \cref{prop:lwp} and \cref{lem:derdcys<} reveals that we only used that $\nabla\g\in L^1$, no other specific structure on $\g$. Thus, one may repeat the aforementioned proofs with $\nabla\g$ replaced by $\k\in L^1$ and no other change. Since $\k$ is divergence-free, the equation conserves average/mass and the lower and upper bounds for the initial data as in \Cref{remark:massconserv,remark:positiveness}, respectively. In particular, given $\mu^0 \in L^\infty(\T^d) \cap \dot{W}^{\al,p}(\T^d)$ for any $\al\geq 0$ and $1\leq p<\infty$,\footnote{If $\al$ is integral, then $p=\infty$ is allowed.} there is a unique global solution $\mu\in C([0,\infty), L^\infty \cap \dot{W}^{\al,p})$.  If $\mu$ is a zero-mean solution, then all the $L^p$ norms are nonincreasing, and for any $1\leq q\leq\infty$, we have the decay estimates
\begin{equation}\label{eq:JWdcay}
\forall t>0, \ n\geq 0,  \qquad \|\nabla^{\otimes n}\mu^t\|_{L^q} \leq \W_{n,q}(\|\mu^0\|_{L^\infty},\sigma^{-1})\|\mu^0\|_{L^1} e^{-C\sigma t} (\sigma t)^{-\frac{n}{2}-\frac{d}{2}\left(1-\frac{1}{q}\right)}.
\end{equation}

\begin{thm}\label{thm:JWunif}
Let $d\geq 1$, $\k$ be a kernel satisfying assumptions (\ref{ass:kL1}), (\ref{ass:kdiv}), (\ref{ass:kW-1inf}). Let $f_N \in L^\infty([0,\infty), L^1(\T^d))$ be an entropy solution to the Liouville equation \eqref{eq:liou},\footnote{The existence of such a solution is sketched in \cite[Section 1.5]{JW2018}.} and let $\mu \in L^\infty([0,\infty), \mathcal{P}(\T^d) \cap W^{2,\infty}(\T^d))$ be a solution to equation \eqref{eq:MFlimk} with $\inf_{\T^d}\mu^0>0$. Then
\begin{equation}\label{eq:JWunif}
\forall t\geq 0, \qquad H_N(f_N^t \vert (\mu^t)^{\otimes N}) \leq \paren*{H_N(f_N^0\vert(\mu^0)^{\otimes N}) + \frac{\mathscr{C}^t}{N}}e^{\mathscr{C}^t},
\end{equation}
where $\mathscr{C}: [0,\infty)\rightarrow [0,\infty)$ is a continuous, nondecreasing function such that $\mathscr{C}^0= 0$ and $\sup_{t\geq 0}\mathscr{C}^t <\infty$. $\mathscr{C}$ depends on $d,\sigma,\|\k\|_{\dot{W}^{-1,\infty}}, \|\mu^0\|_{W^{2,\infty}}, \inf_{\T^d}\mu^0$.
\end{thm}

\begin{remark}
By \cref{rem:MFEcoer}, such a bound implies propagation of chaos with an explicit rate, though one not expected to be optimal.
\end{remark}

\begin{remark}
By exploiting the Fisher information as in \cite{GlBM2021} (see the next paragraph), one can obtain a factor of the form $e^{-C_2t}$ for large $t$ in front of the term $H_{N}(f_N^0\vert (\mu^0)^{\otimes N})$ in \eqref{eq:JWunif}.
\end{remark}

As commented in \cref{ssec:introprob}, Guillin et al. \cite{GlBM2021} have previously shown a comparable result to \cref{thm:JWunif}. Their proof again is a refinement of \cite{JW2018} and uses techniques more common in the probability community (e.g., uniform-in-time log Sobolev inequalities to exploit the Fisher information in the entropy dissipation), as opposed to the ``PDE approach'' of the present paper. They also need uniform-in-time bounds on the $W^{2,\infty}$ norm of $\mu^t$, which they obtain through standard energy methods, akin to the proof of \cref{lem:intmd}, and Sobolev embedding. Relaxation estimates neither enter into their proof nor are established. Therefore, our proof presented below should be taken as an alternative to their work. In particular, our proof demonstrates there is no need to use the Fisher information. 

We will only sketch the proof of \cref{thm:JWunif}, in particular focusing on the main steps and how decay estimates allow to obtain a uniform-in-time result. For justification of the differential identities below---especially, the consideration needed given that $f_N$ is only a weak solution---we refer to the original article of Jabin-Wang \cite[Section 2]{JW2018}.

One may verify that if $\mu$ is a solution to \eqref{eq:MFlimk}, then $\mu^{\otimes N}$ is a solution to the Cauchy problem
\begin{equation}\label{eq:muN}
\begin{cases}
\displaystyle \p_t \bar{f}_N  + \sum_{i=1}^N (\k\ast\mu)(x_i)\cdot\nabla_{x_i}\bar{f}_N = \sigma\sum_{i=1}^N \D_{x_i}\bar{f}_N  \\
\bar{f}_N|_{t=0} = (\mu^0)^{\otimes N},
\end{cases}
\qquad (t,\ux_N) \in [0,\infty)\times (\T^d)^N.
\end{equation}
Using this equation, one can show that (see \cite[Lemma 2]{JW2018})
\begin{multline}\label{eq:dtHN}
\frac{d}{dt}H_N(f_N^t \vert (\mu^t)^{\otimes N}) \leq -\frac{1}{N}\sum_{i=1}^N \int_{(\T^d)^N}\paren*{\k\ast(\mu_{\ux_N}-\mu)}(x_i)\cdot\nabla_{x_i}\log\paren*{(\mu^t)^{\otimes N}}df_N^t(\ux_N) \\
-\frac{\sigma}{N}\int_{(\T^d)^N}\left|\nabla_{x_i}\log\left(\frac{f_N^t}{(\mu^t)^{\otimes N}}\right)\right|^2 df_N^t,
\end{multline}
where $\mu_{\ux_N} \coloneqq \frac{1}{N}\sum_{j=1}^N \d_{x_j}$ and we set $\k(0)\coloneqq 0$, which is harmless given we are modifying on a measure zero set. By assumption (\ref{ass:kW-1inf}), there exists an $L^\infty$ matrix field $(V^{\al\be})_{\al,\be=1}^d$ such that $\k^{\al} = \p_{\be}V^{\al\be}$. Integrating by parts in the variable $x_i^\be$ and performing some manipulations, one arrives at the inequality
\begin{multline}\label{eq:dtHNpreCoL}
\frac{d}{dt}H_N(f_N^t \vert (\mu^t)^{\otimes N}) \leq \frac{1}{N\sigma}\sum_{i=1}^N \int_{(\T^d)^N}\left|\paren*{V\ast (\mu_{\ux_N}-\mu^t)}(x_i)\right|^2 \left|\nabla\log(\mu^t)(x_i)\right|^2 df_N^t(\ux_N) \\
+ \frac{1}{N}\sum_{i=1}^N \int_{(\T^d)^N}\paren*{V\ast(\mu_{\ux_N}-\mu)}(x_i) : \frac{\nabla^{\otimes 2}\mu^t(x_i)}{\mu^t(x_i)} df_N^t(\ux_N).
\end{multline}

To control the right-hand side, we recall the following convexity inequality (see \cite[Lemma 1]{JW2018}), sometimes called the Donsker-Varadhan lemma, which allows one to change the $N$-particle law with respect to which we compute expectations. This is useful because we do not have much information about the $N$-particle law $f_N$, as opposed to the tensorized mean-field law $(\mu)^{\otimes N}$.

\begin{lemma}\label{lem:CoL}
Let $\mu_N,\nu_N \in \P((\T^d)^N)$, and let $\Phi\in L^\infty((\T^d)^N)$. Then $\forall \eta>0$,
\begin{equation}
\int_{(\T^d)^N}\Phi d\mu_N \leq \frac{1}{\eta}\Bigg(H_N(\mu_N\vert\nu_N) + \frac{1}{N}\log\paren*{\int_{(\T^d)^N}e^{\eta N\Phi}d\nu_N}\Bigg).
\end{equation}
\end{lemma}

Applying \cref{lem:CoL} to each of the two terms in the right-hand side of \eqref{eq:dtHNpreCoL}, we obtain
\begin{multline}\label{eq:dtHNrhs}
\frac{d}{dt}H_N(f_N^t \vert (\mu^t)^{\otimes N}) \leq \paren*{\frac{\|\nabla \log \mu^t\|_{L^\infty}^2}{\sigma\eta}+\frac{1}{\eta'}}H_N(f_N^t \vert (\mu^t)^{\otimes N})\\
 + \frac{\|\nabla \log \mu^t\|_{L^\infty}^2}{N^2\sigma\eta}\sum_{\al,\be=1}^d \sum_{i=1}^N \log\paren*{\int_{(\T^d)^N} \exp\paren*{\eta N\left|\paren*{V^{\al\be}\ast(\mu_{\ux_N}-\mu^t}(x_i)\right|^2}d(\mu^t)^{\otimes N}(\ux_N)} \\
 + \frac{1}{N\eta'}\log\paren*{\int_{(\T^d)^N}\exp\paren*{\eta'\sum_{i=1}^N \paren*{V\ast (\mu_{\ux_N}-\mu^t)}(x_i) : \frac{\nabla^{\otimes 2}\mu^t(x_i)}{\mu^t(x_i)}} d(\mu^t)^{\otimes N}(\ux_N)},
\end{multline}
where the value of the parameters $\eta,\eta'>0$ will be specified momentarily. Note that by symmetry, the integral in the second line is independent of the index $i$. 

To close the estimate for the relative entropy, we now recall two functional inequalities. The first is a law of large numbers at exponential scale. For a proof, see \cite[Section 4]{JW2018}; but note the result is a consequence of classical exponential inequalities for sums of random vectors \cite{Prokhorov1968, Yurinskii1976}. 

\begin{prop}\label{prop:ExpLLN}
There exist constants $C_1,C_2>0$ such that for any $\phi\in L^\infty(\T^d)$ with $\|\phi\|_{L^\infty} \leq 1$ and any probability measure $\mu$ on $\T^d$,
\begin{equation}
\int_{(\T^d)^N}\exp\paren*{\frac{N}{C_1}\left|\int_{\T^d}\phi(x)d\paren*{\mu_{\ux_N}-\mu}(x)\right|^2 } d\mu^{\otimes N}(\ux_N) \leq C_2.
\end{equation}
\end{prop}

The next inequality is a large deviation estimate, which is proved in \cite[Theorem 4]{JW2018}.\footnote{As noted by Jabin-Wang, this estimate and more would follow from classical large deviations work \cite{BaB1990}, which in turn builds on \cite{Bolthausen1986}, if $\phi$ were continuous.} A much simpler probabilistic proof of this estimate has been given in \cite[Section 5]{LLN2020}.

\begin{prop}\label{prop:LD}
Let $\mu$ be a probability density on $\T^d$. Suppose that $\phi \in L^\infty((\T^d)^2)$ satisfies
\begin{equation}\label{eq:phican}
\forall x\in \T^d, \quad \int_{\T^d}\phi(x,z)d\mu(z)= 0 \quad \text{and} \quad \forall z\in\T^d, \quad \int_{\T^d}\phi(x,z)d\mu(x) = 0.
\end{equation}
Then there is a universal constant $C_3>0$ such that if $\sqrt{C_3}\|\phi\|_{L^\infty}<1$, then
\begin{equation}
\int_{(\T^d)^N}\exp\paren*{N \int_{(\T^d)^2}\phi(x,z)d\mu_{\ux_N}^{\otimes 2}(x,z)} d\mu^{\otimes N}(\ux_N) \leq \frac{2}{1 - C_3\|\phi\|_{L^\infty}^2}.
\end{equation}
\end{prop}

Let us now see how to use \Cref{prop:ExpLLN,prop:LD} to complete the proof of the estimate for the evolution of the relative entropy. Let $C_1,C_2$ be the constants in the statement of \cref{prop:ExpLLN}. For $1\leq \al,\be\leq d$, we set
\begin{equation}
\phi(x) \coloneqq \sqrt{\eta C_1}V^{\al\be}(x_1-x), \qquad \forall x\in\T^d,
\end{equation}
so that
\begin{equation}
\exp\paren*{\eta N\left|\paren*{V^{\al\be}\ast(\mu_{\ux_N}-\mu^t}(x_i)\right|^2} = \exp\paren*{\frac{N}{C_1}\left|\int_{\T^d}\phi(x)d(\mu_{\ux_N}-\mu^t)(x)\right|^2 }.
\end{equation}
We choose $\eta>0$ sufficiently small so that $\sqrt{\eta C_1}\max_{\al,\be} \|V^{\al\be}\|_{L^\infty}= 1$, which ensures that $\|\phi\|_{L^\infty}\leq 1$. Applying \cref{prop:ExpLLN} pointwise in $t$, we obtain that
\begin{multline}\label{eq:JWExpLLNapp}
 \frac{\|\nabla \log \mu^t\|_{L^\infty}^2}{N^2\sigma\eta}\sum_{\al,\be=1}^d \sum_{i=1}^N \log\paren*{\int_{(\T^d)^N} \exp\paren*{\eta N\left|\paren*{V^{\al\be}\ast(\mu_{\ux_N}-\mu^t}(x_i)\right|^2}d(\mu^t)^{\otimes N}}\\
 \leq \frac{d^2\|\nabla \log \mu^t\|_{L^\infty}^2C_1 \|V\|_{L^\infty}^2(\log C_2)}{N\sigma}.
\end{multline}

Next, set
\begin{equation}
\phi(x,z) \coloneqq \eta'\paren*{V(x-z) - V\ast\mu^t(x)} : \frac{\nabla^{\otimes 2}\mu^t(x)}{\mu^t(x)}, \qquad \forall x,z\in\T^d.
\end{equation}
Using that $\k^\al=\p_\be V^{\al\be}$ is divergence-free, one checks that $\phi$ satisfies condition \eqref{eq:phican}. Now
\begin{equation}
\eta'\sum_{i=1}^N \paren*{V\ast (\mu_{\ux_N}-\mu^t)}(x_i) : \frac{\nabla^{\otimes 2}\mu^t(x_i)}{\mu^t(x_i)} = N\int_{(\T^d)^2} \phi(x,z)d(\mu_{\ux_N})^{\otimes 2}(x,z).
\end{equation}
We choose $\eta'>0$ to satisfy
\begin{equation}
\sqrt{C_3}\|\phi\|_{L^\infty} \leq 2\sqrt{C_3}\eta' \|V\|_{L^\infty} \frac{\|\nabla^{\otimes 2}\mu^t\|_{L^\infty}}{\inf\mu^t} = \frac12,
\end{equation}
so that by applying \cref{prop:LD} pointwise in $t$,
\begin{multline}\label{eq:JWLDapp}
\frac{1}{N\eta'}\log\paren*{\int_{(\T^d)^N}\exp\paren*{\eta'\sum_{i=1}^N \paren*{V\ast (\mu_{\ux_N}-\mu^t)}(x_i) : \frac{\nabla^{\otimes 2}\mu^t(x_i)}{\mu^t(x_i)}} d(\mu^t)^{\otimes N}(\ux_N)} \\
\leq \frac{4\sqrt{C_3}\|V\|_{L^\infty}\|\nabla^{\otimes 2}\mu^t\|_{L^\infty}(\log 4)}{N\inf\mu^t}.
\end{multline}
Applying the estimates \eqref{eq:JWExpLLNapp}, \eqref{eq:JWLDapp} to the right-hand side of \eqref{eq:dtHNrhs} and substituting in our choices for $\eta,\eta'$, we find
\begin{multline}\label{eq:dtHNpen}
\frac{d}{dt}H_N(f_N^t \vert (\mu^t)^{\otimes N}) \leq \paren*{\frac{\|\nabla \log \mu^t\|_{L^\infty}^2 C_1\|V\|_{L^\infty}^2}{\sigma}+\frac{4\sqrt{C_3}\|V\|_{L^\infty}\|\nabla^{\otimes 2}\mu^t\|_{L^\infty}}{\inf\mu^t}}H_N(f_N^t \vert (\mu^t)^{\otimes N})\\
+ \frac{d^2\|\nabla \log \mu^t\|_{L^\infty}^2C_1 \|V\|_{L^\infty}^2(\log C_2)}{N\sigma } + \frac{4\sqrt{C_3}\|V\|_{L^\infty}\|\nabla^{\otimes 2}\mu^t\|_{L^\infty}(\log 4)}{N\inf\mu^t}.
\end{multline}
By the local well-posedness theory for \eqref{eq:MFlimk}, there exists a time $T_0>0$ comparable to $\frac{\sigma}{\|\mu^0\|_{W^{2,\infty}}^2}$, such that $\|\mu\|_{C([0,T_0], W^{2,\infty})} \leq 2 \|\mu^0\|_{W^{2,\infty}}$. Since $\inf\mu^t \geq \inf\mu^0$, we have that
\begin{align}
\forall t\in[0,T_0], \qquad \|\nabla\log\mu^t\|_{L^\infty} \leq \frac{2 \|\mu^0\|_{W^{2,\infty}}}{\inf\mu^0}, \\
\forall t\geq T_0, \qquad \|\nabla\log\mu^t\|_{L^\infty} \leq \frac{\W_{1,\infty}(\|\mu^0\|_{L^\infty},\sigma^{-1})}{\inf\mu^0} e^{-C_4\sigma t} (\sigma t)^{-\frac{d+1}{2}},\\
\forall t\geq T_0, \qquad \|\nabla^{\otimes 2}\mu^t\|_{L^\infty} \leq \W_{2,\infty}(\|\mu^0\|_{L^\infty},\sigma^{-1}) e^{-C_4\sigma t} (\sigma t)^{-\frac{d+2}{2}},
\end{align}
where the second and third assertions follows from use of \eqref{eq:JWdcay}. Integrating both sides of the differential inequality \eqref{eq:dtHNpen}, then applying the Gr\"{o}nwall-Bellman lemma, we ultimately find that
\begin{align}
H_N(f_N^t \vert (\mu^t)^{\otimes N}) \leq e^{\mathscr{C}^t}\left(H_N(f_N^0 \vert (\mu^0)^{\otimes N}) + \frac{\mathscr{C}^t}{N}\right),
\end{align}
where
\begin{multline}
\mathscr{C}^t \coloneqq \frac{C_1\min(t,T_0)}{\sigma}\paren*{\frac{\|V\|_{L^\infty}\|\mu^0\|_{W^{2,\infty}}}{\inf\mu^0}}^2 + \frac{\|V\|_{L^\infty}\|\mu^0\|_{W^{2,\infty}}\min(t,T_0)}{\inf\mu^0} \\
+ \Bigg(\frac{C_1}{\sigma^2}\paren*{\frac{\|V\|_{L^\infty}}{\inf\mu^0}}^2\paren*{\frac{\W_{1,\infty}(\|\mu^0\|_{L^\infty},\sigma^{-1}) e^{-C_2\sigma T_0}}{(\sigma T_0)^{\frac{d+1}{2}}} }^2 + \frac{C_1\|V\|_{L^\infty} \W_{2,\infty}(\|\mu^0\|_{L^\infty},\sigma^{-1}) e^{-C_2\sigma T_0}}{\sigma (\sigma T_0)^{\frac{d+2}{2}}\inf\mu^0}\Bigg)\indic_{t\geq T_0}.
\end{multline}
By inspection, one checks that $\mathscr{C}^0=0$ and $\sup_{t\geq 0}\mathscr{C}^t <\infty$. This completes the proof of \cref{thm:JWunif}.

\appendix
\section{Proof of \cref{lem:entsol}}\label{app}
We give here the proof of \cref{lem:entsol} on the existence of entropy solutions to the Liouville equation \eqref{eq:liou}.

Let $\chi\in C_c^\infty$ be a bump function with values between in $0$ and $1$ which is identically $1$ on $B(0,\frac{1}{16})$ and zero outside $B(0,\frac18)$. Given $\ep>0$, set $\chi_\ep(x)\coloneqq \ep^{-d}\chi(x/\ep)$ and define the truncated potential $\g_{(\ep)} \coloneqq \g_E(1-\chi_\ep) + (\g-\g_E)$.\footnote{$\g_{(\ep)}$ should not be confused with the truncated potential $\g_\ep$ from \cref{sec:PRP}.} Evidently, $\g_{(\ep)} \in C^\infty(\T^d)$ coincides with $\g$ if $|x|\geq \frac{\ep}{8}$.

Consider the Cauchy problem for the regularized Liouville equation,
\begin{equation}\label{eq:lioureg}
\begin{cases}
\p_t f_{N,\ep} = \displaystyle-\sum_{i=1}^N \div_{x_i}\paren*{f_{N,\ep}\frac{1}{N}\sum_{1\leq j\leq N : j \neq i}\M\nabla\gep(x_i-x_j)} + \sigma\sum_{i=1}^N\D_{x_i}f_{N,\ep} \\
f_{N,\ep}|_{t=0} = f_N^0.
\end{cases}
\end{equation}
By standard well-posedness theory for transport-diffusion equations (e.g., see \cite[Section 3.4]{BCD2011}), \eqref{eq:lioureg} has a solution $f_{N,\ep}\in L^\infty([0,\infty), \mathcal{P}((\T^d)^N))$, which is $C^\infty$ for positive times. Letting $G_{N,\ep}$ denote the analogue of $G_N$ from \cref{def:entsol} with $\g$ replaced by $\gep$, the reader may verify the entropy bound
\begin{multline}\label{eq:regentbnd}
\forall t\geq 0, \qquad \int_{(\T^d)^N}\log\paren*{\frac{f_{N,\ep}^t}{G_{N,\ep}}} df_{N,\ep}^t + \sigma\sum_{i=1}^N \int_0^t \int_{(\T^d)^N}\left|\nabla_{x_i}\log\paren*{\frac{f_{N,\ep}^\tau}{G_{N,\ep}}} \right|^2 df_{N,\ep}^\tau \\
\leq \int_{(\T^d)^N}\log\paren*{\frac{f_N^0}{G_{N,\ep}}}df_N^0.
\end{multline}
Since $\g_E$ is decreasing and by the properties of $\chi$, the preceding right-hand is $\leq$ $\int_{(\T^d)^N}\log(\frac{f_N^0}{G_{N}})df_N^0$.

From the relation \eqref{eq:ggE},
\begin{equation}
 -\log(G_{N,\ep}) = \frac{1}{2N\sigma}\sum_{1\leq i\neq j\leq N}\gep(x_i-x_j) \geq   -\frac{N}{\sigma}\paren*{\sup_{|x|\geq \frac14}|\g(x)| + \sup_{|x|\leq\frac14}|\g_E(x)-\g(x)| } > -\infty,
\end{equation}
it follows from \eqref{eq:regentbnd} that
\begin{equation}\label{eq:LlogLfNep}
\sup_{\ep>0} \sup_{t\geq 0} \int_{(\T^d)^N}\log f_{N,\ep}^t df_{N,\ep}^t < \int_{(\T^d)^N}\log\paren*{\frac{f_N^0}{G_{N}}}df_N^0 + \frac{N}{\sigma}\paren*{\sup_{|x|\geq \frac14}|\g(x)| + \sup_{|x|\leq\frac14}|\g_E(x)-\g(x)| }.
\end{equation}
Hence, by the Dunford-Pettis theorem, after passing to a subsequence, $f_{N,\ep}$ converges weakly in $L^1([0,\infty), L^1((\T^d)^N))$ to an element $f_N \in L^1([0,\infty), L^1((\T^d)^N))$. It is easy to check that $f_N^t\geq 0$ for a.e. $(t,x)$,\footnote{By redefinition of $f_N$ on a set of measure zero, we may assume without loss of generality that $f_N\geq 0$ everywhere.}. We in fact have that $f\in L^\infty([0,\infty), L\log L((\T^d)^N))$. Indeed, let $\rho\in L^1([0,\infty))$ be a temporal function. Then for any $\varphi \in C((\T^d)^N)$, we have
\begin{equation}
\lim_{\ep\rightarrow 0}\int_{0}^\infty \rho(t)\paren*{\int_{(\T^d)^N}\varphi df_{N,\ep}^t - \int_{(\T^d)^N}e^{\varphi}}dt = \int_{0}^\infty \rho(t)\paren*{\int_{(\T^d)^N}\varphi df_N^t - \int_{(\T^d)^N}e^{\varphi}}dt.
\end{equation}
Now by the variational formulation of entropy, if $\varphi\geq 0$, 
\begin{align}
\int_{0}^\infty \rho(t)\paren*{\int_{(\T^d)^N}\varphi df_{N,\ep}^t - \int_{(\T^d)^N}e^{\varphi}}dt &\leq \int_0^\infty \rho(t)\paren*{-1+\int_{(\T^d)^N}\log f_{N,\ep}^t df_{N,\ep}^t} dt \nn\\
&\leq \paren*{-1+\sup_{\ep'>0} \sup_{\tau\geq 0} \int_{(\T^d)^N}\log f_{N,\ep'}^\tau df_{N,\ep}^\tau} \|\rho\|_{L^1}.
\end{align}
Since $\rho$ was arbitrary, the preceding implies that
\begin{equation}
\text{a.e.} \ t\geq 0, \qquad \paren*{\int_{(\T^d)^N}\varphi df_N^t - \int_{(\T^d)^N}e^{\varphi}} \leq \paren*{-1+\sup_{\ep'>0} \sup_{\tau\geq 0} \int_{(\T^d)^N}\log f_{N,\ep'}^\tau df_{N,\ep}^\tau}.
\end{equation}
Taking the $\sup$ over $\varphi\in C((\T^d)^N)$ in the left-hand side and again using the variational formulation of entropy, we see that
\begin{equation}\label{eq:FatouLlogL}
\text{a.e.} \ t\geq 0, \qquad \int_{(\T^d)^N}\log f_{N}^tdf_N^t \leq \sup_{\ep'>0} \sup_{\tau\geq 0} \int_{(\T^d)^N}\log f_{N,\ep'}^\tau df_{N,\ep}^\tau,
\end{equation}
the right-hand side of which is bounded by the right-hand side of \eqref{eq:LlogLfNep}.

Additionally, for any temporal test function $\rho\in C^\infty([0,\infty))$, we have by the weak convergence with spacetime test function $(x,t)\mapsto \rho(t)$,
\begin{equation}
\int_0^\infty \rho(t) dt=\lim_{\ep\rightarrow 0} \int_0^\infty \int_{(\T^d)^N}\rho(t) df_{N,\ep}^t dt = \int_0^\infty\rho(t)\int_{(\T^d)^N} df_N^t dt.
\end{equation}
Since $\rho$ was arbitrary, this implies that $\int_{(\T^d)^N}df_N^t = 1$ for a.e. $t$.

Now for any fixed $\ep_0 \in (0,1]$, we have that
\begin{equation}
\forall \ep\leq \frac{\ep_0}{2}, \qquad -\int_{(\T^d)^N}\log G_{N,\ep_0} df_{N,\ep}^t \leq -\int_{(\T^d)^N}\log G_{N,\ep} df_{N,\ep}^t.
\end{equation}
Note that $\log G_{N,\ep_0}$ is $C^\infty$ and so may be taken as a test function. Let $\rho\in C^\infty([0,\infty))$ be an arbitrary test function of time, the weak convergence of $f_{N,\ep}$ to $f_N$ implies
\begin{equation}
\lim_{\ep\rightarrow 0} \int_{0}^\infty\int_{(\T^d)^N}\rho(t) \log G_{N,\ep_0}df_{N,\ep}^t dt = \int_{0}^\infty\int_{(\T^d)^N}\rho(t) \log G_{N,\ep_0}df_{N}^t dt.
\end{equation}
So by monotone convergence,
\begin{align}
 \int_{0}^\infty\int_{(\T^d)^N}\rho(t)\log G_{N} df_{N}^tdt &= \lim_{\ep_0\rightarrow 0} -\int_{0}^\infty\int_{(\T^d)^N}\rho(t)\log G_{N,\ep_0} df_{N}^tdt \nn\\
 &\leq \liminf_{\ep\rightarrow 0}-\int_0^\infty\int_{(\T^d)^N}\rho(t)\log G_{N,\ep} df_{N,\ep}^tdt.
\end{align}
Since $\rho$ was arbitrary, the preceding inequality together with \eqref{eq:FatouLlogL} implies that
\begin{equation}
\text{a.e.} \ t>0, \qquad  \int_{(\T^d)^N} \log\paren*{\frac{f_{N}^t}{G_{N}}}df_{N}^t\leq \liminf_{\ep\rightarrow 0}\int_{(\T^d)^N} \log\paren*{\frac{f_{N,\ep}^t}{G_{N,\ep}}}df_{N,\ep}^t.
\end{equation}

Next, observe that $\frac{f_{N,\ep}}{G_{N,\ep}}$ converges in the sense of spacetime distributions to $\frac{f_N}{G_N}$. Note also from the monotonicity of $\g_{E}$ and the properties of $\chi$ that
\begin{equation}
\forall \ep<\frac{\ep_0}{2}, \qquad \int_{(\T^d)^N} \frac{|\nabla\frac{f_{N,\ep}}{G_{N,\ep}}|^2}{\frac{f_{N,\ep}}{G_{N,\ep}}}dG_{N,\ep_0} \leq \int_{(\T^d)^N} \frac{|\nabla\frac{f_{N,\ep}}{G_{N,\ep}}|^2}{\frac{f_{N,\ep}}{G_{N,\ep}}}dG_{N,\ep}
\end{equation}
By weak lower semicontinuity, we have
\begin{equation}
\int_0^T \int_{(\T^d)^N} \frac{|\nabla\frac{f_{N}}{G_{N}}|^2}{\frac{f_{N}}{G_{N}}}dG_{N,\ep_0}dt \leq \liminf_{\ep\rightarrow 0}\int_0^T \int_{(\T^d)^N} \frac{|\nabla \frac{f_{N,\ep}}{G_{N,\ep}}|^2}{\frac{f_{N,\ep}}{G_{N,\ep}}}dG_{N,\ep_0}dt.
\end{equation}
By another application of monotone convergence theorem, it follows that
\begin{align}
\sigma\int_0^T \int_{(\T^d)^N} \frac{|\nabla\frac{f_{N}}{G_{N}}|^2}{\frac{f_{N}}{G_{N}}}dG_{N}dt &= \sigma\lim_{\ep_0\rightarrow 0} \int_0^T \int_{(\T^d)^N} \frac{|\nabla\frac{f_{N}}{G_{N}}|^2}{\frac{f_{N}}{G_{N}}}dG_{N,\ep_0}dt \nn\\
&\leq \int_{(\T^d)^N}\log(\frac{f_N^0}{G_{N}})df_N^0 - \int_{(\T^d)^N} \log\paren*{\frac{f_{N}^T}{G_{N}}}df_{N}^T.
\end{align}

Finally, we check that the limit $f_{N}$ satisfies the original Liouville equation \eqref{eq:liou} in the distributional sense on $[0,\infty)\times (\T^d)^N$. Observe that for any test function $\varphi\in C^\infty((\T^d)^N)$ and a.e. $t$, 
\begin{equation}
\int_{(\T^d)^N}\varphi\nabla\log\paren*{\frac{f_N^t}{G_N}}df_N^t
\end{equation}
is absolutely convergent. Indeed, this follows since for a.e. $t$, $\nabla\log(\frac{f_N^t}{G_N})\sqrt{f_N^t}\in L^2$ and therefore by Cauchy-Schwarz,
\begin{align}
\int_{(\T^d)^N}\left|\varphi\nabla\log\paren*{\frac{f_N^t}{G_N}}\right| df_N^t &\leq \paren*{\int_{(\T^d)^N}|\varphi|^2 df_N^t}^{1/2} \paren*{\int_{(\T^d)^N}\left|\nabla\log(\frac{f_N^t}{G_N})\right|^2 df_N^t}^{1/2} \nn\\
&\leq \|\varphi\|_{L^\infty}\paren*{\int_{(\T^d)^N}\left|\nabla\log(\frac{f_N^t}{G_N})\right|^2 df_N^t}^{1/2}, \label{eq:CSgradlog}
\end{align}
where we use that $f_N^t$ is a probability density to obtain the final line. We now want to show that for any spacetime test function $\psi$,
\begin{equation}
\lim_{\ep\rightarrow 0} \int_0^\infty\int_{(\T^d)^N} \psi^t \nabla\log\paren*{\frac{f_{N,\ep}^t}{G_{N,\ep}}}df_{N,\ep}^tdt = \int_0^\infty\int_{(\T^d)^N}\psi^t\nabla\log\paren*{\frac{f_N^t}{G_N}}df_N^t dt.
\end{equation}
Let $M\gg 1$ and decompose
\begin{multline}
\int_{(\T^d)^N} \psi^t \nabla\log\paren*{\frac{f_{N,\ep}^t}{G_{N,\ep}}}df_{N,\ep}^t = \int_{(\T^d)^N} \psi^t \indic_{G_{N}^{-1}\geq M} \nabla\log\paren*{\frac{f_{N,\ep}^t}{G_{N,\ep}}}df_{N,\ep}^t \\
+\int_{(\T^d)^N} \psi^t \indic_{G_{N}^{-1}< M} \nabla\log\paren*{\frac{f_{N,\ep}^t}{G_{N,\ep}}}df_{N,\ep}^t .
\end{multline}
Observe that
\begin{equation}
r\coloneqq \inf_{i\neq j}\{ |x_i - x_j| : G_N^{-1}(\ux_N)\leq M\} > 0.
\end{equation}
Hence,  for all $\ep\leq r$, $\nabla\log(G_{N,\ep}^{-1})\indic_{G_{N}^{-1}\leq M} = \nabla\log(G_{N}^{-1})\indic_{G_{N}^{-1}\leq M}$. So, by weak convergence,
\begin{equation}
\lim_{\ep\rightarrow 0}\int_0^\infty \int_{(\T^d)^N} \psi^t \indic_{G_{N}^{-1}< M} \nabla\log\paren*{\frac{f_{N,\ep}^t}{G_{N,\ep}}}df_{N,\ep}^t = \int_0^\infty \int_{(\T^d)^N} \psi^t \indic_{G_{N}^{-1}< M} \nabla\log\paren*{\frac{f_{N}^t}{G_{N}}}df_{N}^t
\end{equation}
Finally, arguing similar to \eqref{eq:CSgradlog},
\begin{align}
&\left|\int_{(\T^d)^N} \psi^t \indic_{G_{N}^{-1}\geq M} \nabla\log\paren*{\frac{f_{N,\ep}^t}{G_{N,\ep}}}df_{N,\ep}^t \right| \nn\\
&\leq \|\psi^t\|_{L^\infty}  \paren*{\int_{(\T^d)^N} \left|\nabla\log\paren*{\frac{f_{N,\ep}^t}{G_{N,\ep}}}\right|^2 df_{N,\ep}^t}^{1/2} \paren*{\int_{(\T^d)^N} \indic_{G_{N}^{-1}\geq M}df_{N,\ep}^t  }^{1/2} \nn\\
&\leq \frac{\|\psi^t\|_{L^\infty} }{(\log M)^{1/2}} \paren*{\int_{(\T^d)^N} \left|\nabla\log\paren*{\frac{f_{N,\ep}^t}{G_{N,\ep}}}\right|^2 df_{N,\ep}^t}^{1/2} \sup_{\ep, t>0}\paren*{\int_{(\T^d)^N} \log\paren*{G_{N}^{-1}}df_{N,\ep}^t  }^{1/2}.
\end{align}
Hence by Cauchy-Schwarz,
\begin{multline}
\int_0^\infty \left|\int_{(\T^d)^N} \psi^t \indic_{G_{N}^{-1}\geq M} \nabla\log\paren*{\frac{f_{N,\ep}^t}{G_{N,\ep}}}df_{N,\ep}^t \right| \leq (\log M)^{-1/2}\sup_{\ep, t>0}\paren*{\int_{(\T^d)^N} \log\paren*{G_{N}^{-1}}df_{N,\ep}^t  }^{1/2}\\
\times \left(\int_0^\infty \|\psi^t\|_{L^\infty}^2 dt\right)^{1/2}\sup_{\ep>0}\paren*{\int_0^\infty\int_{(\T^d)^N} \left|\nabla\log\paren*{\frac{f_{N,\ep}^t}{G_{N,\ep}}}\right|^2 df_{N,\ep}^t dt}^{1/2},
\end{multline}
which tends to zero as $M\rightarrow\infty$. This last step completes the proof of the lemma.

\bibliographystyle{amsalpha}
\bibliography{PointVortex}
\end{document}